\documentclass[10pt]{amsart}

\title{Unipotent Representations attached to the Principal Nilpotent Orbit}

\author{Lucas Mason-Brown}

\usepackage{amsmath,amssymb,amsthm,amsfonts,graphics ,tikz-cd, mdframed,enumitem,appendix, scalerel,stackengine,float,comment,todonotes}
\usepackage[hidelinks]{hyperref}
\numberwithin{equation}{subsection}

\theoremstyle{plain}
\newtheorem{theorem}[equation]{Theorem}
\newtheorem{lemma}[equation]{Lemma}
\newtheorem{remark}[equation]{Remark}
\newtheorem{corollary}[equation]{Corollary}

\newtheorem{example}[equation]{Example}
\newtheorem{proposition}[equation]{Proposition}
\newtheorem{definition}[equation]{Definition}

\newcommand{\gr}{\operatorname{gr}}
\newcommand{\Coh}{\operatorname{Coh}}

\newcommand{\AV}{\operatorname{AV}}
\newcommand{\ad}{\operatorname{ad}}
\newcommand{\Ad}{\operatorname{Ad}}

\newcommand{\cO}{\mathcal{O}}

\newcommand{\cN}{\mathcal{N}}

\newcommand{\unip}{\mathrm{Unip}}

\newcommand{\BB}{\mathrm{BB}}

\newcommand{\CC}{\mathbb{C}}

\newcommand{\RR}{\mathbb{R}}

\newcommand{\ZZ}{\mathbb{Z}}

\newcommand{\fb}{\mathfrak{b}}

\newcommand{\fh}{\mathfrak{h}}
\newcommand{\fq}{\mathfrak{q}}
\newcommand{\fg}{\mathfrak{g}}

\newcommand{\fl}{\mathfrak{l}}

\DeclareMathOperator{\Aut}{Aut}

\DeclareMathOperator{\Ind}{Ind}

\begin{document}
\maketitle

\begin{abstract}
In this paper, we construct and classify the special unipotent representations of a real reductive group attached to the principal nilpotent orbit. We give formulas for the $\mathbf{K}$-types, associated varieties, and Langlands parameters of all such representations. 
\end{abstract}

\section{Introduction}

Let $G$ be the real points of a connected reductive algebraic group. In \cite{ABV}, Adams, Barbasch, and Vogan, following ideas of Arthur (\cite{Arthur1983},\cite{Arthur1989}), defined a finite set of irreducible representations of $G$, called \emph{special unipotent representations}. These representations are conjectured to possess an array of interesting properties (see \cite[Chapter 1]{ABV}), including:

\begin{enumerate}
    \item They are conjectured to be unitary.
    \item They are conjectured to appear in spaces of automorphic forms.
    \item They are conjectured to \emph{generate} (through various kinds of induction) all irreducible unitary representations of $G$ of integral infinitesimal character
\end{enumerate}
These representations are naturally indexed by special nilpotent orbits for the complexification of $G$. For example, the trivial representation of $G$ is a unipotent representation attached to the nilpotent orbit $\{0\}$. If $G$ is quasi-split, then the spherical principal series representation $\Ind^G_B \CC$ is a unipotent representation attached to the principal nilpotent orbit (there are no other easy examples).

No general classification of special unipotent representations is known. However, properties (1)-(3) above suggest that obtaining one may be an essential ingredient in the classification of the irreducible unitary representations of $G$. In this paper, we will classify and construct all special unipotent representations attached to the principal nilpotent orbit. 

\subsection{Special unipotent representations}

Let $\mathbf{G}$ be the complexification of $G$, and let $\mathbf{G}^{\vee}$ be the dual group. If we fix a Cartan subalgbera $\mathfrak{h} \subset \mathfrak{g}$, there is a Cartan subalgebra $\mathfrak{h}^{\vee} \subset \mathfrak{g}^{\vee}$ which is naturally identified with $\mathfrak{h}^*$. The nilpotent co-adjoint orbits for $\mathbf{G}$ and $\mathbf{G}^{\vee}$ are related by \emph{Barbasch-Vogan duality}, first defined in \cite{BarbaschVogan1985}. This is a map
$$d: \{\text{nilpotent orbits } \cO^{\vee} \subset \mathfrak{g}^{\vee}\} \to \{\text{nilpotent orbits } \cO \subset \mathfrak{g}\}$$
A nilpotent orbit $\cO \subset \fg$ is \emph{special} if it lies in the image of $d$.

Every nilpotent $\mathbf{G}^{\vee}$-orbit $\cO^{\vee} \subset \mathfrak{g}^{\vee}$ gives rise to an infinitesimal character $\lambda_{\cO^{\vee}}$ for $U(\mathfrak{g})$ as follows. First, choose an element $e^{\vee} \in \cO^{\vee}$ and an $\mathfrak{sl}(2)$-triple $(e^{\vee},f^{\vee},h^{\vee})$. Conjugating by $\mathbf{G}^{\vee}$ if necessary, we can arrange so that $h^{\vee} \in \mathfrak{h}^{\vee}$. We define
$$\lambda_{\cO^{\vee}} := \frac{1}{2}h^{\vee} \in \mathfrak{h}^{\vee} \simeq \mathfrak{h}^*$$
This element is well-defined modulo the action of the Weyl group and therefore defines an infinitesimal character for $U(\mathfrak{g})$ (still denoted $\lambda_{\cO^{\vee}}$) by means of the Harish-Chandra isomorphism.

\begin{definition}\label{def:unipotentinfl}
Suppose $\cO \subset \mathfrak{g}$ is a special nilpotent $\mathbf{G}$-orbit. A \emph{unipotent infinitesimal character} attached to $\cO$ is one of the form $\lambda_{\cO^{\vee}}$ for $d(\cO^{\vee}) = \cO$. Denote the set of all such $\lambda_{\cO^{\vee}}$ by $\unip_Z(\cO)$.
\end{definition}

To any two-sided ideal $I \subset U(\fg)$, one can attach a $\mathbf{G}$-invariant subset $\AV(I) \subset \fg$ called the \emph{associated variety} of $I$. If $I$ is primitive (i.e. the annihilator of an irreducible $U(\fg)$-module), then $\AV(I)$ is the closure of a single nilpotent $\mathbf{G}$-orbit. Still assuming $I$ is primitive, the intersection of $I$ with the center of $U(\fg)$ is a maximal ideal (this is an easy consequence of Schur's lemma), and hence defines an infinitesimal character for $U(\fg)$. 

\begin{definition}\label{def:unipotentideal}
Suppose $\cO \subset \mathfrak{g}$ is a special nilpotent $\mathbf{G}$-orbit. A \emph{unipotent ideal} attached to $\cO$ is a primitive ideal $I \subset U(\fg)$ such that
\begin{itemize}
    \item[(i)] The infinitesimal character of $I$ belongs to $\unip_Z(\cO)$
    \item[(ii)] $\AV(I) = \overline{\cO}$.
\end{itemize}
Denote the set of all such ideals by $\mathrm{Unip}_I(\cO)$.
\end{definition}

Choose a maximal compact subgroup $K \subset G$. Let $\mathbf{K}$ be the complexification of $K$ and let $\mathfrak{k}$ be the Lie algebra of $\mathbf{K}$. 

\begin{definition}\label{def:unipotentrep}
Suppose $\cO \subset \mathfrak{g}$ is a special nilpotent $\mathbf{G}$-orbit. A \emph{unipotent representation} attached to $\cO$ is an irreducible $(\mathfrak{g},\mathbf{K})$-module $X$ such that $\mathrm{Ann}_{U(\fg)}(X) \in \mathrm{Unip}_I(\cO)$. Denote the set of (isomorphism classes of) such representations by $\mathrm{Unip}_R(\cO)$.
\end{definition}

If $\cO \subset \mathfrak{g}$ is the \emph{principal} nilpotent orbit, then $d^{-1}(\cO)$ consists of a single $\mathbf{G}^{\vee}$-orbit, $\{0\}$, and so $\unip_Z(\cO) =\{0\}$. Hence, a unipotent ideal attached to $\cO$ (a \emph{principal unipotent ideal} for short) is a primitive ideal $I \subset U(\fg)$ of infinitesimal character $0$ and associated variety $\mathcal{N}$. A \emph{principal unipotent representation} is an irreducible $(\mathfrak{g},\mathbf{K})$-module which is annihilated by such an ideal. We will see that in the principal case, the associated variety condition in Definition \ref{def:unipotentideal} is vacuous, but this will require some work.

\subsection{Main results}

Let $\cO \subset \cN$ be the principal nilpotent orbit. In Section \ref{sec:classification}, we will give two parameterizations of $\unip_R(\cO)$. Very roughly:

\begin{enumerate}
    \item We will construct the elements of $\unip_R(\cO)$ from (certain) characters of (certain) Borel subgroups using the Beilinson-Bernstein construction.
    \item We will construct the elements of $\unip_R(\cO)$ from (approximately) spherical principal series representations of $\theta$-stable parabolic subalgebras using cohomological induction.
\end{enumerate}
The precise statement is given in Corollary \ref{cor:principalunipotentreps}. Each paramaterization has its advantages. Parameterization (1) leads to a simple description of the Langlands parameters of principal unipotent representations (this is done in Section \ref{sec:Langlands}). Parameterization (2) leads to simple formulas for the associated varieties and $\mathbf{K}$-multiplicities of the representations in question (this is done in Section \ref{subsec:Ktypes}).

\tableofcontents

\section{Preliminaries}\label{sec:preliminaries}

Let $G$ be the real points of a connected reductive algebraic group defined over $\RR$. Choose a Cartan involution $\theta$ of $G$ and let $K \subset G$ be the fixed points of $\theta$. Denote the (real) Lie algebras of $K$ and $G$ by $\mathfrak{k}_0$ and $\mathfrak{g}_0$. Differentiating at the identity, $\theta$ gives rise to an involution of $\mathfrak{g}_0$ (which we will continue to denote by $\theta$), and hence a decomposition of $\mathfrak{g}_0$ into $+1$ and $-1$ eigenspaces
$$\mathfrak{g}_0 = \mathfrak{k}_0 \oplus \mathfrak{p}_0$$
Since $G$ and $K$ are algebraic, we can form their complexifications $\mathbf{G}$ and $\mathbf{K}$. $\mathbf{G}$ is a complex connected reductive algebraic group equipped with an antiholomorphic involution $\sigma$ with fixed points equal to $G$. The complexification of $\theta$ is a holomorphic involution of $\mathbf{G}$ (which we will continue to denote by $\theta$) with fixed points equal to $\mathbf{K}$. Note that $\sigma$ and $\theta$ commute. 

Denote the (complex) Lie algebras of $\mathbf{K}$ and $\mathbf{G}$ by $\mathfrak{k}$ and $\mathfrak{g}$. Again, $\theta$ gives rise to an involution of $\mathfrak{g}$ (which we will continue to denote by $\theta$), and hence a decomposition into $+1$ and $-1$ eigenspaces
$$\mathfrak{g} = \mathfrak{k} \oplus \mathfrak{p}$$
%
Certain aspects of this notation will be generalized without comment: we will use capital letters $A,B,...$ to denote Lie groups, boldface capital letters $\mathbf{A}, \mathbf{B}, ...$ for their complexifications, lower-case gothic letters with subscripts $\mathfrak{a}_0,\mathfrak{b}_0,...$ for the real Lie algebras, and unadorned gothic letters $\mathfrak{a},\mathfrak{b},...$ for the complexified Lie algebras.

\subsection{Cartan subgroups}\label{sec:cartansubgroups}

Recall that a Cartan subalgebra of $\mathfrak{g}_0$ is by definition a subalgebra $\mathfrak{h}_0 \subset \mathfrak{g}_0$ whose complexification $\mathfrak{h}$ is a Cartan subalgebra of $\mathfrak{g}$. A Cartan subgroup of $G$ is by definition the centralizer in $G$ of a Cartan subalgebra of $\mathfrak{g}_0$. Any such subgroup is conjugate by $G$ to one preserved by $\theta$. If $H \subset G$ is a $\theta$-stable Cartan subgroup of $G$, then we can define 
$$T := H \cap K \quad \mathfrak{a}_0 = \mathfrak{h}_0 \cap \mathfrak{p}_0 \quad A_0 := \exp(\mathfrak{a}_0)$$
Then the Cartan decomposition of $H$ is a direct product
$$H = TA$$
Under our assumptions on $G$, $H$ is abelian (though possibly disconnected).

\subsection{Roots}\label{sec:roots}

Let $H$ be a $\theta$-stable Cartan subgroup of $G$. We will write $\Delta(\mathfrak{g},\mathfrak{h}) \subset \mathfrak{h}^*$ (resp. $\Delta(\mathfrak{g},H)$) for the roots of $\mathfrak{h}$ (resp. $H$) on $\mathfrak{g}$. There is a natural bijection $\Delta(H,\mathfrak{g}) \simeq \Delta(\mathfrak{h},\mathfrak{g})$ (differentiation), which we will often use without comment. Since $H$ is $\theta$-stable, there is a natural action of $\theta$ on $\Delta(\mathfrak{g},\mathfrak{h})$, defined by
$$\theta(\alpha) := \alpha \circ \theta \quad \alpha \in \Delta(\mathfrak{g},\mathfrak{h})$$
In general, roots come in three different varieties.

\begin{proposition}\label{prop:realimagcomplex}
Every root $\alpha \in \Delta(\mathfrak{g},\mathfrak{h})$ takes real values on $\mathfrak{a}_0$ and imaginary values on $\mathfrak{t}_0$. It is \emph{real} if one (any) of the following equivalent conditions is satisfied
\begin{enumerate}
    \item $\alpha(\mathfrak{h}_0) \subset \mathbb{R}$
    \item $\alpha(H) \subset \mathbb{R}^{\times}$
    \item $\alpha|_{\mathfrak{t}_0} \equiv 0$
    \item $\theta(\alpha) = -\alpha$
    \item $\sigma(\alpha) = \alpha$
\end{enumerate}
It is \emph{imaginary} if one (any) of the following equivalent conditions is satisfied
\begin{enumerate}
    \item $\alpha(\mathfrak{h}_0) \subset i\mathbb{R}$
    \item $\alpha(H) \subset S^1$
    \item $\alpha|_{\mathfrak{a}_0} \equiv 0$
    \item $\theta(\alpha) = \alpha$
    \item $\sigma(\alpha) = - \alpha$
\end{enumerate}
It is \emph{complex} if it is neither real nor imaginary. 
\end{proposition}

Define
\begin{align*}
    \Delta_{\mathbb{R}}(\mathfrak{g},\mathfrak{h}) &:= \{\alpha \in \Delta(\mathfrak{g},\mathfrak{h}): \alpha \text{ real}\}\\
    \Delta_{i\mathbb{R}}(\mathfrak{g},\mathfrak{h}) &:= \{\alpha \in \Delta(\mathfrak{g},\mathfrak{h}): \alpha \text{ imaginary}\}
\end{align*}
It is clear from Proposition \ref{prop:realimagcomplex} that  $\Delta_{\mathbb{R}}(\mathfrak{g},\mathfrak{h})$ and $\Delta_{i\mathbb{R}}(\mathfrak{g},\mathfrak{h})$ form root subsystems of $\Delta(\mathfrak{g},\mathfrak{h})$. If $\alpha \in \Delta_{i\mathbb{R}}(\mathfrak{g},\mathfrak{h})$, then $\theta(\alpha) =\alpha$ and hence $\theta(\mathfrak{g}_{\alpha}) = \mathfrak{g}_{\alpha}$, where $\mathfrak{g}_{\alpha}$ is the root space for $\alpha$. Since $\mathfrak{g}_{\alpha}$ is one-dimensional, this means either $\mathfrak{g}_{\alpha} \subset \mathfrak{k}$ or $\mathfrak{g}_{\alpha} \subset \mathfrak{p}$. We say that $\alpha$ is \emph{compact} or \emph{noncompact}, accordingly. 

Define
\begin{align*}
    \Delta_c(\mathfrak{g},\mathfrak{h}) &:= \{\alpha \in \Delta_{i\mathbb{R}}(\mathfrak{g},\mathfrak{h}): \alpha \text{ compact}\}\\
    \Delta_n(\mathfrak{g},\mathfrak{h}) &:= \{\alpha \in \Delta_{i\mathbb{R}}(\mathfrak{g},\mathfrak{h}): \alpha \text{ noncompact}\}
\end{align*}
We get a $\ZZ_2$-grading $\epsilon$ on $\Delta_{i\mathbb{R}}(\mathfrak{g},\mathfrak{h})$, defined by
\begin{align*}
    \epsilon(\beta) &= 0 \quad \text{if } \beta \in \Delta_c(\mathfrak{g},\mathfrak{h})\\
    \epsilon(\beta) &= 1 \quad \text{if } \beta \in \Delta_n(\mathfrak{g},\mathfrak{h})
\end{align*}

If $\mathfrak{o}$ is a finite-dimensional $\mathfrak{h}$-module, we will write $\Delta(\mathfrak{o},\mathfrak{h})$ for the multi-set of $\mathfrak{h}$-weights on $\mathfrak{o}$ and define
$$\rho(\mathfrak{o}) := \frac{1}{2}\sum \Delta(\mathfrak{o},\mathfrak{h}) \in \mathfrak{h}^*$$
Usually, $\mathfrak{o}$ will be the nilradical $\mathfrak{u}$ of a parabolic subalgebra of $\mathfrak{g}$ (or its intersection with $\mathfrak{k}$ or with $\mathfrak{p}$). In this case, the functional $2\rho(\mathfrak{u})$ corresponds to a complex character of $H$. If $\mathfrak{q}$ is $\sigma$-stable, then this complex character is real and we can take its absolute value $|2\rho(\mathfrak{u})|$. In this case, we define
$$|\rho(\mathfrak{u})| := \sqrt{|2\rho(\mathfrak{u})|}$$

\subsection{Cayley transforms: preliminaries}

Write $E,F,D$ for the usual (split) basis of $\mathfrak{sl}_2(\mathbb{C})$:
$$E = \begin{pmatrix} 0 & 1\\0 & 0\end{pmatrix} \quad F = \begin{pmatrix} 0 & 0\\1 & 0\end{pmatrix} \quad D = \begin{pmatrix} 1 & 0\\0 & -1\end{pmatrix}$$
and $E_c,F_c,D_c$ for the (compact) basis:
    
$$E_c = \frac{1}{2}\begin{pmatrix}1 & -i \\ -i & -1\end{pmatrix} \quad F_c = \frac{1}{2}\begin{pmatrix}1 & i \\ i & -1\end{pmatrix} \quad D_c = \begin{pmatrix}0 & i\\-i & 0\end{pmatrix}$$

\begin{proposition}\label{prop:rootsl2s}
Let $\alpha \in \Delta(\mathfrak{g},\mathfrak{h})$ be real or noncompact imaginary. Write $\mathfrak{s}_{\alpha} \subset \mathfrak{g}$ for the three-dimensional subalgebra generated by the root spaces $\mathfrak{g}_{\alpha}$ and $\mathfrak{g}_{-\alpha}$. Let $\theta_s$ be the involution of $\mathfrak{sl}_2(\mathbb{C})$ defined by $\theta_s(X) = -X^t$ and let $\sigma_s$ be complex conjugation. There is an isomorphism
$$\phi_{\alpha}: \mathfrak{sl}_2(\mathbb{C}) \to \mathfrak{s}_{\alpha}$$
intertwining $\theta$ with $\theta_s$ and $\sigma$ with $\sigma_s$. If $\alpha$ is real, we can choose $\phi_{\alpha}$ so that
$$\phi_{\alpha}(E) \in \mathfrak{g}_{\alpha} \quad \phi_{\alpha}(F) \in \mathfrak{g}_{-\alpha} \quad \phi_{\alpha}(D) = \alpha^{\vee}$$
This isomorphism is unique up to pre-conjugation by
$$\begin{pmatrix}1 & 0\\0 & -1\end{pmatrix}$$
If $\alpha$ is noncompact imaginary, we can choose $\phi_{\alpha}$ so that
$$\phi_{\alpha}(E_c) \in \mathfrak{g}_{\alpha} \quad \phi_{\alpha}(F_c) \in \mathfrak{g}_{-\alpha} \quad \phi_{\alpha}(D_c) = \alpha^{\vee}$$
This isomorphism is unique up to pre-conjugation by $O_2(\mathbb{R})$.
\end{proposition}

\begin{proof}
The existence statements are immediate from Theorem \ref{thm:kostantsekiguchi1}. If $\alpha$ is real, two isomorphisms of the type described in the proposition differ by an automorphism $\zeta$ of $\mathfrak{sl}_2(\mathbb{C})$ satisfying
\begin{enumerate}
    \item $\zeta \circ \theta_s = \theta_s \circ \zeta$
    \item $\zeta \circ \sigma_s = \sigma_s \circ \zeta$
    \item $\zeta(E) \in \mathbb{C}E$
    \item $\zeta(D) = D$
\end{enumerate}
Every automorphism of $\mathfrak{sl}_2(\mathbb{C})$ corresponds to conjugation by a matrix $g \in GL_2(\mathbb{C})$. By an easy calculation in $SL_2(\mathbb{C})$, $g \in \{\pm \mathrm{Id}\} \cup \{\pm \begin{pmatrix} 1 &0\\0 & -1\end{pmatrix}\}$. The noncompact imaginary case is handled similarly. 
\end{proof}

\subsection{Cayley transforms through real roots}\label{sec:Cayleyreal}

In this subsection, we will describe a well-known procedure for producing, from a $\theta$-stable Cartan subgroup and a real root, a new Cartan subgroup which is slightly more compact. For details and proofs, we refer the reader to \cite{Knapp1996}.

Let $H$ be a $\theta$-stable Cartan subgroup of $G$ and let $\alpha \in \Delta_{\mathbb{R}}(\mathfrak{g},H)$ be a real root. Fix an isomorphism
$$\phi_{\alpha}: \mathfrak{sl}_2(\mathbb{C}) \to \mathfrak{s}_{\alpha}$$
as in Proposition \ref{prop:rootsl2s}. Define a new $\theta$-stable Cartan subalgebra $\mathfrak{h}_0^{\alpha}$ of $\mathfrak{g}_0$
$$\mathfrak{t}_0^{\alpha} := \mathfrak{t}_0 \oplus i\mathbb{R}\phi_{\alpha}(D_c) \quad \mathfrak{a}_0^{\alpha} := \ker{\alpha} \cap \mathfrak{a}_0 \quad \mathfrak{h}^{\alpha}_0 := \mathfrak{t}_0^{\alpha} \oplus \mathfrak{a}_0^{\alpha}$$
and write $H^{\alpha}$ for the corresponding Cartan subgroup of $G$
$$T^{\alpha} := Z_K(\mathfrak{t}_0^{\alpha}) \quad A:= \mathrm{exp}(\mathfrak{a}_0^{\alpha}) \quad H^{\alpha} := T^{\alpha}A^{\alpha}$$
Although the element $\phi_{\alpha}(D_c)$ depends on $\phi_{\alpha}$, pre-conjugation by
$$\begin{pmatrix}1 & 0\\0 & -1\end{pmatrix}$$
takes $\phi_{\alpha}(D_c)$ to $-\phi_{\alpha}(D_c)$. In particular, the real line $\mathbb{R}\phi_{\alpha}(D_c)$ is independent of $\phi_{\alpha}$ and hence $H^{\alpha}$ is well-defined.

The subalgebras $\mathfrak{h}_0$ and $\mathfrak{h}_0^{\alpha}$ are non-conjugate under $G$ (since their compact dimensions differ). But their complexifications $\mathfrak{h}$ and $\mathfrak{h}^{\alpha}$ (like any pair of complex Cartan subalgebras) are conjugate under $\mathbf{G}$. The Cayley transforms $c_{\alpha}^{\pm}$ are explicitly-defined inner automorphisms of $\mathfrak{g}$ mapping $\mathfrak{h}$ onto $\mathfrak{h}^{\alpha}$.

\begin{definition}\label{def:cayleytransformreal}
The Cayley transforms of $\mathfrak{g}$ through $\alpha$ are the inner automorphisms
$$c_{\alpha}^{\pm} := \exp (\mathrm{ad} (\frac{\pm i\pi}{4} \phi_{\alpha}(E+F)))$$
\end{definition}

This definition depends on $\phi_{\alpha}$, but not in a serious way. Pre-conjugation by
$$\begin{pmatrix}1 & 0 \\ 0 & -1\end{pmatrix}$$
takes $E+F$ to $-E-F$. Hence, the pair $c_{\alpha}^{\pm}$ is independent of $\phi_{\alpha}$ (although the invidual automorphisms are not).

\begin{proposition}\label{prop:calphacartan}
Both $c_{\alpha}^{\pm}$ act by the identity on $\ker{\alpha} \subset \mathfrak{h}$ and on $\alpha^{\vee} \in \mathfrak{h}$ by
$$c_{\alpha}^{\pm}\alpha^{\vee} = -\pm \phi_{\alpha}(D_c) $$
In particular,
$$c_{\alpha}^{\pm}\mathfrak{h} = \mathfrak{h}^{\alpha}$$
\end{proposition}

In view of Proposition \ref{prop:calphacartan}, the automorphisms $c_{\alpha}^{\pm}$ induce bijections (which we will continue to denote by $c_{\alpha}^{\pm}$):
$$c_{\alpha}^{\pm}: \Delta(\mathfrak{g},\mathfrak{h}) \to \Delta(\mathfrak{g},\mathfrak{h}^{\alpha}) \quad \beta \mapsto \beta \circ (c_{\alpha}^{\pm})^{-1}$$
One can understand completely how these bijections behave with respect to the properties of being real, imaginary, complex, compact, and noncompact. For our purposes, the following proposition is sufficient.
\begin{proposition}\label{prop:calphafacts}
Write $\Delta(\mathfrak{g},\mathfrak{h})^{\alpha}$ for the set of roots orthogonal to $\alpha$ (i.e. roots $\beta$ with $\langle \beta, \alpha^{\vee}\rangle =0$). Then
\begin{enumerate}
    \item The bijections
    $$c_{\alpha}^{\pm}:\Delta(\mathfrak{g},\mathfrak{h}) \cong \Delta(\mathfrak{g},\mathfrak{h}^{\alpha})$$
    are related by
    $$c_{\alpha}^- = c_{\alpha}^+ \circ s_{\alpha} \quad c_{\alpha}^+ = c_{\alpha}^- \circ s_{\alpha}$$
    \item The roots $c_{\alpha}^{\pm}\alpha$ are noncompact imaginary (and, by $(1)$, negatives of one another)
    \item $c_{\alpha}^{\pm}$ restrict to a (single, well-defined) bijection
    $$c_{\alpha}:\Delta_{\mathbb{R}}(\mathfrak{g},\mathfrak{h})^{\alpha} \cong \Delta_{\mathbb{R}}(\mathfrak{g},\mathfrak{h}^{\alpha})$$
\end{enumerate}
\end{proposition}

Since $\phi_{\alpha}$ commutes with complex conjugation, it restricts to an isomorphism
$$\phi_{\alpha}: \mathfrak{sl}_2(\mathbb{R}) \to \mathfrak{s}_{\alpha} \cap \mathfrak{g}_0$$
Because $G$ is algebraic, this integrates to a group homomorphism
$$\Phi_{\alpha}: SL_2(\mathbb{R}) \to G$$
Define the element
$$m_{\alpha}:=\Phi_{\alpha}\begin{pmatrix}-1 & 0\\0 & -1 \end{pmatrix} \in T$$
By Proposition \ref{prop:rootsl2s}, this element is independent of $\phi_{\alpha}$. We will need the following structural fact:

\begin{lemma}[\cite{Vogan1981}, Lemma 8.3.13]\label{lem:structureofTalpha}
Define
$$T^{\alpha}_1 := \ker{\alpha} \cap T$$
Then 
$$\Phi_{\alpha}(SO_2(\mathbb{R})) \cap T^{\alpha}_1 = \{1,m_{\alpha}\}$$
and there is a decomposition
$$T^{\alpha} = \Phi_{\alpha}(SO_2(\mathbb{R})) T^{\alpha}_1$$
\end{lemma}

Now, suppose $\chi$ is a (complex) character of $H$. Since $m_{\alpha}^2 = 1$ and $\chi$ is a group homomorphism, we have $\chi(m_{\alpha}) = \pm 1$. 

\begin{definition}
If $\alpha \in \Delta_{\mathbb{R}}(\mathfrak{g},H)$ is a real root and $\chi$ is a character of $H$, we say that $\alpha$ is even (resp. odd) for $\chi$ if $\chi(m_{\alpha}) = 1$ (resp. $-1$). 
\end{definition}

If $\alpha$ is odd for $\chi$, we will define two characters $c_{\alpha}^{\pm}\chi$ of $H^{\alpha}$ called the Cayley transforms of $\chi$. First, define two characters $\tau_{\pm 1}$ of $SO_2(\mathbb{R})$

\begin{equation}\label{eqn:twocharactersofSO2R}
    d\tau_{\pm 1}\begin{pmatrix}0 & 1\\-1 & 0 \end{pmatrix} = \pm i
\end{equation}

Then 

\begin{definition}\label{def:cayleytransformsofcharacter}
Since $\phi_{\alpha}$ is injective, $\ker{\Phi_{\alpha}} \subset Z(SL_2(\mathbb{R})) = \{\pm 1\}$. Under the assumption that $\alpha$ is odd, we must have $\Phi_{\alpha}(-1) = m_{\alpha} \neq 1$. So in this case, $\Phi_{\alpha}$ is an embedding. Define characters $c_{\alpha}^{\pm}\chi$ of the product group $\Phi_{\alpha}(SO_2(\mathbb{R})) \times T^{\alpha}_1$ by the formulas
$$c_{\alpha}^{\pm}\chi(\Phi_{\alpha}(g), t) = \tau_{\pm 1}(g)\chi(t)$$
Multiplication defines a group homomorphism
$$\Phi_{\alpha}(SO_2(\mathbb{R})) \times T^{\alpha}_1 \to T^{\alpha}$$
which is surjective with kernel $\{(1,1),(m_{\alpha},m_{\alpha})\}$ by Lemma \ref{lem:structureofTalpha}. Since $\alpha$ is odd,
$$c_{\alpha}^{\pm}\chi (m_{\alpha},m_{\alpha}) = (-1)^2 =1$$
and therefore both characters $c_{\alpha}^{\pm}\chi$ descend to well-defined characters of $T^{\alpha}$. Extend these characters to $H^{\alpha}$ by defining
$$c_{\alpha}^{\pm}\chi(ta) = c_{\alpha}^{\pm}(t)\chi(a) \quad t \in T^{\alpha},a \in A^{\alpha} \subset A$$
\end{definition}

\subsection{Cayley transforms through noncompact imaginary roots}\label{sec:Cayleynoncompact}

In this subsection, we will describe a well-known procedure for producing, from a $\theta$-stable Cartan subgroup and a noncompact imaginary root, a new Cartan subgroup which is slightly less compact. The construction is analogous to that of Section \ref{sec:Cayleyreal}. Again, a good reference is \cite{Knapp1996}. 

Let $H \subset G$ be a $\theta$-stable Cartan subgroup and let $\alpha \in \Delta_n(\mathfrak{g},H)$ be a noncompact imaginary root. Fix an isomorphism
$$\phi_{\alpha}: \mathfrak{sl}_2(\mathbb{C}) \to \mathfrak{s}_{\alpha}$$
as in Proposition \ref{prop:rootsl2s}. Define a new $\theta$-stable Cartan subalgebra $\mathfrak{h}^{\alpha}_0$ of $\mathfrak{g}_0$
$$\mathfrak{t}^{\alpha}_0 := \ker{\alpha} \cap \mathfrak{t}_0 \quad \mathfrak{a}^{\alpha}_0 := \mathfrak{a}_0 \oplus \mathbb{R}\phi_{\alpha}D \quad \mathfrak{h}^{\alpha}_0 := \mathfrak{t}^{\alpha}_0 \oplus \mathfrak{a}^{\alpha}_0$$
and write $H^{\alpha}$ for the corresponding Cartan subgroup of $G$
$$T^{\alpha} := Z_K(\mathfrak{t}_0^{\alpha}) \quad A := \mathrm{exp}(\mathfrak{a}^{\alpha}_0) \quad H^{\alpha} := T^{\alpha}A^{\alpha}$$
In contrast to the real case, $H^{\alpha}$ \emph{does} depend on $\phi_{\alpha}$ (as we vary $\phi_{\alpha}$, we get a one-parameter family of $\theta$-stable Cartan subgroups). 

We define

\begin{definition}\label{def:cayleytransformnoncompact}
The Cayley transforms of $\mathfrak{g}$ through $\alpha$ are the inner automorphisms
$$d_{\alpha}^{\pm} :=  \mathrm{exp}(\ad(
\frac{\pm \pi}{4}(\phi_{\alpha}(F_c - E_c)))) 
\in \mathrm{Aut}(\mathfrak{g})$$
\end{definition}

\begin{proposition}\label{prop:dalphacartan}
Both $d_{\alpha}^{\pm}$ act by the identity on $\ker{\alpha} \subset \mathfrak{h}$ and on $\alpha^{\vee} \in \mathfrak{h}$ by
$$d_{\alpha}^{\pm}\alpha^{\vee} = \pm \phi_{\alpha}(D) $$
In particular,
$$d_{\alpha}^{\pm}\mathfrak{h} = \mathfrak{h}^{\alpha}$$
\end{proposition}

In view of Proposition \ref{prop:dalphacartan}, the automorphisms $d_{\alpha}^{\pm}$ induce bijections:
$$d_{\alpha}^{\pm}: \Delta(\mathfrak{g},\mathfrak{h}) \to \Delta(\mathfrak{g},\mathfrak{h}^{\alpha}) \quad \beta \mapsto \beta \circ (d_{\alpha}^{\pm})^{-1}$$
and
\begin{proposition}\label{prop:dalphafacts}
We have
\begin{enumerate}
    \item The bijections
    $$d_{\alpha}^{\pm}:\Delta(\mathfrak{g},\mathfrak{h}) \cong \Delta(\mathfrak{g},\mathfrak{h}^{\alpha})$$
    are related by
    $$d_{\alpha}^- = d_{\alpha}^+ \circ s_{\alpha} \quad d_{\alpha}^+ = d_{\alpha}^- \circ s_{\alpha}$$
    \item $d_{\alpha}^{\pm}\alpha$ are real roots (and, by $(1)$, negatives of one another)
    \item $d_{\alpha}^{\pm}$ restrict to a (single, well-defined) bijection
    $$d_{\alpha}:\Delta_{i\mathbb{R}}(\mathfrak{g},\mathfrak{h})^{\alpha} \cong \Delta_{i\mathbb{R}}(\mathfrak{g},\mathfrak{h}^{\alpha})$$
    \item and a bijection
    $$d_{\alpha}:  \{\beta \in \Delta_n(\mathfrak{g},\mathfrak{h})^{\alpha}: \alpha \pm \beta \notin \Delta(\mathfrak{g},\mathfrak{h})\} \ \cup \ \{\beta \in \Delta_c(\mathfrak{g},\mathfrak{h})^{\alpha}: \alpha \pm \beta \in \Delta(\mathfrak{g},\mathfrak{h})\} \cong \Delta_n(\mathfrak{g},\mathfrak{h}^{\alpha})$$
\end{enumerate}
\end{proposition}

A reformulation of $(3)$ and $(4)$ is helpful. Write $\epsilon$ for the $\mathbb{Z}/2\mathbb{Z}$-grading on $\Delta_{i\mathbb{R}}(\mathfrak{g},\mathfrak{h})$ defined in Section \ref{sec:roots}, and use $d_{\alpha}$ to identify the subsystem $\Delta_{i\mathbb{R}}(\mathfrak{g},\mathfrak{h})^{\alpha}$ with $\Delta_{i\mathbb{R}}(\mathfrak{g},\mathfrak{h}^{\alpha})$. The $\mathbb{Z}/2\mathbb{Z}$-grading on $\Delta_{i\mathbb{R}}(\mathfrak{g},\mathfrak{h}^{\alpha})$ induces a $\mathbb{Z}/2\mathbb{Z}$-grading $d_{\alpha}\epsilon$ on $\Delta_{i\mathbb{R}}(\mathfrak{g},\mathfrak{h})^{\alpha}$, which is given by
\begin{align*}
    (d_{\alpha}\epsilon)(\beta) &= 0 \quad \text{if
} \epsilon(\beta)=0 \text{ and } \alpha \pm \beta \notin \Delta_{i\mathbb{R}}\\
    (d_{\alpha}\epsilon)(\beta) &= 1 \quad \text{if
} \epsilon(\beta)=0 \text{ and } \alpha \pm \beta \in \Delta_{i\mathbb{R}}\\
    (d_{\alpha}\epsilon)(\beta) &= 0 \quad \text{if
} \epsilon(\beta)=1 \text{ and } \alpha \pm \beta \in \Delta_{i\mathbb{R}}\\
    (d_{\alpha}\epsilon)(\beta) &= 1 \quad \text{if
} \epsilon(\beta)=1 \text{ and } \alpha \pm \beta \notin \Delta_{i\mathbb{R}}
\end{align*}

We need the following technical lemma. Part (2) is precisely \cite[Lem 5.12]{VoganIC4}. Part (1) is proved analogously.

\begin{lemma}\label{lem:simplynoncompact}
Let $(\Delta_{i\mathbb{R}}, \Delta_{i\mathbb{R}}^+, \epsilon)$ be a $\mathbb{Z}/2\mathbb{Z}$-graded root system with a choice of positive roots. The triple $(\Delta_{i\mathbb{R}},\Delta_{i\mathbb{R}}^+,\epsilon)$ is \emph{large} if every $\beta \in \Pi^+_{i\mathbb{R}}$ has $\epsilon(\beta)=1$. Let $\alpha \in \Pi^+_{i\mathbb{R}}$, and consider the grading $d_{\alpha}\epsilon$ on $\Delta_{i\mathbb{R}}^{\alpha}$. 

\begin{enumerate}
    \item If  $(\Delta_{i\mathbb{R}}, \Delta_{i\mathbb{R}}^+, \epsilon)$ is large, so is $(\Delta_{i\mathbb{R}}^{\alpha}, (\Delta_{i\mathbb{R}}^+)^{\alpha}, d_{\alpha}\epsilon)$
    \item If $(\Delta_{i\mathbb{R}}^{\alpha}, (\Delta_{i\mathbb{R}}^+)^{\alpha}, d_{\alpha}\epsilon)$ is large, either $(\Delta_{i\mathbb{R}},\Delta_{i\mathbb{R}}^+, \epsilon)$ or $(\Delta_{i\mathbb{R}}, s_{\alpha}\Delta_{i\mathbb{R}}^+, \epsilon)$ is too.
\end{enumerate}
\end{lemma}

The automorphisms $c_{\alpha}^{\pm}$ and $d_{\alpha}^{\pm}$ are essentially inverse to one another. More precisely

\begin{proposition}\label{prop:cayleysareinverse}
 If $\alpha \in \Delta(\mathfrak{g},\mathfrak{h})$ is real, then $c_{\alpha}^+\alpha \in \Delta(\mathfrak{g},\mathfrak{h}^{\alpha})$ is noncompact imaginary. There is a choice of $\phi_{c_{\alpha}^+\alpha}$ as in Proposition \ref{prop:rootsl2s} such that
 $$d^+_{c_{\alpha}^+\alpha} \circ c_{\alpha}^+ = c_{\alpha}^+ \circ d^+_{c_{\alpha}^+\alpha} = 1 $$
If $\beta \in \Delta(\mathfrak{g},\mathfrak{h})$ is noncompact imaginary, then $d_{\beta}^+\beta \in \Delta(\mathfrak{g},\mathfrak{h}^{\beta})$ is real. There is a choice of $\phi_{d_{\beta}^+\beta}$ as in Proposition \ref{prop:rootsl2s} such that
$$c^+_{d_{\beta}^+\beta} \circ d_{\beta}^+ = d_{\beta}^+ \circ c^+_{d_{\beta}^+\beta} = 1$$
\end{proposition}

\end{definition}

%
%


%
%


To do homological algebra in $M(\mathfrak{g},\mathbf{K})$, we will need

\begin{proposition}[\cite{KnappVogan1995}, Corollary 2.26]\label{prop:enoughinjectives}
The category $M(\mathfrak{g},\mathbf{K})$ has enough injectives.
\end{proposition}

If $G$ is a real reductive group, then $(\mathfrak{g},\mathbf{K})$ is a pair. In this case, it makes sense to define

\begin{definition}
Suppose $G$ is a real reductive group. A $(\mathfrak{g},\mathbf{K})$-module $X$ is admissible if each irreducible representation of $\mathbf{K}$ appears in $X$ with finite multiplicity. 
\end{definition}

We now have three distinct finiteness conditions on $M(\mathfrak{g},\mathbf{K})$: finite-generation, finite-length, and admissibility. They are related by the following proposition

\begin{proposition}\label{prop:threefinitenessconditions}
Let $G$ be a real reductive group and let $X$ be a $(\mathfrak{g},\mathbf{K})$-module. Write %
$$I := \mathrm{Ann}(X) := \{v \in U(\mathfrak{g}): vX =0\}$$

The following are equivalent
\begin{enumerate}
\item\label{fgad} $X$ is finitely-generated and admissible
\item\label{fgfs} $X$ is finitely generated and
$$I \cap Z(\mathfrak{g}) \subset Z(\mathfrak{g})$$
has finite codimension
\item\label{adfs} $X$ is admissible and
$$I \cap Z(\mathfrak{g}) \subset Z(\mathfrak{g})$$
has finite codimension
\item\label{fl} $X$ has finite-length
\end{enumerate}
\end{proposition}
All of these equivalences are standard. The hardest part (that $(4)$ implies $(3)$) is Theorem 10.1 in \cite{KnappVogan1995}.
\end{comment}

\subsection{Three nilpotent cones}\label{sec:threenilpotentcones}

Let $\mathcal{N} \subset \mathfrak{g}$ be the complex nilpotent cone. It is classically known that $\mathbf{G}$ acts on $\mathcal{N}$ with finitely-many orbits. There are two subcones of $\mathcal{N}$ which play a special role in the representation theory of $G$:
$$\mathcal{N}_{\theta} = \mathcal{N} \cap \mathfrak{p} \qquad \mathcal{N}_0 = \mathcal{N} \cap \mathfrak{g}_0$$
These subcones are invariant under the actions of $\mathbf{K}$ and $G$, respectively, and in both cases, there are finitely-many orbits (see \cite{KostantRallis1971}).

There is an elegant relationship between $\mathbf{K}$-orbits on $\mathcal{N}_{\theta}$ and $G$-orbits on $\mathcal{N}_0$, first observed by Sekiguchi (\cite{Sekiguchi1987}). The following formulation is due to Vogan:

\begin{theorem}[\cite{Vogan1991}, Theorem 6.4] \label{thm:kostantsekiguchi1}
Write $\theta_s$ for the involution of $\mathfrak{sl}_2(\mathbb{C})$ defined by $\theta_s(X) := -X^t$ and $\sigma_s$ for complex conjugation. Then the following sets are in natural bijection

\begin{enumerate}
    \item $\mathcal{N}_0/G$
    \item $G$-conjugacy classes of homomorphisms
    $$\phi_{\sigma}: \mathfrak{sl}_2(\mathbb{C}) \to \mathfrak{g}$$
    intertwining $\sigma$ with $\sigma_s$
    \item $K$-conjugacy classes of homomorphisms
    $$\phi_{\sigma,\theta}: \mathfrak{sl}_2(\mathbb{C}) \to \mathfrak{g}$$
    intertwining $\sigma$ with $\sigma_s$ and $\theta$ with $\theta_s$
    \item $\mathbf{K}$-conjugacy classes of homomorphisms
    $$\phi_{\theta}: \mathfrak{sl}_2(\mathbb{C}) \to \mathfrak{g}$$
    intertwining $\theta$ with $\theta_s$
    \item $\mathcal{N}_{\theta}/\mathbf{K}$
\end{enumerate}

The maps from $(3)$ to $(2)$ and $(3)$ to $(4)$ are the inclusions. The map from $(2)$ to $(1)$ is defined by
$$\phi_{\sigma} \mapsto \phi_{\sigma}(E)$$
The map from $(4)$ to $(5)$ is defined by
$$\phi_{\theta} \mapsto \phi_{\theta}(E_c) $$
\end{theorem}

Choose a non-degenerate symmetric bilinear form $\langle \cdot, \cdot \rangle$ on $\mathfrak{g}_0$ making $\mathfrak{k}_0$ and $\mathfrak{p}_0$ orthogonal. Using this form, we get a $\mathbf{G}$-invariant identification $\varphi: \mathfrak{g} \simeq \mathfrak{g}^*$. Define
$$\mathcal{N}^* := \varphi(\mathcal{N}) \qquad \mathcal{N}^*_{\theta} := \varphi(\mathcal{N}_{\theta}) = \mathcal{N}^* \cap (\mathfrak{g}/\mathfrak{k})^* \qquad \mathcal{N}_0^* := \varphi(\mathcal{N}_0) = \mathcal{N}^* \cap \mathfrak{g}_0^*$$
Note that $\langle \cdot, \cdot \rangle$ is unique up to scalar multiplication. Hence, the subsets $\mathcal{N}^*, \mathcal{N}_{\theta}^*, \mathcal{N}_0^* \subset \mathfrak{g}$ are well-defined. By construction, these subsets are invariant under the (co-adjoint) actions of $\mathbf{G}$, $\mathbf{K}$, and $G$ (respectively), and in each case there are finitely-many orbits. Each $\mathbf{G}$-orbit on $\mathcal{N}$ carries a distinguished symplectic form. This is one reason for preferring these `dual' cones to their counterparts in $\mathfrak{g}$. Of course, the bijection $\mathcal{N}_{\theta}/\mathbf{K} \cong \mathcal{N}_0/G$ of Theorem \ref{thm:kostantsekiguchi1} induces a bijection $\mathcal{N}_{\theta}^*/\mathbf{K} \cong \mathcal{N}_0^*/G$. Some features of this correspondence are slightly easier to see on the `dual' side:

\begin{theorem}[Kostant-Sekiguchi-Vergne-Barbasch-Sepanski, \cite{KostantRallis1971}, \cite{Sekiguchi1987}, \cite{Vergne1995},\cite{BarbaschSepanski1998}]\label{thm:kostantsekiguchi2}
The bijection
$$\eta: \mathcal{N}^*_{\theta}/\mathbf{K} \to \mathcal{N}^*_0/G
$$
defined by the requirements of Theorem \ref{thm:kostantsekiguchi1} has the following properties:
\begin{enumerate}
    \item $\eta$ respects the closure orderings on $\mathcal{N}_{\theta}^*/\mathbf{K}$ and $\mathcal{N}^*_0/G$.
    \item For every $\mathcal{O} \in\mathcal{N}^*_{\theta}/\mathbf{K}$, there is a $K$-invariant diffeomorphism
    $$\mathcal{O} \cong \eta(\mathcal{O})$$
    \item For every $\mathcal{O} \in\mathcal{N}^*_{\theta}/\mathbf{K}$
    $$\mathbf{G} \cdot \mathcal{O} = \mathbf{G} \cdot \eta(\mathcal{O})$$
    Inside this co-adjoint $\mathbf{G}$-orbit, $\mathcal{O}$ is a Lagrangian submanifold, and $\eta(\mathcal{O})$ is a real form. 
\end{enumerate}
\end{theorem}

\subsection{Associated varieties}\label{sec:assvar}

Equip $U(\fg)$ with its usual filtration and let $I \subset U(\fg)$ be a two-sided ideal. Then $\gr(I)$ is a graded ideal in the commutative ring $\mathrm{gr}U(\mathfrak{g}) \simeq S(\mathfrak{g}) \simeq \mathbb{C}[\mathfrak{g}^*]$.

\begin{definition}
The \emph{associated variety} of $I$ is the $\mathbf{G}$ and $\mathbb{C}^{\times}$-invariant, Zariski-closed subset of $\mathfrak{g}^*$ defined by the graded ideal $\gr(I) \subset \mathbb{C}[\mathfrak{g}^*]$
$$\AV(I) := V(\gr(I)) \subset \mathfrak{g}^*$$
\end{definition}

If $I \cap Z(\mathfrak{g}) \subset Z(\mathfrak{g})$ is an ideal of finite codimension, then $\AV(I) \subset \mathcal{N}^*$ (see e.g. \cite[Thm 5.7]{Vogan1991}. If $\mathcal{O}^{\mathbb{C}}_1,...,\mathcal{O}^{\mathbb{C}}_n$ are the open $\mathbf{G}$-orbits in $\AV(I)$, then $\overline{\mathcal{O}}^{\mathbb{C}}_i$ are its irreducible components.

\begin{theorem}[Joseph, Borho-Brylinski, \cite{Joseph1985}, \cite{BorhoBrylinski1985}]\label{thm:Josephirreducibility}
If $I$ is primitive (i.e. the annihilator of an irreducible $\mathfrak{g}$-module), then $\AV(I)$ is irreducible, i.e. there is a $\mathbf{G}$-orbit $\mathcal{O}^{\mathbb{C}} \subset \mathcal{N}^*$ such that
$$\AV(I) = \overline{\mathcal{O}}^{\mathbb{C}}$$
\end{theorem}

Now suppose $X$ is a $(\mathfrak{g},\mathbf{K})$-module. The annihilator of $X$
$$\mathrm{Ann}(X) := \{a \in U(\mathfrak{g}) \mid aX=0\}$$
is a two-sided ideal in $U(\mathfrak{g})$. Thus, we can define its associated variety $\AV(\mathrm{Ann}(X)) \subset \mathfrak{g}^*$. There is a refinement of this invariant, which can distinguish between $(\mathfrak{g},\mathbf{K})$-modules with coinciding annihilators.

\begin{definition}[\cite{Vogan1991}]
A filtration on $X$
$$...\subseteq X_{-1} \subseteq X_0 \subseteq X_1 \subseteq ... , \qquad \bigcap_m X_m = 0, \qquad \bigcup_m X_m = X$$
by complex subspaces is \emph{compatible} if
\begin{enumerate}
\item $U_m(\mathfrak{g})X_n \subseteq X_{m+n} \qquad \forall \ m,n \in \mathbb{Z}$
\item $KX_m \subseteq X_m \qquad \forall \ m \in \mathbb{Z}$
\end{enumerate}
The first condition allows us to define on $\gr(X)$ the structure of a graded $S(\mathfrak{g})$-module. The second condition allows us to define on $\gr(X)$ a graded algebraic $\mathbf{K}$-action. These two structures are compatible in the following ways

\begin{enumerate}
\item The action map $S(\mathfrak{g}) \otimes \gr(X) \to \gr(X)$ is $\mathbf{K}$-equivariant,
\item The subspace $\mathfrak{k} \subset \mathfrak{g} \subset S(\mathfrak{g})$ acts by $0$ on $\gr(X)$
\end{enumerate}
In short, $\gr(X)$ has the structure of a graded, $\mathbf{K}$-equivariant $S(\mathfrak{g}/\mathfrak{k})$-module. A compatible filtration is \emph{good} if
\begin{enumerate}[resume]
\item\label{cond3} $\gr(X)$ is a finitely-generated $S(\mathfrak{g}/\mathfrak{k})$-module
\end{enumerate}
\end{definition}

Thus, for any good filtration, $\mathrm{gr}(X)$ can be identified with a graded, $\mathbf{K}$-equivariant coherent sheaf on $(\mathfrak{g}/\mathfrak{k})^*$. We note that if $X$ is finitely-generated, good filtrations exist. For example, if $X_0 \subset X$ is a finite-dimensional $\mathbf{K}$-invariant generating subspace, we can define
$$X_m := U_m(\mathfrak{g})X_0 \quad m \geq 0$$

Now, assume $X$ has finite length. Then $X$ is finitely-generated as a $U(\mathfrak{g})$-module. If we choose a good filtration on $X$, there is an obvious containment (of ideals)
$$\gr(\mathrm{Ann}(X)) \subseteq \mathrm{Ann}(\gr(X)) $$
and hence a containment (of sets)
$$\mathrm{Supp}(\gr(X)) \subseteq \AV(\mathrm{Ann}(X))$$
where $\mathrm{Supp}(\gr(X)) = V(\mathrm{Ann}(\gr(X))$ denotes the (set-theoretic) support. Since $X$ has finite-length, $\mathrm{Ann}(X) \cap \mathfrak{z}(\mathfrak{g}) \subset \mathfrak{z}(\mathfrak{g})$ is an ideal of finite codimension. By the remarks preceding Theorem \ref{thm:Josephirreducibility} there is a containment
$$\mathrm{Supp}(\gr(X)) \subseteq \mathcal{N}^* \cap (\mathfrak{g}/\mathfrak{k})^* = \mathcal{N}_{\theta}^*$$
Let $M^f(\mathfrak{g},\mathbf{K})$ be the abelian category of finite-length $(\mathfrak{g},\mathbf{K})$-modules and let $K^f(\mathfrak{g},\mathbf{K})$ be its Grothendieck group. Then $K^f(\mathfrak{g},\mathbf{K})$ is a free $\mathbb{Z}$-module with basis equal to the set of (isomorphism classes of) irreducible $(\mathfrak{g},\mathbf{K})$-modules. If $X \in M^f(\mathfrak{g},\mathbf{K})$, write $[X]  \in K^f(\mathfrak{g},\mathbf{K})$ for its class. By the remarks above, $\gr(X)$ defines a class in $K^{\mathbf{K}}(\mathcal{N}_{\theta}^*)$. Although $\gr(X)$ depends on the filtration used to define it, its class $[\gr(X)]$ does not. More precisely

\begin{proposition}[\cite{Vogan1991}, Proposition 2.2]\label{prop:grprop}
$\gr$ defines a group homomorphism
$$
K^f(\mathfrak{g},\mathbf{K}) \to K^{\mathbf{K}}(\mathcal{N}_{\theta}^*), \qquad X \mapsto [\gr(X)]
$$
\end{proposition}

Thus, we can define

\begin{definition}\label{def:AVfinitelength}
Let $X$ be a $(\mathfrak{g},\mathbf{K})$-module of finite-length. The \emph{associated variety} of $X$ is the $\mathbf{K}$ and $\mathbb{C}^{\times}$-invariant, Zariski-closed subset of $\mathcal{N}_{\theta}^*$ defined by the graded ideal $\mathrm{Ann}(\gr(X)) \subset S(\mathfrak{g}/\mathfrak{k})$
$$\AV(X) := \mathrm{Supp}(\gr X) = V(\mathrm{Ann}(\gr(X))) \subseteq \mathcal{N}_{\theta}^*$$
By Proposition \ref{prop:grprop}, $\mathrm{AV}(X)$ is well-defined.
\end{definition}

Hence if $X$ has finite-length, $\AV(X)$ is a finite union of $\mathbf{K}$-orbits. If $\mathcal{O}_1,...,\mathcal{O}_n$ are the open $\mathbf{K}$-orbits in $\AV(X)$, then $\overline{\mathcal{O}}_i$ are the irreducible components. If $X$ is irreducible, then the $\mathbf{K}$-orbits $\mathcal{O}_1,...,\mathcal{O}_n$ are related to the $\mathbf{G}$-orbit $\mathcal{O}^{\mathbb{C}}$ of Proposition \ref{thm:Josephirreducibility} by the following result of Vogan

\begin{theorem}[\cite{Vogan1991}, Theorem 8.4]\label{thm:twoAVs}
Let $X$ be an irreducible $(\mathfrak{g},\mathbf{K})$-module. Let $\mathcal{O}_1,...,\mathcal{O}_n$ be the open $\mathbf{K}$-orbits in $\AV(X)$ and let $\mathcal{O}^{\mathbb{C}}$ be the open $\mathbf{G}$-orbit in $\AV(I)$. Then 
$$\mathcal{O}^{\mathbb{C}} = \mathbf{G}\cdot \mathcal{O}_i$$
for each $i=1,...,n$. In particular (by Theorem \ref{thm:kostantsekiguchi2})
$$\dim(\mathcal{O}_i) = \frac{1}{2}\dim(\mathcal{O}^{\mathbb{C}}) \qquad \text{for } i=1,...,n $$
\end{theorem}

Let $\mathrm{Rep}^f(\mathbf{K})$ be the abelian category of admissible representations of $\mathbf{K}$. Restriciton to $\mathbf{K}$ defines an exact functor 
$$\mathrm{res}^{(\mathfrak{g},\mathbb{K})}_{\mathbf{K}}: M^f(\mathfrak{g},\mathbf{K}) \to \mathrm{Rep}^f(\mathbf{K})$$
Let $K^f(\mathbf{K})$ be the Grothendieck group of $\mathrm{Rep}^f(\mathbf{K})$. Note that $K^f(\mathbf{K})$ is identified with functions $\mathbb{Z} \to \widehat{\mathbf{K}}$. The restriction functor induces a group homomorphism
$$\mathrm{res}^{(\mathfrak{g},\mathbb{K})}_{\mathbf{K}}: K^f(\mathfrak{g},\mathbf{K}) \to K^f(\mathbf{K})$$
to the Grothendieck group $K^f(\mathbf{K})$ of $\mathrm{Rep}^f(\mathbf{K})$.

Now suppose $M \in \Coh^{\mathbf{K}}(\mathcal{N}_{\theta}^*)$. Then $\Gamma(\mathcal{N}_{\theta}^*,M)$ has the structure of a $\mathbf{K}$-equivariant $\mathbb{C}[\mathcal{N}_{\theta}^*]$-module. Since $\mathbf{K}$ is reductive, this module is admissible (when regarded as a representation of $\mathbf{K})$. Hence, we obtain a functor
$$\mathrm{res}^{\mathrm{coh}}_{\mathbf{K}}: \Coh^{\mathbf{K}}(\mathcal{N}_{\theta}^*) \to \mathrm{Rep}^f(\mathbf{K})$$
which is exact, since $\mathcal{N}_{\theta}^*$ is an affine variety. This functor induces a group homomorphism
$$\mathrm{res}^{\mathrm{coh}}_{\mathbf{K}}: K^{\mathbf{K}}(\mathcal{N}_{\theta}^*) \to K^f(\mathbf{K})$$
It is clear from definitions that the following diagram commutes
\begin{center}
\begin{tikzcd}
K^f(\mathfrak{g},\mathbf{K}) \arrow[r, "\gr"] \arrow[dr, "\mathrm{res}^{(\mathfrak{g},\mathbb{K})}_{\mathbf{K}}"] & K^{\mathbf{K}}(\mathcal{N}_{\theta}^*)  \arrow[d, "\mathrm{res}^{\mathrm{coh}}_{\mathbf{K}}"]\\
& K^f(\mathbf{K})
\end{tikzcd}
\end{center}

\begin{theorem}[\cite{AdamsVogan2019}, Corollary 6.4]\label{thm:resKinjective}
The restriction map 
$$\mathrm{res}^{\mathrm{coh}}_{\mathbf{K}}: K^{\mathbf{K}}(\mathcal{N}_{\theta}^*) \to K^f(\mathbf{K})$$
is injective.
\end{theorem}

An immediate consequence of Theorem \ref{thm:resKinjective} is the somewhat surprising fact that the associated variety of a finite-length $(\mathfrak{g},\mathbf{K})$-module is determined by its $\mathbf{K}$-types.

\subsection{Parabolic induction: general theory}\label{subsec:induction}

In this section, we will review the general theory of parabolic induction. See e.g. \cite{Vogan1981} or \cite{KnappVogan1995} for more details and proofs.

Let $\mathbf{Q} \subset \mathbf{G}$ be a parabolic subgroup. We will always assume that $\mathbf{Q}$ has a Levi decomposition $\mathbf{Q} = \mathbf{L} \mathbf{U}$ with $\theta$-stable Levi factor $\mathbf{L} \subset \mathbf{G}$. Parabolic induction is a left-exact functor
\begin{equation}\label{eqn:zuckermanfunctor}
    \mathbf{I}^{(\mathfrak{g},\mathbf{K})}_{(\mathfrak{l},\mathbf{L}\cap \mathbf{K})}: M(\mathfrak{l},\mathbf{L} \cap \mathbf{K}) \to M(\mathfrak{g},\mathbf{K})
\end{equation}
Roughly speaking, $\mathbf{I}^{(\mathfrak{g},\mathbf{K})}_{(\mathfrak{l},\mathbf{Q}\cap \mathbf{K})}W$ is the $(\mathfrak{g},\mathbf{K})$-module
$$\mathbf{K}-\text{finite vectors in } \mathrm{Hom}_{\mathfrak{q}}(U(\mathfrak{g}),W \otimes \det(\mathfrak{u}))$$
where $\det(\mathfrak{u})$ is the top exterior power of $\mathfrak{u}$. This definition is not quite correct (or meaningful, strictly speaking) if $\mathbf{K}$ is disconnected. For a precise definition, we refer the reader to \cite[Chapter 5]{Vogan1981}.

The category $M(\mathrm{g},\mathbf{K})$ has enough injectives (see \cite{KnappVogan1995}, Corollary 2.26). Hence, we can define the right derived functors:
$$R^i\mathbf{I}^{(\mathfrak{g},\mathbf{K})}_{(\mathfrak{l},\mathbf{L}\cap \mathbf{K})}: M(\mathfrak{l},\mathbf{L}\cap \mathbf{K}) \to M(\mathfrak{g},\mathbf{K})$$
The following Proposition catalogs the key properties of these functors.

\begin{proposition}\label{prop:propsofind}
The following are true:
\begin{itemize}
    \item[(i)] If $W \in M^f(\mathfrak{l},\mathbf{L} \cap \mathbf{K})$, then
    $$R^i\mathbf{I}^{(\mathfrak{g},\mathbf{K})}_{(\mathfrak{l},\mathbf{L} \cap \mathbf{K})}W \in M^f(\mathfrak{g},\mathbf{K}) \qquad \forall i \geq 0$$
    \item[(ii)] Let $\mathfrak{h} \subset \mathfrak{l}$ be a Cartan subalgebra and suppose $W \in M(\mathfrak{l},\mathbf{L} \cap \mathbf{K})$ has infinitesimal character $\lambda \in \mathfrak{h}^*$. Then
    $$R^i\mathbf{I}^{(\mathfrak{g},\mathbf{K})}_{(\mathfrak{l},\mathbf{L} \cap \mathbf{K})}W \text{ has infinitesimal character } \lambda + \rho(\mathfrak{u}) \in \mathfrak{h}^* \qquad \forall i \geq 0$$ 
    \item[(iii)] There is an $s \geq 0$ such that for every $W \in M(\mathfrak{l},\mathbf{L} \cap \mathbf{K})$ we have
    $$R^i\mathbf{I}^{(\mathfrak{g},\mathbf{K})}_{(\mathfrak{l},\mathbf{L} \cap \mathbf{K})}W =0 \qquad \forall i >s$$
\end{itemize}
\end{proposition}

In light of Proposition \ref{prop:propsofind}(i) and (iii), there is a group homomorphism
\begin{equation}\label{eqn:eulerinduction}
I(\mathfrak{l},\mathfrak{q}, \cdot): K^f(\mathfrak{l}, \mathbf{L} \cap \mathbf{K}) \to K^f(\mathfrak{g},\mathbf{K}) \qquad I(\mathfrak{l},\mathfrak{q},[W]) : = \sum_i (-1)^i[R^i\mathbf{I}^{(\mathfrak{g},\mathbf{K})}_{(\mathfrak{l},\mathbf{L}\cap \mathbf{K})} W]
\end{equation}

\begin{proposition}\label{prop:indbystages}
Suppose $\mathbf{Q}' = \mathbf{L}' \mathbf{U}'$ is a parabolic subgroup of $\mathbf{L}$ with $\theta$-stable Levi factor $\mathbf{L}'$. Then $\mathbf{Q}'\mathbf{U}$ is a parabolic subgroup of $\mathbf{G}$ with $\theta$-stable Levi $\mathbf{L}'$ and
$$I(\mathfrak{l},\mathfrak{q},\cdot) \circ I(\mathfrak{l}',\mathfrak{q}',\cdot) = I(\mathfrak{l}',\mathfrak{q}'\oplus \mathfrak{u},\cdot)$$
\end{proposition}

There are two `extreme' cases of parabolic induction which are particularly well-understood: real parabolic and cohomological induction. We will briefly review these special cases.

\subsection{Real parabolic induction}

Assume that $\mathbf{Q}$ is stable under $\sigma$. Then $Q := \mathbf{Q}^{\sigma}$ is a parabolic subgroup of $G$ and there is an identification
\begin{equation}\label{eqn:realinductioninfinitesimal}
I^{(\mathfrak{g},\mathbf{K})}_{(\mathfrak{l}, \mathbf{L}\cap \mathbf{K})}W \cong_{(\mathfrak{g},\mathbf{K})} \mathrm{Ind}^G_Q (W \otimes |\rho(\mathfrak{u})|) \quad W  \in M(\mathfrak{l}, \mathbf{L}\cap \mathbf{K})\end{equation}
where $\mathrm{Ind}^G_Q$ is the usual (analytically-defined) functor of parabolic induction (see \cite[Sec 11.2]{KnappVogan1995} for a proof).

The main facts we will need in this case are the following.

\begin{theorem}\label{thm:realinductionexact}
Suppose $\mathbf{Q}$ is $\sigma$-stable. Then the functor
$$I^{(\mathfrak{g},\mathbf{K})}_{(\mathfrak{l},\mathbf{L} \cap \mathbf{K})}: M^f(\mathfrak{l}, \mathbf{L} \cap \mathbf{K}) \to M^f(\mathfrak{g},\mathbf{K})$$
\begin{enumerate}
    \item is exact, and
    \item takes nonzero modules to nonzero modules
\end{enumerate}
\end{theorem}

\begin{proof}
Part $(1)$ is \cite[Prop 11.52]{KnappVogan1995}. Part $(2)$ is clear from the analytic description of $I^{(\mathfrak{g},\mathbf{K})}_{(\mathfrak{l},\mathbf{L} \cap \mathbf{K})}$.
\end{proof}

Let $\mathfrak{h}_0^{\mathrm{split}}$ be a maximally split $\theta$-stable Cartan subalgebra of $\mathfrak{g}_0$. Choose an element $a \in \mathfrak{a}_0^{\mathrm{split}}$ such that $\alpha(a) \neq 0$ for every nonimaginary $\alpha \in \Delta(\mathfrak{g},\mathfrak{h}^{\mathrm{split}})$ and define the parabolic subalgebra
$$\mathfrak{l}^{\mathrm{min}} := \mathfrak{h}^{\mathrm{split}} \oplus \bigoplus_{\alpha(a) = 0} \mathfrak{g}_{\alpha} \qquad \mathfrak{u}^{\mathrm{min}} := \bigoplus_{\alpha(a)>0} \mathfrak{g}_{\alpha} \qquad \mathfrak{q}^{\mathrm{min}} = \mathfrak{l}\oplus \mathfrak{u}$$
Since $\sigma(a) = a$, $\mathfrak{q}^{\mathrm{min}}$ is $\sigma$-stable and since $\theta(a) = -a$, $\mathfrak{l}^{\mathrm{min}}$ is $\theta$-stable. The corresponding parabolic subgroup $Q^{\mathrm{min}} = (\mathbf{Q}^{\mathrm{min}})^{\sigma}$ is minimal among parabolics of $G$. We will eventually need the following deep result of Casselman:

\begin{theorem}[Casselman Subrepresentation Theorem, \cite{Casselman1978}]\label{thm:Casselman}
Let $X$ be an irreducible $(\mathfrak{g},\mathbf{K})$-module. Then there is a (finite-dimensional) irreducible representation $V$ of $L^{\mathrm{min}}$ and an embedding of $(\mathfrak{g},\mathbf{K})$-modules
$$X \subseteq \mathbf{I}^{(\mathfrak{g},\mathbf{K})}_{(\mathfrak{l}^{\mathrm{min}},\mathbf{L}^{\mathrm{min}}\cap \mathbf{K})}V$$
\end{theorem}

\subsection{Cohomological induction}

Assume that $\mathbf{Q}$ is stable under $\theta$. Choose a Cartan subalgebra $\mathfrak{h} \subset \mathfrak{l}$, and let $W$ be a finite-length $(\mathfrak{l},\mathbf{L} \cap \mathbf{K})$-module with infinitesimal character $\lambda \in \mathfrak{h}^*$. We say that $W$ (or $\lambda$) is \emph{in the weakly good range} if 
$$\mathrm{Re} \langle \lambda+ \rho(\mathfrak{u}), \alpha^{\vee}\rangle  \geq 0 \qquad \forall \alpha \in \Delta(\mathfrak{u},\mathfrak{h})$$
When applied to irreducible $(\mathfrak{l},\mathbf{L} \cap \mathbf{K})$-modules in the weakly good range, the functors $R^i\mathbf{I}^{(\mathfrak{g},\mathbf{K})}_{(\mathfrak{l},\mathbf{L} \cap \mathbf{K})}$ are particularly well-behaved.

\begin{theorem}[\cite{KnappVogan1995},Theorem 8.2]\label{thm:cohindirreducible}
There is an integer $t \geq 0$ (depending only on the parabolic $\mathbf{Q} \subset \mathbf{G}$) such that for every irreducible $(\mathfrak{l}, \mathbf{L} \cap \mathbf{K})$-module $W$ in the weakly good range, the $(\mathfrak{g},\mathbf{K})$-module $R^i\mathbf{I}^{(\mathfrak{g},\mathbf{K})}_{(\mathfrak{l},\mathbf{L} \cap \mathbf{K})}W$
\begin{enumerate}
    \item is irreducible, or $0$, if $i=t$, and
    \item is $0$ if $i \neq t$
\end{enumerate}
The functor $R^t\mathbf{I}^{(\mathfrak{g},\mathbf{K})}_{(\mathfrak{l},\mathbf{L} \cap \mathbf{K})}$ is called \emph{cohomological induction}. 
\end{theorem}

In the setting of Theorem \ref{thm:cohindirreducible}, one can formulate necessary and sufficient conditions on $W$ guaranteeing that $R^t\mathbf{I}^{(\mathfrak{g},\mathbf{K})}_{(\mathfrak{l},\mathbf{L} \cap \mathbf{K})}W \neq 0$.
For this, we will need to define the \emph{minimal $\mathbf{K}$-types} of a $(\mathfrak{g},\mathbf{K})$-module $X$. Choose a maximally compact $\theta$-stable Cartan subalgebra $\mathfrak{h}\subset \mathfrak{g}$ and a positive system $\Delta^+(\mathfrak{k},\mathfrak{t})$. If $\tau_{\mu}$ is an irreducible $\mathbf{K}$-representation with highest weight $\mu \in \mathfrak{t}^*$, define 
$$|\tau_{\mu}| := B(\mu + 2\rho_{\mathfrak{k}}, \mu + 2\rho_{\mathfrak{k}})$$
A minimal $\mathbf{K}$-type of $X$ is a $\mathbf{K}$-type $\tau_{\mu}$ with minimal norm among all $\mathbf{K}$-types occuring in $X$. It is easy to see that minimal $\mathbf{K}$-types exist (if $X \neq 0$) and are independent of $\Delta^+(\mathfrak{k},\mathfrak{t})$. 

\begin{theorem}[\cite{KnappVogan1995}, Theorem 10.44,]\label{thm:LKTscohind}
Assume $\mathfrak{h}$ is a maximally compact $\theta$-stable Cartan subalgebra of $\mathfrak{l}$, and choose a positive system $\Delta^+(\mathfrak{k},\mathfrak{t})$. Suppose $W$ is a finite-length $(\mathfrak{l},\mathbf{L} \cap \mathbf{K})$-module in the weakly good range. Write $\mu_1,...,\mu_n \in \mathfrak{h}^*$ for highest-weights of the minimal $\mathbf{L} \cap \mathbf{K}$-types of $W$. Then $R^t\mathbf{I}^{(\mathfrak{g},\mathbf{K})}_{(\mathfrak{l},\mathbf{L} \cap \mathbf{K})}W \neq 0$ if and only if some of the weights
$$\mu_i + 2\rho(\mathfrak{u} \cap \mathfrak{p})$$
are dominant for $\Delta^+(\mathfrak{k},\mathfrak{t})$. In this case, the dominant weights of this form are minimal $\mathbf{K}$-types of $R^t\mathbf{I}^{(\mathfrak{g},\mathbf{K})}_{(\mathfrak{l},\mathbf{L} \cap \mathbf{K})}W$.
\end{theorem}

For $W$ a finite-length $(\mathfrak{l},\mathbf{L} \cap \mathbf{K})$-module in the weakly good range, the associated variety of $R^t\mathbf{I}^{(\mathfrak{g},\mathbf{K})}_{(\mathfrak{l},\mathbf{L} \cap \mathbf{K})}W$ can be easily computed. Consider the restriction map
$$\pi_{\mathfrak{q},\theta}: (\mathfrak{g}/(\mathfrak{u}+\mathfrak{k}))^* \to (\mathfrak{q}/(\mathfrak{u}+\mathfrak{k}))^* \simeq (\mathfrak{l}/(\mathfrak{l}\cap \mathfrak{k}))^*$$
Note that
$$\pi_{\mathfrak{q},\theta}^{-1}(\mathcal{N}_{\mathfrak{l},\theta}^*) \subseteq \mathcal{N}_{\mathfrak{g},\theta}^*$$
The following statement is well-known to the experts. See, e.g. \cite[Prop 5.4]{Trapa2001} for a proof.

\begin{proposition}\label{prop:AVcohind}
Suppose $W$ is a finite-length $(\mathfrak{l},\mathbf{L}\cap \mathbf{K})$-module in the weakly good range and assume 
$$R^t\mathbf{I}^{(\mathfrak{g},\mathbf{K})}_{(\mathfrak{l},\mathbf{L} \cap \mathbf{K})} \neq 0$$
Then 
\begin{equation}\label{eqn:AVcohind}
\AV(R^t\mathbf{I}^{(\mathfrak{g},\mathbf{K})}_{(\mathfrak{l},\mathbf{L} \cap \mathbf{K})}) = \mathbf{K}\left(\pi_{\mathfrak{q},\theta}^{-1}(\AV(W))\right) \subset \mathcal{N}_{\mathfrak{g},\theta}^*
\end{equation}
\end{proposition}

\begin{remark}
If we regard $\AV(R^t\mathbf{I}^{(\mathfrak{g},\mathbf{K})}_{(\mathfrak{l},\mathbf{L} \cap \mathbf{K})})$ as a subset of $\mathcal{N}_{\mathfrak{g},\theta}$ and $\AV(W)$ as a subset of $\mathcal{N}_{\mathfrak{l},\theta}$ (as will sometimes be convenient), (\ref{eqn:AVcohind}) becomes
$$\AV(R^t\mathbf{I}^{(\mathfrak{g},\mathbf{K})}_{(\mathfrak{l},\mathbf{L} \cap \mathbf{K})}) = \mathbf{K} \left(\AV(W) + \mathfrak{u} \cap \mathfrak{p}\right) $$
\end{remark}

\section{Two Classifications of Principal Unipotent Representations}\label{sec:classification}

Let $\cO \subset \cN$ be the principal nilpotent orbit. In this section, we will give two classifications of $\unip_R(\cO)$. The parameters (\emph{principal unipotent Beilinson-Bernstein parameters} and \emph{principal unipotent Zuckerman parameters}) are defined in Sections \ref{subsec:unipotentBB} and \ref{subsec:Zuckerman}, respectively. The main result is Corollary \ref{cor:principalunipotentreps}.

\subsection{Beilinson-Bernstein parameters}\label{subsec:BBparameters}

Fix a Borel subalgebra $\mathfrak{b} \subset \mathfrak{g}$ and a Cartan subalgebra $\mathfrak{h} \subset \mathfrak{b}$. Let $\lambda \in \mathfrak{h}^*$ be an \emph{integrally dominant} weight
\begin{equation}\label{eqn:integraldominance}
\langle \lambda, \alpha^{\vee}\rangle \notin \{-1,-2,...\} \qquad \forall \alpha \in \Delta(\mathfrak{b},\mathfrak{h})
\end{equation}
%

\begin{definition}\label{def:BB}
A Beilinson-Bernstein parameter for $G$ of infinitesimal character $\lambda$ (a BB parameter for short) is a $\mathbf{K}$-conjugacy class of triples $(\fh,\fb,\chi)$ consisting of
\begin{itemize}
    \item[(i)] a $\theta$-stable Cartan subalgebra $\fh = \mathfrak{t}\oplus \mathfrak{a} \subset \mathfrak{g}$,
    \item[(ii)] a Borel subalgebra $\fb = \mathfrak{h} \oplus \mathfrak{n} \subset \fg$, and
    \item[(iii)] a one-dimensional $(\fh, \mathbf{T})$-module $\chi$ such that $d\chi + \rho(\mathfrak{n}) = \lambda$
\end{itemize}
Denote the $\mathbf{K}$-conjugacy class of $(\mathfrak{h},\mathfrak{b},\chi)$ by $[\mathfrak{h},\mathfrak{b},\chi]$ and denote the set of such classes by $\BB_{\lambda}(G)$.
\end{definition}

\begin{remark}
If $[\mathfrak{h},\mathfrak{b},\chi] \in \BB_{\lambda}(G)$, we can (and will) choose $\mathfrak{h}$ to be stable under $\sigma$. This allows us to define the Cartan subgroup $H := Z_G(\mathfrak{h})$. Now one-dimensional $(\mathfrak{h},T)$-modules correspond precisely to continuous characters of $H$.
\end{remark}

\begin{proposition}
$\BB_{\lambda}(G)$ is a finite set.
\end{proposition}

\begin{proof}
It is classically known that there are only finitely many $\mathbf{K}$-conjugacy classes of Cartan and Borel subalgebras of $\fg$ (see \cite[Thms 1,2]{Wolf1974}). For each pair $(\fh,\fb)$ consisting of a $\theta$-stable Cartan subalgebra $\fh=\mathfrak{t}\oplus \mathfrak{a}$ and Borel subalgebra $\mathfrak{b}=\mathfrak{h} \oplus \mathfrak{n}$, there is a finite number of one-dimensional $(\mathfrak{h},\mathbf{T})$-modules $\chi$ satisfying $d\chi+\rho(\mathfrak{n})=\lambda$. Indeed, such a module is uniquely determined by an algebraic character of $\mathbf{T}$, of which there are finitely many.
\end{proof}

Beilinson-Bernstein parameters can be parabolically induced. The assignment
$$(\mathfrak{h},\mathfrak{b},\chi) \mapsto I(\mathfrak{b},\mathfrak{b},\chi) :=  \sum_i (-1)^i[R^i\mathbf{I}^{(\mathfrak{g},\mathbf{K})}_{(\mathfrak{h},\mathbf{T})} \chi] \in K_{\lambda}^f(\fg,\mathbf{K})$$
(cf. Section \ref{subsec:induction}) is constant on $\mathbf{K}$-conjugacy class and hence gives rise to a function
$$I: \BB_{\lambda}(G) \to K_{\lambda}^f(\mathfrak{g},\mathbf{K})$$
As our terminology suggests, this function admits an alternative description via the Beilinson-Bernstein localization theory. We will briefly summarize the main ideas (for more details, we refer the reader to \cite{HechtMilicicSchmidWolf}). 

Let $\mathcal{B} = \{\fb \subset \fg\}$ be the flag variety for $\mathbf{G}$. Note that $\mathbf{K}$ acts on $\mathcal{B}$ with finitely many orbits. The functional $\lambda \in \fh^*$ determines a sheaf $\mathcal{D}_{\lambda-\rho}$ of twisted differential operators (TDOs) on $\mathcal{B}$. We will consider the abelian category $M(\mathcal{D}_{\lambda-\rho},\mathbf{K})$ of $\mathbf{K}$-equivariant quasi-coherent $\mathcal{D}_{\lambda-\rho}$-modules on $\mathcal{B}$. The irreducible objects in this category are parameterized by $\mathbf{K}$-orbits on $\BB_{\lambda}(G)$. The construction is as follows. Fix a parameter $[\fh,\fb, \chi] \in \BB_{\lambda}(G)$. The borel subalgebra $\fb \subset \fg$ determines a $\mathbf{K}$-orbit $Z = \mathbf{K}\cdot \fb \subset \mathcal{B}$. Denote the locally-closed embedding by $j:Z \subset \mathcal{B}$. On the $\mathbf{K}$-orbit $Z$, there is a sheaf of TDOs $\mathcal{D}_{\lambda-\rho}^Z$ obtained by restricting $\mathcal{D}_{\lambda-\rho}$ along $Z \subset \mathcal{B}$, and the one-dimensional $(\fh,\mathbf{T})$-module $\chi$ determines an irreducible object $\mathcal{L}_{\chi} \in M(\mathcal{D}^Z_{\lambda-\rho},\mathbf{K})$. There is a left-exact functor $j_!: M(\mathcal{D}^Z_{\lambda-\rho},\mathbf{K}) \to M(\mathcal{D}_{\lambda-\rho},\mathbf{K})$ called the \emph{exceptional pushforward}. The object $j_!\mathcal{L}_{\chi} \in M(\mathcal{D}_{\lambda-\rho},\mathbf{K})$ contains a unique irreducible subobject, and this defines a bijection
\begin{align*}
\BB_{\lambda}(G) &\xrightarrow{\sim} \{\text{irreducibles in } M(\mathcal{D}_{\lambda-\rho},\mathbf{K})\}\\ [\mathfrak{h},\mathfrak{b},\chi] &\mapsto \text{unique irreducible subobject in } j_!\mathcal{L}_{\chi}
\end{align*}
The $\mathbf{G}$-action on $\mathcal{B}$ induces an algebra homomorphism
$$\phi: U(\mathfrak{g}) \to \Gamma(\mathcal{B}, \mathcal{D}_{\lambda-\rho})$$
This map is surjective with kernel equal to the two-sided ideal generated by the kernel of the infinitesimal character corresponding to $\lambda$ under the Harish-Chandra isomorphism. If $\mathcal{M} \in M(\mathcal{D}_{\lambda-\rho},\mathbf{K})$, then $\Gamma(\mathcal{B},\mathcal{M})$ can be regarded (using $\phi$) as a finite-length $(\fg,\mathbf{K})$-module of infinitesimal character $\lambda$. This defines a functor
$$\Gamma: M(\mathcal{D}_{\lambda-\rho},\mathbf{K}) \to M^f_{\lambda}(\fg,\mathbf{K})$$
\begin{theorem}\label{thm:globalsections}
Under the dominance condition (\ref{eqn:integraldominance}), $\Gamma$ is exact and induces a bijection
$$\Gamma: \{\text{irreducibles in } M(\mathcal{D}_{\lambda-\rho},\mathbf{K}) \text{ with non-zero global sections}\} \xrightarrow{\sim} \{\text{irreducibles in } M_{\lambda}(\fg,\mathbf{K})\}$$
\end{theorem}

We now have two methods for producing $(\fg,\mathbf{K})$-modules from BB parameters: parabolic induction $[\mathfrak{h},\mathfrak{b},\chi] \mapsto I[\mathfrak{h},\mathfrak{b},\chi]$ and the Beilinson-Bernstein construction $[\mathfrak{h},\mathfrak{b},\chi] \mapsto \Gamma(\mathcal{B},j_!\mathcal{L}_{\chi})$. The Duality Theorem of Hecht, Milicic, Schmid, and Wolf asserts that these two constructions (essentially) coincide.

\begin{theorem}[\cite{HechtMilicicSchmidWolf}, Theorem 4.3]\label{thm:dualitythm}
Suppose $[\fh,\fb,\chi] \in \BB_{\lambda}(G)$. Then there is an equality in $K^f(\fg,\mathbf{K})$
$$[\Gamma(\mathcal{B},j_!\mathcal{L}_{\chi})] = \pm I[\fh,\fb,\chi]$$
\end{theorem}

\begin{corollary}\label{cor:Gammainjective}
Let $[\mathfrak{h}_1,\mathfrak{b}_1,\chi_1], [\mathfrak{h}_2,\mathfrak{b}_2,\chi_2] \in \BB_{\lambda}(G)$. Suppose $I[\mathfrak{h}_1,\mathfrak{b}_1,\chi_1]$ and  $I[\mathfrak{h}_2,\mathfrak{b}_2,\chi_2]$ are nonzero and irreducible, and that
$$I[\mathfrak{h}_1,\mathfrak{b}_1,\chi_1] = \pm I[\mathfrak{h}_2,\mathfrak{b}_2,\chi_2]$$
Then $[\mathfrak{h}_1,\mathfrak{b}_1,\chi_1] = [\mathfrak{h}_2,\mathfrak{b}_2,\chi_2]$.
\end{corollary}

\subsection{Principal unipotent Beilinson-Bernstein parameters}\label{subsec:unipotentBB}

We will construct the elements of $\mathrm{Unip}_R(\cO)$ from a very special set of BB parameters. To define this set of parameters, we will need several preliminary notions.

\begin{definition}\label{def:largetypeZtypeL}
Let $\mathfrak{h} \subset \mathfrak{g}$ be a $\theta$-stable Cartan subalgebra and let $\Delta^+(\mathfrak{g},\mathfrak{h}) \subset \Delta(\mathfrak{g},\mathfrak{h})$ be a positive system. We say that $\Delta^+(\mathfrak{g},\mathfrak{h})$ is 
\begin{itemize}
    \item[(i)] large if every imaginary simple root is noncompact.
    \item[(ii)] small if every imaginary simple root is compact
    \item[(iii)] type Z if for every complex simple root $\alpha$
    $$\theta(\alpha) \in \Delta^+(\mathfrak{g},\mathfrak{h})$$
    \item[(iv)] type L if for every complex simple root $\alpha$
    $$\theta(\alpha) \in -\Delta^+(\mathfrak{g},\mathfrak{h})$$
\end{itemize}
If $[\fh,\mathfrak{b},\chi] \in \mathrm{BB}_{\lambda}(G)$, the Borel subalgebra $\mathfrak{b} \subset \mathfrak{g}$ defines a positive system $\Delta^+(\mathfrak{g},\mathfrak{h}) = \Delta(\mathfrak{b},\mathfrak{h})$ for $\Delta(\mathfrak{g},\mathfrak{h})$. We say that $[\mathfrak{h},\mathfrak{b},\chi]$ is large, type Z, or type L according to the properties of this positive system. 
\end{definition}

If $\mathfrak{h} \subset \mathfrak{g}$ is a $\theta$-stable Cartan subalgebra and $\Delta^+(\mathfrak{g},\mathfrak{h})$ is a positive system, there are two naturally defined parabolic subalgebras $\mathfrak{q}^Z,\mathfrak{q}^L \subset \mathfrak{g}$. The first, $\mathfrak{q}^Z$, is the standard parabolic corresponding to the real roots for $\Delta^+(\mathfrak{g},\mathfrak{h})$
\begin{equation}\label{eqn:definitionofqZ}
\mathfrak{l}^Z := \mathfrak{h} \oplus \bigoplus_{\alpha \in \Delta_{\mathbb{R}}} \mathfrak{g}_{\alpha} \qquad \mathfrak{u}^Z := \bigoplus_{\alpha \in \Delta^+ \setminus \Delta_{\mathbb{R}}} \mathfrak{g}_{\alpha} \qquad \mathfrak{q}^Z := \mathfrak{l} \oplus \mathfrak{u}
\end{equation}
The second, $\mathfrak{q}^L$, is the standard parabolic corresponding to the imaginary roots for $\Delta^+(\mathfrak{g},\mathfrak{h})$
\begin{equation}\label{eqn:definitionofqL}
\mathfrak{l}^L := \mathfrak{h} \oplus \bigoplus_{\alpha \in \Delta_{i\mathbb{R}}} \mathfrak{g}_{\alpha} \qquad \mathfrak{u}^L := \bigoplus_{\alpha \in \Delta^+ \setminus \Delta_{i\mathbb{R}}} \mathfrak{g}_{\alpha} \qquad \mathfrak{q}^L := \mathfrak{l} \oplus \mathfrak{u}
\end{equation}

\begin{proposition}\label{prop:typeZLparabolic}
In the setting described above

\begin{itemize}
    \item[(i)] $\mathfrak{q}^Z$ is $\theta$-stable if and only if $\Delta^+(\mathfrak{g},\mathfrak{h})$ is type Z
    \item[(ii)] $\mathfrak{q}^L$ is $\sigma$-stable (i.e. real) if and only if $\Delta^+(\mathfrak{g},\mathfrak{h})$ is type L.
\end{itemize}
\end{proposition}

\begin{proof}
We will only prove the first statement. The second statement can be proved using a similar argument, replacing $\theta$ with $-\theta$. 

For the first statement, one implication is clear: if $\mathfrak{q}^Z$ is $\theta$-stable, then the set of complex positive roots is preserved by $\theta$. In particular, every complex simple root $\alpha \in \Delta^+(\mathfrak{g},\mathfrak{h})$ satisfies $\theta(\alpha) \in \Delta^+(\mathfrak{g},\mathfrak{h})$.

Conversely, suppose $\Delta^+(\mathfrak{g},\mathfrak{h})$ is type Z. To prove that $\mathfrak{q}^Z$ is $\theta$-stable, it suffices to show that $\theta$ preserves the set of complex positive roots. Denote the real, imaginary, and complex simple roots by $\alpha_i, \beta_j$, and $\gamma_k$, respectively.  

Every root $\mu \in \Delta(\mathfrak{g},\mathfrak{h})$ has a unique decomposition
$$\mu = \sum l_i\alpha_i + \sum m_j\beta_j + \sum n_k\gamma_k$$
for integers $l_i,m_j,n_k$ which are either all nonnegative or all nonpositive. 

Now suppose $\mu$ is positive and complex. If all $n_k =0$, then
$$\theta(\mu) = -\sum l_i\alpha_i + \sum m_j \beta_j \notin \Delta(\mathfrak{g},\mathfrak{h})$$
a contradiction. So every complex root has at least one complex simple root in its simple root decomposition. 

Assuming still that $\mu$ is positive and complex, 
$$\theta(\mu) = -\sum l_i\alpha_i + \sum m_j\beta_j + \sum n_k\theta(\gamma_k)$$
By hypothesis, each $\theta(\gamma_k)$ is positive and complex. And therefore, each has a complex simple root in its simple root decomposition. Consequently, $\theta(\mu)$ has at least one complex simple root in its simple root decomposition. Since $\theta(\mu)$ is a root, this means $\theta(\mu) \in \Delta^+(\mathfrak{g},\mathfrak{h})$.
\end{proof}

\begin{definition}\label{def:unipotentBB}
Suppose $[\mathfrak{h},\mathfrak{b},\chi] \in \BB_0(G)$. We say that $[\mathfrak{h},\mathfrak{b},\chi]$ is \emph{principal unipotent} if
\begin{itemize}
    \item[(i)] $\Delta^+(\mathfrak{g},\mathfrak{h})$ is large
    \item[(ii)] $\Delta^+(\mathfrak{g},\mathfrak{h})$ is type $Z$
    \item[(iii)] Every simple real root for $H$ is even for $\chi$
\end{itemize}
Denote the set of principal unipotent BB parameters by $\BB_0^*(G)$.
\end{definition}

\subsection{Cayley transforms and simple reflections on $\BB_0(G)$}

Following \cite{Vogan1981}, we will define two operations on the parameter space $\BB_0(G)$: simple reflections (through complex simple roots) and Cayley transforms (through odd simple real roots).

Let $[\fh,\mathfrak{b},\chi] \in \mathrm{BB}_0(G)$. If $\alpha \in \Delta^+(\mathfrak{g},\fh)$ is a complex simple root, we will define a new parameter
$$s_{\alpha}[\fh,\mathfrak{b},\chi] \in \mathrm{BB}_0(G)$$
called the simple reflection of $[\fh,\mathfrak{b},\chi]$ through $\alpha$. 

If $\beta \in \Delta^+(\mathfrak{g},\fh)$ is an odd simple real root, we will define two new parameters
$$c_{\beta}^{\pm}[\fh,\mathfrak{b},\chi] \in \mathrm{BB}_0(G)$$
called the \emph{Cayley transforms} of $[\fh,\mathfrak{b},\chi]$ through $\beta$. 

The definitions are rigged so that $c_{\alpha}^{\pm}$ and $s_{\beta}$ commute (approximately) with parabolic induction $I: \mathrm{BB}_0(G) \to K^f(\mathfrak{g},\mathbf{K})$. As a result, these operations can be used in inductive arguments to relate the induced modules $I[\fh,\mathfrak{b},\chi]$ as $[\fh,\mathfrak{b},\chi]$ varies over $\mathrm{BB}_0(G)$. 
For these inductive arguments to work, we will need a numerical invariant which keeps track of how many operations have been performed. There are several good candidates for this invariant. We will use
$$d[\fh,\mathfrak{b},\chi] := \mathrm{dim}(\mathfrak{b} \cap \mathfrak{k})$$
We will see that Cayley transforms and simple reflections have a predictable effect on $d[\fh,\mathfrak{b},\chi]$.

\subsubsection{Simple Reflections Through Complex Roots}

Let $[\mathfrak{h},\mathfrak{b},\chi] \in \mathrm{BB}_0(G)$ and let $\alpha \in \Delta^+(\mathfrak{g},\mathfrak{h})$ be a simple root. 

\begin{definition}\label{def:simplereflection}
Let $s_{\alpha}\mathfrak{b} \subset \mathfrak{g}$ be the Borel subalgebra corresponding to the positive system
$$s_{\alpha}\Delta^+(\mathfrak{g},\mathfrak{h}) = \Delta^+(\mathfrak{g},\mathfrak{h}) \cup \{-\alpha\} \setminus \{\alpha\}$$
Define
$$s_{\alpha}\chi := \chi \otimes \alpha$$
Finally, let
$$s_{\alpha}[\mathfrak{h},\mathfrak{b},\chi] := [\mathfrak{h},s_{\alpha}\mathfrak{b},s_{\alpha}\chi]$$
\end{definition}

\begin{proposition}
In the setting of Definition \ref{def:simplereflection},
$$s_{\alpha}[\mathfrak{h},\mathfrak{b},\chi] \in \mathrm{BB}_0(G)$$
\end{proposition}

\begin{proof}
We need only to verify that
$$ds_{\alpha}\chi = - \rho(s_{\alpha}\mathfrak{n})$$
This follows trivially from definitions
$$ds_{\alpha}\chi = d\chi + d\delta(\alpha) = -\rho(\mathfrak{n}) + \alpha = -\rho(s_{\alpha}\mathfrak{n})$$
\end{proof}

Although $s_{\alpha}[\mathfrak{h},\mathfrak{b},\chi]$ is well-defined for any simple root, we will give special attention to the case when $\alpha$ is complex. This is due to the following

\begin{theorem}\label{thm:transfer}
Let $[\mathfrak{h},\mathfrak{b},\chi] \in \mathrm{BB}_0(G)$ and let $\alpha \in \Delta^+(\mathfrak{g},\mathfrak{h})$ be a complex simple root. Then
$$I[\mathfrak{h},\mathfrak{b},\chi] = -I(s_{\alpha}[\mathfrak{h},\mathfrak{b},\chi])$$
\end{theorem}

\begin{proof}
This is an easy consequence of the Transfer Theorem of Knapp and Vogan (See \cite{KnappVogan1995}, Theorem 11.87).
\end{proof}

Under the conditions of Theorem \ref{thm:transfer}, $s_{\alpha}$ has a predictable effect on $d[\mathfrak{h},\mathfrak{b},\chi]$.

\begin{proposition}\label{prop:simplereflectionincreasesd}
Let $[\mathfrak{h},\mathfrak{b},\chi] \in \mathrm{BB}_0(G)$ and let $\alpha \in \Delta^+(\mathfrak{g},\mathfrak{h})$ be a complex simple root.
\begin{itemize}
    \item[(i)] If 
    $$\theta(\alpha) \in -\Delta^+(\mathfrak{g},\mathfrak{h})$$
    then
    $$d(s_{\alpha}[\mathfrak{h},\mathfrak{b},\chi]) = d[\mathfrak{h},\mathfrak{b},\chi]+1$$
    \item[(ii)] If 
    $$\theta(\alpha) \in \Delta^+(\mathfrak{g},\mathfrak{h})$$
    then
    $$d(s_{\alpha}[\mathfrak{h},\mathfrak{b},\chi]) = d[\mathfrak{h},\mathfrak{b},\chi] - 1$$
\end{itemize}
\end{proposition}

\begin{proof}
Since $\mathfrak{h}$ is $\theta$-stable, we have
$$d[\mathfrak{h},\mathfrak{b},\chi] = \dim(\mathfrak{t}) + \dim(\mathfrak{n} \cap \mathfrak{k})$$
Choose a basis of root vectors $X_{\mu}$ for $\mathfrak{n}$ as in the proof of Proposition \ref{prop:cayleytransformsincreased}. Then $\mathfrak{n} \cap \mathfrak{k}$ is spanned by

\begin{align*}
&\{X_{\mu}:\mu \text{ positive compact imaginary} \} \cup \\
&\{X_{\mu} + X_{\theta\mu}: \text{pairs } \{\mu,\theta \mu\} \text{ of complex positive roots}\}
\end{align*}
Since
$$s_{\alpha}\Delta^+(\mathfrak{g},\mathfrak{h}) = \Delta^+(\mathfrak{g},\mathfrak{h}) \cup \{-\alpha\} \setminus \{\alpha\} $$
the positive systems $s_{\alpha}\Delta^+(\mathfrak{g},\mathfrak{h})$ and $\Delta^+(\mathfrak{g},\mathfrak{h})$ contain the same number of compact imaginary roots. If $\theta(\alpha) \in -\Delta^+(\mathfrak{g},\mathfrak{h})$, $s_{\alpha}\Delta^+(\mathfrak{g},\mathfrak{h})$ contains one additional pair $\{-\alpha,\theta \alpha\}$ of complex positive roots. If $\theta(\alpha) \in \Delta^+(\mathfrak{g},\mathfrak{h})$, then $\Delta^+(\mathfrak{g},\mathfrak{h})$ contains one additional pair $\{\alpha, \theta\alpha\}$ of complex positive roots. 
\end{proof}

A useful consequence of Theorem \ref{thm:transfer} is the following:

\begin{proposition}\label{prop:maketypeZtypeL}
Let $[\mathfrak{h},\mathfrak{b},\chi] \in \mathrm{BB}_0(G)$.

\begin{itemize}
    \item[(i)] There is a BB parameter $[\mathfrak{h},\mathfrak{b}_1,\chi_1] \in \mathrm{BB}_0(G)$ of type Z such that
    $$I[\mathfrak{h},\mathfrak{b},\chi] = \pm I[\mathfrak{h},\mathfrak{b}_1,\chi_1]$$
    \item[(ii)] There is a BB parameter $[\mathfrak{h},\mathfrak{b}_2,\chi_2]$ of type L such that
    $$I[\mathfrak{h},\mathfrak{b},\chi] = \pm I[\mathfrak{h},\mathfrak{b}_2,\chi_2]$$
\end{itemize}
\end{proposition}

\begin{proof}
Let $S \subset \mathrm{BB}_0(G)$ be the set of all BB parameters which can be obtained from $[\mathfrak{h},\mathfrak{b},\chi]$ through a finite sequence of simple reflections through complex simple roots $\alpha$ satisfying $\theta(\alpha) \in -\Delta^+(\mathfrak{g},\mathfrak{h})$. Since $S$ is finite, there is an element $(\mathfrak{h},\mathfrak{b}_1,\chi_1) \in S$ which maximizes $d$. If $\alpha \in \Delta^+(\mathfrak{g},\mathfrak{h})$ is a complex simple root satisfying $\theta(\alpha) \in -\Delta^+(\mathfrak{g},\mathfrak{h})$, then $s_{\alpha}[\mathfrak{h},\mathfrak{b}_1,\chi_1] \in S$ and by Proposition \ref{prop:simplereflectionincreasesd} $d(s_{\alpha}[\mathfrak{h},\mathfrak{b}_1,\chi_1]) > d[\mathfrak{h},\mathfrak{b}_1,\chi_1]$, a contradiction. Hence, $[\mathfrak{h},\mathfrak{b}_1,\chi_1]$ is type Z. The equation
$$I[\mathfrak{h},\mathfrak{b},\chi] = \pm I[\mathfrak{h},\mathfrak{b}_1,\chi_1]$$
follows by induction from Theorem \ref{thm:transfer}. This proves (i).

The proof of (ii) is analogous. Let $S' \subset \mathrm{BB}_0(G)$ be the set of all BB parameters which can be obtained from $[\mathfrak{h},\mathfrak{b},\chi]$ through a sequence of simple reflections through complex simple roots $\alpha$ satisfying $\theta(\alpha) \in \Delta^+(\mathfrak{g},\mathfrak{h})$. Since $d \geq 0$, there is an element $[\mathfrak{h},\mathfrak{b}_2,\chi_2] \in S$ which minimizes $d$. If $\alpha \in \Delta^+(\mathfrak{g},\mathfrak{h})$ is a complex simple root satisfying $\theta(\alpha) \in \Delta^+(\mathfrak{g},\mathfrak{h})$, then $s_{\alpha}[\mathfrak{h},\mathfrak{b}_2,\chi_2] \in S'$ and by Proposition \ref{prop:simplereflectionincreasesd} $d(s_{\alpha}[\mathfrak{h},\mathfrak{b},\chi]) < d[\mathfrak{h},\mathfrak{b}_2,\chi_2]$, a contradiction. Hence, $[\mathfrak{h},\mathfrak{b}_2,\chi_2]$ is type L. The equation
$$I[\mathfrak{h},\mathfrak{b},\chi] = \pm I[\mathfrak{h},\mathfrak{b}_2,\chi_2]$$
follows by induction from Theorem \ref{thm:transfer}. 
\end{proof}

\subsubsection{Cayley Transforms Through Real Roots}

Let $[\mathfrak{h},\mathfrak{b},\chi] \in \mathrm{BB}_0(G)$ and let $\alpha \in \Delta(\mathfrak{g},H)$ be an odd real root. Recall the Cartan subgroup $H^{\alpha}$, the inner automorphisms $c_{\alpha}^{\pm}$ of $\mathfrak{g}$, and the characters $c_{\alpha}^{\pm}\chi$ of $H^{\alpha}$ defined in Section \ref{sec:Cayleyreal}. 
\begin{definition}\label{def:cayleytransformofparameter}
In the setting described above, let
$$c_{\alpha}^{\pm}[\mathfrak{h},\mathfrak{b},\chi] := [\mathfrak{h}^{\alpha},c_{\alpha}^{\pm}\mathfrak{b},c_{\alpha}^{\pm}\chi]$$
\end{definition}
Note that the pair $c_{\alpha}^{\pm}[\mathfrak{h},\mathfrak{b},\chi]$ is independent of the isomorphism $\phi_{\alpha}: \mathfrak{sl}_2(\mathbb{C}) \to \mathfrak{s}_{\alpha}$ used to define it (see the remarks preceding Proposition \ref{prop:calphacartan}). 

\begin{proposition}
In the setting of Definition \ref{def:cayleytransformofparameter},
$$c_{\alpha}^{\pm}[\mathfrak{h},\mathfrak{b},\chi] \in \mathrm{BB}_0(G)$$
\end{proposition}

\begin{proof}
The only condition to check is
$$dc_{\alpha}^{\pm}\chi = - \rho(c_{\alpha}^{\pm}\mathfrak{n})$$
This will follow from the condition
$$d\chi = - \rho(\mathfrak{n})$$
if we can prove that $dc_{\alpha}^{\pm}\chi = d\chi \circ (c_{\alpha}^{\pm})^{-1}$
We will check this equality independently on $\ker{\alpha}$ and $\phi(D_c)$ (which together span $\mathfrak{h}^{\alpha}$). By definition, $dc_{\alpha}^{\pm}\chi|_{\ker{\alpha}} = d\chi|_{\ker{\alpha}}$ and on $\ker{\alpha}$, both $c_{\alpha}^{\pm}$ act by the identity (see Proposition \ref{prop:calphacartan}). On $\phi_{\alpha}(D_c)$, we use Proposition \ref{prop:calphacartan} again to compute
$$d\chi (c_{\alpha}^{\pm})^{-1}\phi_{\alpha}(D_c) = \pm d\chi(\alpha^{\vee}) = \pm \langle \rho(\mathfrak{n}),\alpha^{\vee}\rangle = \pm 1$$
and indeed
$$dc_{\alpha}^{\pm}\chi(\phi(D_c)) = \tau_{\pm 1}(D_c) = \pm 1$$
\end{proof}

Make the following 

\begin{definition}[\cite{Vogan1981}, Def 8.3.4]\label{def:type1type2}
Suppose $\alpha \in \Delta(\mathfrak{g},H)$ is a real root. The image of $T$ under $\alpha$ is a compact subgroup of $\mathbb{R}^{\times}$. Hence, $\alpha$ restricts to a group homomorphism
$$\alpha: T \to \{\pm 1\}$$
We say that $\alpha$ is type 1 (resp. type 2) if the image of this map is $\{1\}$ (resp. $\{\pm 1\}$). 
\end{definition}

The analogue of Theorem \ref{thm:transfer} for Cayley transforms is the following.

\begin{theorem}\label{thm:inductionandCayleytransforms}
Let $[\mathfrak{h},\mathfrak{b},\chi] \in \mathrm{BB}_0(G)$ and let $\alpha \in \Delta^+(\mathfrak{g},H)$ be an odd simple real root. Then
\begin{itemize}
    \item[(i)] If $\alpha$ is type 1,
    $$I[\mathfrak{h},\mathfrak{b},\chi] = -I(c_{\alpha}^+[\mathfrak{h},\mathfrak{b},\chi]) - I(c_{\alpha}^-[\mathfrak{h},\mathfrak{b},\chi])$$
    \item[(ii)] If $\alpha$ is type 2,
    $$I[\mathfrak{h},\mathfrak{b},\chi] = -I(c_{\alpha}^+[\mathfrak{h},\mathfrak{b},\chi]) = -I(c_{\alpha}^-[\mathfrak{h},\mathfrak{b},\chi])$$
\end{itemize}
\end{theorem}

Our proof of Theorem \ref{thm:inductionandCayleytransforms} will involve reduction to $SL_2(\mathbb{R})$ and, in the type 2 case, the basic Clifford theory of index $2$ subpairs. To simplify notation, let 
$$G^s = SL_2(\mathbb{R}) \quad K^s=SO_2(\mathbb{R}) \quad H^s = \{\begin{pmatrix} t & 0 \\ 0 & t^{-1}\end{pmatrix}\} \quad T^s = \{\pm \mathrm{Id}\}$$
Let $\epsilon \otimes -1$ be the character of $H^s$ defined by
$$(\epsilon \otimes -1)(\begin{pmatrix} t & 0 \\ 0 & t^{-1}\end{pmatrix} = t^{-1}$$
Let $\tau_{\pm 1}$ be the characters of $K^s$ defined in (\ref{eqn:twocharactersofSO2R}). Let $\mathfrak{b}^s \subset \mathfrak{g}^s$ be the Borel subalgebra of upper triangular matrices, and let $\mathfrak{b}^s_c(\pm) \subset \mathfrak{g}^s$ be the Borel subalgebras containing $\mathfrak{k}^s$. Arrange the signs so that
$$d\tau_1 = -\rho(\mathfrak{n}^s_c(+)) \quad d\tau_{-1} = -\rho(\mathfrak{n}^s_c(-))$$
We will need the following basic fact about $SL_2(\mathbb{R})$:

\begin{proposition}\label{prop:SL2Rdecomposition}
There is an equality in $KM(\mathfrak{g}^s,K^s)$
$$I[\mathfrak{h}^s,\mathfrak{b}^s, \epsilon \otimes -1] = -I[\mathfrak{k}^s,\mathfrak{b}^s_c(+),\tau_1] - I[\mathfrak{k}^s,\mathfrak{b}^s_c(-),\tau_{-1}]$$
The classes on the right correspond to irreducible $(\mathfrak{g}^s,K^s)$-modules. 
\end{proposition}

\begin{proof}
There is a well-known decomposition of the non-spherical principal series representation
$$\mathrm{Ind}^{G^s}_{B^S} \epsilon \otimes 0$$
into the two limit of discrete series representations, which are obtained by cohomological induction (in degree $1$) from the $\theta$-stable Borel subalgebras $\mathfrak{b}_c^s(\pm)$
$$\mathrm{Ind}^{G^s}_{B^s}\epsilon \otimes 0 \cong_{(\mathfrak{g}^s,K^s)} R^1\mathbf{I}^{(\mathfrak{g}^s,K^s)}_{(\mathfrak{b}^s_c(+),K^s)} \tau_1 \oplus R^1\mathbf{I}^{(\mathfrak{g}^s,K^s)}_{(\mathfrak{b}^s_c(-),K^s)} \tau_{-1}$$
Now use Theorems \ref{thm:realinductionexact}, \ref{thm:cohindirreducible}, and
the identification (\ref{eqn:realinductioninfinitesimal}).
\end{proof}

Now let $[\mathfrak{h},\mathfrak{b},\chi] \in \mathrm{BB}_0(G)$ and let $\alpha \in \Delta^+(\mathfrak{g},H)$ be any real root. Define a $\theta$-stable Levi subgroup
$$L_{\alpha} := Z_G(\ker{\alpha}) $$
Since $\Phi_{\alpha}(G^s)$ centralizes $\ker{\alpha}$, there is a group homomorphism
$$r_{\alpha}: G^s \times {\ker{\alpha}} \to L_{\alpha} \quad r_{\alpha}(g,h) = \Phi_{\alpha}(g)h$$
We will also consider its restrictions
$$r_{\alpha}: H^s \times \ker{\alpha} \to H \quad r_{\alpha}: K^s \times \ker{\alpha} \to H^{\alpha}$$
\begin{proposition}\label{prop:structuretheorytype1type2}
The group homomorphisms
\begin{align}
r_{\alpha}: G^s \times \ker{\alpha} &\to L_{\alpha} \label{eqn:hom1} \\
r_{\alpha}: H^s \times \ker{\alpha} &\to H \label{eqn:hom2} \\
r_{\alpha}: K^s \times \ker{\alpha} &\to H^{\alpha} \label{eqn:hom3}
\end{align}
have the following properties
\begin{itemize}
    \item[(i)] All three homomorphisms are isogenies (i.e. give rise to Lie algebra isomorphisms)
    \item[(ii)] All three homomorphisms are two-to-one, with 
    $$\ker{r_{\alpha}} = \{(1,1),(-1,m_{\alpha})\}$$
    \item[(iii)] Homomorphism \ref{eqn:hom3} is surjective (independent of $\alpha$). Homomorphisms \ref{eqn:hom1} and \ref{eqn:hom2} are either surjective (if $\alpha$ is type 1) or surjective onto index-2 subgroups (if $\alpha$ is type 2).
\end{itemize}
\end{proposition}

\begin{proof}

The differential of $r_{\alpha}$ is given by
    $$dr_{\alpha}: \mathfrak{g}^s \oplus \ker{\alpha} \to Z_{\mathfrak{g}}(\ker{\alpha}) \quad dr_{\alpha}(X,H) = \phi_{\alpha}(X) + H$$
    The map $\phi_{\alpha}: \mathfrak{g}^s \to \mathfrak{g}$ is an injection, with image $\mathbb{C}\phi_{\alpha}(H) \oplus \mathfrak{g}_{\alpha} \oplus \mathfrak{g}_{-\alpha}$. In particular, $\phi_{\alpha}(\mathfrak{g}^s) \cap \ker{\alpha} = 0$ and $dr_{\alpha}$ is injective. To conclude that homomorphisms \ref{eqn:hom1}, \ref{eqn:hom2}, and \ref{eqn:hom3} are isogenies, note that
    \begin{align*}
        dr_{\alpha}(\mathfrak{g}^s \oplus \ker{\alpha}) &= \mathbb{C}\phi_{\alpha}(H) \oplus \mathfrak{g}_{\alpha} \oplus \mathfrak{g}_{-\alpha} \oplus \ker{\alpha} = \mathfrak{h} \oplus \mathfrak{g}_{\alpha} \oplus \mathfrak{g}_{-\alpha} =
        Z_{\mathfrak{g}}(\ker{\alpha})\\
        dr_{\alpha}(\mathfrak{h}^s \oplus \ker{\alpha}) &= \mathbb{C}\phi_{\alpha}(H) \oplus \ker{\alpha} = \mathfrak{h}\\
        dr_{\alpha}(\mathfrak{k}^s \oplus \ker{\alpha}) &= \mathbb{C}\phi_{\alpha}(H_c) \oplus \ker{\alpha} = \mathfrak{h}^{\alpha}
    \end{align*}
This completes the proof of (i).

Since $dr_{\alpha}$ is injective, $\ker{r_{\alpha}}$ is a discrete, normal subgroup of $G^s \times \ker{\alpha}$. Let $\pi_1: G^s \times \ker{\alpha} \to G^s$ be the projection map. Then $\pi_1(\ker{r_{\alpha}})$ is a discrete, normal subgroup of $G^s$, and hence a subgroup of $Z(G^s)=\{\pm 1\}$. On the other hand, the restriction of $\pi_1$ to $\ker{r_{\alpha}}$ is injective. Hence, $\pi_1: \ker{r_{\alpha}} \subseteq \{\pm 1\}$. Note finally that $r_{\alpha}(-1,m_{\alpha}) = \Phi_{\alpha}(-1)m_{\alpha} = m_{\alpha}^2 =1$. So indeed, $\ker{r_{\alpha}} = \{(1,1),(-1,m_{\alpha})\}$, proving (ii).

The surjectivity statement for homomorphism \ref{eqn:hom3} follows from Lemma \ref{lem:structureofTalpha}. We will prove the statement for homomorphism \ref{eqn:hom2}. The statement for homomorphism \ref{eqn:hom1} will then follow from the Bruhat decomposition for $L_{\alpha}$. The restriction of $\Phi_{\alpha}$ to $H^s$ coincides with the co-root $\alpha^{\vee}: \mathbb{R}^{\times} \to H$. Since $\alpha(\alpha^{\vee}(t)) = t^2 >0$, there is an inclusion
    $$r_{\alpha}(H^s \times \ker{\alpha}) \subseteq \alpha^{-1}(\mathbb{R}_{>0})$$
    The reverse inclusion is equally clear: if $h \in H$ has $\alpha(h) >0$, then
    $$\alpha^{\vee}(\sqrt{\alpha(h)})^{-1}h \in \ker{\alpha}$$
    and hence
    $$h = \alpha^{\vee}(\sqrt{\alpha(h)})\left(\alpha^{\vee}(\sqrt{\alpha(h)})^{-1}h\right) \in \alpha^{\vee}(\mathbb{R}^{\times})\ker{\alpha} = r_{\alpha}(H^s \times \ker{\alpha})$$
    Combining these facts, we obtain
    $$r_{\alpha}(H^s \times \ker{\alpha}) = \alpha^{-1}(\mathbb{R}_{>0})$$
    Now, $\alpha(H) \subseteq \mathbb{R}^{\times}$ is a finite-index subgroup. There are two such subgroups of $\mathbb{R}^{\times}$: $\mathbb{R}^{\times}$ and $\mathbb{R}_{>0}$. If $\alpha$ is type 1, then $\alpha(H) = \alpha(TA) = \alpha(A)$, which is connected, and therefore necessarily $\mathbb{R}_{>0}$. In this case $r_{\alpha}(H^s \times \ker{\alpha}) = H$. If $\alpha$ is type 2, then $\alpha(H)$ contains $-1$ and is therefore the full multiplicative group $\mathbb{R}^{\times}$. In this case, $r_{\alpha}(H^s \times \ker{\alpha})$ has index 2 in $H$. 
    
    Since $r_{\alpha}(G^s \times \ker{\alpha}) \cap H = r_{\alpha}(H^s \times \ker{\alpha})$, the inclusion $H \subset L_{\alpha}$ induces an injection of cosets
    \begin{equation}\label{eqn:cosetinclusion}H/ r_{\alpha}(H^s \times \ker{\alpha}) \subseteq L_{\alpha}/r_{\alpha}(G^s \times \ker{\alpha})\end{equation}
    We want to show that this mapping is onto. First note that the element 
    $$\sigma_{\alpha}:=\Phi_{\alpha}(E-F)$$
    normalizes $\mathfrak{h}$ nontrivially. Hence, $\sigma_{\alpha}$ represents the nontrivial element of the real Weyl group $W(L_{\alpha},H) \cong \mathbb{Z}/2\mathbb{Z}$. Let $U_{\alpha} = \exp(\mathfrak{g}_{\alpha} \cap \mathfrak{g}_0)$. By the Bruhat decomposition for real reductive groups (see \cite{Knapp1996}, Theorem 7.40),
    $$L_{\alpha} = HU_{\alpha} \sqcup HU_{\alpha}\sigma_{\alpha}U_{\alpha}$$
    Since $U_{\alpha}$ is connected and $\mathfrak{u}_{\alpha} \subset \mathfrak{l}_{\alpha}$, $U_{\alpha} \subseteq r_{\alpha}(G^s \times \ker{\alpha})$. Also, $\sigma_{\alpha} \in r_{\alpha}(G^s \times \ker{\alpha})$ by definition. It follows from these observations and the decomposition above that every coset in $ L_{\alpha}/r_{\alpha}(G^s \times \ker{\alpha})$ has a representative in $H$ and hence the injection \ref{eqn:cosetinclusion} is onto. 
    
    The conclusion is that $[L_{\alpha}:r_{\alpha}(G^s \times \ker{\alpha})] = [H: r_{\alpha}(H^s \times \ker{\alpha})]$. Hence, the statement for homomorphism \ref{eqn:hom2} implies the statement for homomorphism \ref{eqn:hom1}. This completes the proof of (iii).
\end{proof}

Form three Borel subalgebras of $\mathfrak{l}_{\alpha}$
$$\mathfrak{b}^{\mathfrak{l}_{\alpha}}:= \mathfrak{b} \cap \mathfrak{l}_{\alpha} \supset \mathfrak{h} \quad \mathfrak{b}_c^{\mathfrak{l}_{\alpha}}(\pm) = c_{\alpha}^{\pm}\mathfrak{b} \cap \mathfrak{l}_{\alpha}\supset c_{\alpha}\mathfrak{h}$$
The conjugacy of $\mathfrak{b}^{\mathfrak{l}_{\alpha}}_c(\pm)$ under $L_{\alpha}$ depends on the type of $\alpha$.

\begin{lemma}\label{lem:conjugacyofcompactborels}
If $\alpha$ is type 1, then $\mathfrak{b}^{\mathfrak{l}_{\alpha}}_c(\pm)$ are non-conjugate under $L_{\alpha}$. If $\alpha$ is type 2, then for any element $t \in T$ with $\alpha(t) = -1$, $\Ad(t)$ acts by inversion on $\Phi_{\alpha}(K^s)$ and interchanges $\mathfrak{b}^{\mathfrak{l}_{\alpha}}_c(\pm)$.
\end{lemma}

\begin{proof}
Suppose $\alpha$ is type 1 and assume there is a group element $g \in L_{\alpha}$ such that
$$\Ad(g)\mathfrak{b}^{\mathfrak{l}_{\alpha}}_c(+) = \mathfrak{b}^{\mathfrak{l}_{\alpha}}_c(-)$$
By Proposition \ref{prop:structuretheorytype1type2}
\begin{equation}\label{eqn:conjborels}
\Phi_{\alpha}(G^s)\ker{\alpha}=L_{\alpha}\end{equation}
Write $g = \Phi_{\alpha}(g')h$ for elements $g' \in G^s$ and $h \in \ker{\alpha}$. Since $\ker{\alpha}$ is central in $L_{\alpha}$, we can replace $g$ with $\Phi_{\alpha}(g')$ in Equation \ref{eqn:conjborels} above. We deduce that
$$\Ad(g')(\mathfrak{b}^{\mathfrak{l}_{\alpha}}_c(+) \cap \mathfrak{g}^s) = \mathfrak{b}^{\mathfrak{l}_{\alpha}}_c(-) \cap \mathfrak{g}^s$$
The Borels appearing above are exactly $\mathfrak{b}_c^s(\pm)$. These are non-conjugate under $G^s$ by an explicit calculation. We deduce that $\mathfrak{b}^{\mathfrak{l}_{\alpha}}_c(\pm)$ are non-conjugate under $L_{\alpha}$.

Now suppose $\alpha$ is type 2. Choose $t \in T$ with $\alpha(t) =-1$. Then
\begin{align*}
&\mathrm{Ad}(t)\phi_{\alpha}(D) = \phi_{\alpha}(D)\\  &\mathrm{Ad}(t)\phi_{\alpha}(E) = \alpha(t)\phi_{\alpha}(E)= -\phi_{\alpha}(E)\\ &\mathrm{Ad}(t)\phi_{\alpha}(F) = \alpha(t)^{-1}\phi_{\alpha}(F) = -\phi_{\alpha}(F)
\end{align*}
Therefore,
$$\mathrm{Ad}(t)\phi_{\alpha}(E-F) = -\phi_{\alpha}(E-F)$$
Therefore, since $E-F$ spans $\mathfrak{k}^s$, $\mathrm{Ad}(t)$ acts by negation on $\phi_{\alpha}(\mathfrak{k}^s)$ and consequently, since $\Phi_{\alpha}(K^s)$ is connected, by inversion on $\Phi_{\alpha}(K^s)$.

By definition, $\mathfrak{h}^{\alpha} = \phi_{\alpha}(\mathfrak{k}^s) \oplus \ker{\alpha}$. $\Ad(t)$ preserves this Cartan subalgebra: it normalizes the first factor by the computation above and centralizes the second factor since $H$ is abelian. A Lie algebra automorphism of $\mathfrak{l}_{\alpha}$ which preserves $\mathfrak{h}_{\alpha}$ permutes the Borel subalgebras containing it. Hence, $\Ad(t)$ permutes $\mathfrak{b}^{\mathfrak{l}_{\alpha}}_c(\pm)$. Since $\Ad(t)$ acts nontrivially on the co-root $c_{\alpha}\alpha^{\vee} \in \mathfrak{h}^{\alpha}$, it acts by the nontrivial permutation. 
\end{proof}

\begin{proposition}\label{prop:decompforlevi}
Let $[\mathfrak{h},\mathfrak{b},\chi] \in \mathrm{BB}_0(G)$ and let $\alpha \in \Delta^+(\mathfrak{g},H)$ be an odd real simple root. Then

\begin{itemize}
\item[(i)] If $\alpha$ is type 1, there is an equality in $K^f(\mathfrak{l}_{\alpha},L_{\alpha}\cap K)$
$$I[\mathfrak{h},\mathfrak{b}^{\mathfrak{l}_{\alpha}},\chi] = - I[\mathfrak{h}^{\alpha},\mathfrak{b}_c^{\mathfrak{l}_{\alpha}}(+), c_{\alpha}^+\chi] -  I[\mathfrak{h}^{\alpha},\mathfrak{b}_c^{\mathfrak{l}_{\alpha}}(-), c_{\alpha}^-\chi] $$
and the terms on the right are irreducible.
\item[(ii)] If $\alpha$ is type 2, there are equalities in $K^f(\mathfrak{l}_{\alpha},L_{\alpha}\cap K)$
$$I[\mathfrak{h},\mathfrak{b}^{\mathfrak{l}_{\alpha}},\chi] = - I[\mathfrak{h}^{\alpha},\mathfrak{b}_c^{\mathfrak{l}_{\alpha}}(+), c_{\alpha}^+\chi] = -  I[\mathfrak{h}^{\alpha},\mathfrak{b}_c^{\mathfrak{l}_{\alpha}}(-), c_{\alpha}^-\chi] $$
and all terms are irreducible.
\end{itemize}
\end{proposition}

\begin{proof}
First, assume $\alpha$ is type 1. Consider the character $r_{\alpha}^*\chi$ of $H^s \times \ker{\alpha}$ obtained by pulling back $\chi$ along the surjective homomorphism
$$r_{\alpha}: H^s \times \ker{\alpha} \to H$$
This character has the form
$$r_{\alpha}^*\chi = \tau \otimes \chi|_{\ker{\alpha}}$$
for some character $\tau$ of $H^s$. Since $\alpha$ is odd, $\tau(-1)=-1$, and since $d\chi = -\rho(\mathfrak{n})$, 
$$d\tau(H) = d\chi (\alpha^{\vee}) = -\rho(\mathfrak{n})(\alpha^{\vee})= -1$$
Hence, $\tau = \epsilon \otimes -1$. Applying Proposition \ref{prop:SL2Rdecomposition}, we obtain an equality in $K^f(\mathfrak{g}^s \oplus \ker{\alpha}, K^S \times (\ker{\alpha} \cap T))$
\begin{align}\label{eqn:reductiontosl2}
I[\mathfrak{h}^s \times \ker{\alpha}, \mathfrak{b}^s \oplus \ker{\alpha},r_{\alpha}^*\chi] &= I[\mathfrak{h}^s,\mathfrak{b}^s,\epsilon \otimes -1] \otimes \chi|_{\ker{\alpha}}\\
&= -\left(I[\mathfrak{k}^s,\mathfrak{b}_c^s(+),\tau_1] + I[\mathfrak{k}^s,\mathfrak{b}_c^s(-),\tau_{-1}] \right) \otimes \chi|_{\ker{\alpha}} \nonumber \\
&= -I[\mathfrak{k}^s \times \ker{\alpha}, \mathfrak{b}^s_c(+) \oplus \ker{\alpha},\tau_1\otimes \chi] - I[\mathfrak{k}^s \times \ker{\alpha}, \mathfrak{b}^s_c(-) \oplus \ker{\alpha},\tau_{-1}\otimes \chi] \nonumber
\end{align}
and the terms on the right are irreducible classes by the second half of the same proposition. By the definitions of $c_{\alpha}^{\pm}\chi$ and $r_{\alpha}$, we have $r_{\alpha}^*c_{\alpha}^{\pm}\chi = \tau_{\pm 1} \otimes \chi$. Substituting these identities into \ref{eqn:reductiontosl2}, we get
$$I[\mathfrak{h}^s \times \ker{\alpha},\mathfrak{b}^s \oplus \ker{\alpha},r_{\alpha}^*\chi] = - I[\mathfrak{k}^s \times \ker{\alpha},\mathfrak{b}_c^s(+) \oplus \ker{\alpha}, r_{\alpha}^*c_{\alpha}^+\chi] - I[\mathfrak{k}^s \times \ker{\alpha},\mathfrak{b}_c^s(-) \oplus \ker{\alpha}, r_{\alpha}^*c_{\alpha}^-\chi] $$
Since the homomorphisms of Proposition \ref{prop:structuretheorytype1type2} are surjective, we can move $r_{\alpha}^*$ past $I$, thus obtaining
$$r_{\alpha}^*I[\mathfrak{h},\mathfrak{b}^{\mathfrak{l}_{\alpha}},\chi] = - r_{\alpha}^*I[\mathfrak{h}^{\alpha},\mathfrak{b}_c^{\mathfrak{l}_{\alpha}}(+), c_{\alpha}^+\chi] - r_{\alpha}^*I[\mathfrak{h}^{\alpha},\mathfrak{b}_c^{\mathfrak{l}_{\alpha}}(-), c_{\alpha}^-\chi] $$
which forces an equality in $K^f(\mathfrak{l}_{\alpha},L_{\alpha} \cap K)$
$$I[\mathfrak{h},\mathfrak{b}^{\mathfrak{l}_{\alpha}},\chi] = - I[\mathfrak{h}^{\alpha},\mathfrak{b}_c^{\mathfrak{l}_{\alpha}}(+), c_{\alpha}^+\chi] - I[\mathfrak{h}^{\alpha},\mathfrak{b}_c^{\mathfrak{l}_{\alpha}}(-), c_{\alpha}^-\chi] $$
as desired.

Now suppose $\alpha$ is type 2. Define the subgroups
$$H' := r_{\alpha}(H^s \times \ker{\alpha}) \subset H \quad L_{\alpha}' := r_{\alpha}(G^s \times \ker{\alpha}) \subset L_{\alpha}$$
By Proposition \ref{prop:structuretheorytype1type2}, these are index-2 subgroups. By the argument provided above, there is an equality in $K^f(\mathfrak{l}_{\alpha}, L_{\alpha}' \cap K)$
\begin{equation}\label{eqn:imagedecomp}
    I[\mathfrak{h}', \mathfrak{b}^{\mathfrak{l}_{\alpha}}, \chi] = -I[\mathfrak{h}^{\alpha},\mathfrak{b}_c^{\mathfrak{l}_{\alpha}}(+),c_{\alpha}^+\chi] - I[\mathfrak{h}^{\alpha},\mathfrak{b}_c^{\mathfrak{l}_{\alpha}}(-),c_{\alpha}^-\chi]
\end{equation}
and the terms on the right are irreducible classes.

By Proposition \ref{prop:cliffordtheoryindex2}, 
\begin{equation}\label{eqn:cliffordapp1}
I^{(\mathfrak{h},T)}_{(\mathfrak{h},H'\cap K)}\chi \cong \chi \oplus (\chi \otimes \epsilon)
 \end{equation}
where $\epsilon$ (as in the statement of Proposition \ref{prop:cliffordtheoryindex2}) is the unique nontrivial $(\mathfrak{h},T)$-module with trivial restriction to $(\mathfrak{h},H'\cap K)$. Substituting \ref{eqn:cliffordapp1} into \ref{eqn:imagedecomp}, we get an equality (still in $K^f(\mathfrak{l}_{\alpha},L_{\alpha}'\cap K)$)
\begin{equation}\label{eqn:imagedecomp1}
I[\mathfrak{h},\mathfrak{b}^{\mathfrak{l}_{\alpha}},\chi] + I[\mathfrak{h}, \mathfrak{b}^{\mathfrak{l}_{\alpha}},\chi \otimes \epsilon] = - I[\mathfrak{h}^{\alpha},\mathfrak{b}^{\mathfrak{l}_{\alpha}}_c(+),c_{\alpha}^+\chi] - I[\mathfrak{h}^{\alpha},\mathfrak{b}^{\mathfrak{l}_{\alpha}}_c(-),c_{\alpha}^-\chi]
\end{equation}
To deduce the desired equalities in $K^f(\mathfrak{l}_{\alpha},L_{\alpha} \cap K)$, we must apply Proposition \ref{prop:cliffordtheoryindex2} once more, this time a little less trivially. By Lemma \ref{lem:conjugacyofcompactborels}, the two-element quotient group $(L_{\alpha} \cap K)/(L_{\alpha}' \cap K)$ exchanges the summands appearing on the right hand side of \ref{eqn:imagedecomp1}. Therefore by Proposition \ref{prop:cliffordtheoryindex2}, the classes
$$I[\mathfrak{h}^{\alpha},\mathfrak{b}_c^{\mathfrak{l}_{\alpha}}(\pm), c_{\alpha}^{\pm}\chi] \subset K^f(\mathfrak{l}_{\alpha},L_{\alpha}\cap K)$$
are isomorphic and irreducible. If we apply the (exact) functor $I^{(\mathfrak{l}_{\alpha},L_{\alpha}\cap K)}_{(\mathfrak{l}_{\alpha},L_{\alpha}'\cap K)}$ to both sides of \ref{eqn:imagedecomp1}, we obtain an equality in $K^f(\mathfrak{l}_{\alpha},L_{\alpha}\cap K)$
$$
I[\mathfrak{h},\mathfrak{b}^{\mathfrak{l}_{\alpha}},\chi] + I[\mathfrak{h},\mathfrak{b}^{\mathfrak{l}_{\alpha}},\chi \otimes \epsilon] = -I[\mathfrak{h}^{\alpha},\mathfrak{b}_c^{\mathfrak{l}_{\alpha}}(+), c_{\alpha}^+\chi] - I[\mathfrak{h}^{\alpha},\mathfrak{b}_c^{\mathfrak{l}_{\alpha}}(-), c_{\alpha}^-\chi]
$$
Since all terms are irreducible, this implies
$$I[\mathfrak{h},\mathfrak{b}^{\mathfrak{l}_{\alpha}},\chi] = I[\mathfrak{h},\mathfrak{b}^{\mathfrak{l}_{\alpha}},\chi \otimes \epsilon] = -I[\mathfrak{h}^{\alpha},\mathfrak{b}_c^{\mathfrak{l}_{\alpha}}(+), c_{\alpha}^+\chi] = - I[\mathfrak{h}^{\alpha},\mathfrak{b}_c^{\mathfrak{l}_{\alpha}}(-), c_{\alpha}^-\chi]$$
which proves part (2) of the proposition.
\end{proof}

We are now prepared to prove Theorem \ref{thm:inductionandCayleytransforms}.

\begin{proof}[Proof of Theorem \ref{thm:inductionandCayleytransforms}]

Recall the parabolic subalgebra $\mathfrak{q}^Z = \mathfrak{l}^Z \oplus \mathfrak{u}^Z$ defined in (\ref{eqn:definitionofqZ}). By assumption, $\alpha$ is an odd simple root for the positive system $\Delta^+(\mathfrak{l}^Z,H)$. Let $\mathfrak{p}_{\alpha} = \mathfrak{l}_{\alpha} \oplus \mathfrak{u}_{\alpha} \subset \mathfrak{l}$ be the corresponding  minimal parabolic. Then by definition
$$\mathfrak{b} = \mathfrak{b}^{\mathfrak{l}_{\alpha}} \oplus \mathfrak{u}_{\alpha} \oplus \mathfrak{u}^Z \quad c_{\alpha}^{\pm}\mathfrak{b} = \mathfrak{b}_c^{\mathfrak{l}_{\alpha}}(\pm) \oplus \mathfrak{u}_{\alpha} \oplus \mathfrak{u}^Z$$
Suppose $\alpha$ is type 1. Using Proposition \ref{prop:decompforlevi} and Proposition \ref{prop:indbystages}, we obtain an equality
\begin{align*}
    I[\mathfrak{h},\mathfrak{b},\chi] &= I[\mathfrak{l}^Z,\mathfrak{q}^Z,I[L_{\alpha},\mathfrak{p}_{\alpha}, I[\mathfrak{h},\mathfrak{b}^{\mathfrak{l}_{\alpha}},\chi]]])\\
    &= -I[\mathfrak{l}^Z,\mathfrak{q}^Z,I[\mathfrak{l}_{\alpha},\mathfrak{p}_{\alpha}, I[\mathfrak{h}^{\alpha},\mathfrak{b}_c^{\mathfrak{l}_{\alpha}}(+),c_{\alpha}^+\chi]]] - I[\mathfrak{l}^Z,\mathfrak{q}^Z,I[\mathfrak{l}_{\alpha},\mathfrak{p}_{\alpha}, I[\mathfrak{h}^{\alpha},\mathfrak{b}_c^{\mathfrak{l}_{\alpha}}(-),c_{\alpha}^-\chi]]]\\
    &= - I[\mathfrak{h}^{\alpha},c_{\alpha}^+\mathfrak{b},c_{\alpha}^+\chi] - I[\mathfrak{h}^{\alpha},c_{\alpha}^-\mathfrak{b},c_{\alpha}^-\chi]\\
    &= - I(c_{\alpha}^+[\mathfrak{h},\mathfrak{b},\chi]) - I(c_{\alpha}^-[\mathfrak{h},\mathfrak{b},\chi])
\end{align*}
If $\alpha$ is type 2, we obtain
\begin{align*}
    I[\mathfrak{h},\mathfrak{b},\chi] &= I[\mathfrak{l}^Z,\mathfrak{q}^Z,I[\mathfrak{l}_{\alpha},\mathfrak{p}_{\alpha}, I[\mathfrak{h},\mathfrak{b}^{\mathfrak{l}_{\alpha}},\chi]]]\\
    &= -I[\mathfrak{l}^Z,\mathfrak{q}^Z,I[\mathfrak{l}_{\alpha},\mathfrak{p}_{\alpha}, I[\mathfrak{h}^{\alpha},\mathfrak{b}_c^{\mathfrak{l}_{\alpha}}(\pm),c_{\alpha}^{\pm}\chi]]
    \\
    &= - I[\mathfrak{h}^{\alpha},c_{\alpha}^{\pm}\mathfrak{b},c_{\alpha}^{\pm}\chi]\\
    &= - I(c_{\alpha}^{\pm}[\mathfrak{h},\mathfrak{b},\chi])
\end{align*}
\end{proof}

As promised, the operations $c_{\alpha}^{\pm}$ have a predictable effect on $d[\mathfrak{h},\mathfrak{b},\chi]$. 

\begin{proposition}\label{prop:cayleytransformsincreased}
Let $[\mathfrak{h},\mathfrak{b},\chi] \in \mathrm{BB}_0(G)$ and let $\alpha \in \Delta^+(\mathfrak{g},H)$ be an odd simple real root. Then
$$d(c_{\alpha}^{\pm}[\mathfrak{h},\mathfrak{b},\chi]) = d[\mathfrak{h},\mathfrak{b},\chi] + 1 $$
\end{proposition}

\begin{proof}
Choose a basis of root vectors $X_{\mu}$ in $\mathfrak{n}$ such that 
$$\theta X_{\mu} = X_{\theta \mu}$$
for every pair $(\mu, \theta \mu)$ of positive complex roots. Then the subspace $\mathfrak{n} \cap \mathfrak{k}$ is spanned by the elements
\begin{align*}
&\{X_{\mu}:\mu \text{ positive compact imaginary} \} \cup \\
&\{X_{\mu} + X_{\theta\mu}: \text{pairs } (\mu,\theta \mu) \text{ of positive complex roots}\}
\end{align*}
If we define
$$\mathfrak{n}':=\bigoplus_{\alpha \neq \mu >0} \mathfrak{g}_{\mu} \subset \mathfrak{n}$$
then $\dim(\mathfrak{n} \cap \mathfrak{k}) = \dim(\mathfrak{n}' \cap \mathfrak{k})$, since $\alpha$ is real. Since $\mathfrak{h}$ is $\theta$-stable, we have
$$\mathfrak{b} \cap \mathfrak{k} = \mathfrak{h} \cap \mathfrak{k} \oplus \mathfrak{n} \cap \mathfrak{k}$$
and therefore
$$\dim(\mathfrak{b} \cap \mathfrak{k}) = \dim(\mathfrak{t}) + \dim(\mathfrak{n} \cap \mathfrak{k}) = \dim(\mathfrak{t}) + \dim(\mathfrak{n}' \cap \mathfrak{k})$$
Next, we show that $c_{\alpha}^{\pm}\mathfrak{n}' = \mathfrak{n}'$. Recall,
$$c_{\alpha}^{\pm} = \mathrm{exp}(\mathrm{ad}(X_{\alpha} + X_{-\alpha}))$$
for a particular choice of root vectors $X_{\alpha}$ and $X_{-\alpha}$. If $\beta$ is a positive root not equal to $\alpha$, then
$$[X_{\alpha} + X_{-\alpha}, X_{\beta}] \in \mathfrak{g}_{\alpha+\beta} \oplus \mathfrak{g}_{-\alpha + \beta}$$
If $-\alpha + \beta$ is a root, then the simplicity of $\alpha$ implies that $-\alpha + \beta$ is positive. In any case, neither $\alpha+ \beta$ nor $-\alpha + \beta$ is equal to $\alpha$, so in fact
$$[X_{\alpha} + X_{-\alpha}, X_{\beta}] \in \mathfrak{n}'$$
And hence,
$$c_{\alpha}^{\pm}\mathfrak{n}' \subseteq \mathfrak{n}'$$
by exponentiation. Since $c_{\alpha}^{\pm}$ is an automorphism of $\mathfrak{g}$, this inclusion is an equality. 

Now, we have a decomposition
$$\mathfrak{b} = \mathfrak{h} \oplus \mathfrak{n}' \oplus \mathfrak{g}_{\alpha}$$
and hence a decomposition
$$c_{\alpha}^{\pm}\mathfrak{b} = \mathfrak{h}^{\alpha} \oplus \mathfrak{n}' \oplus \mathfrak{g}_{c_{\alpha}^{\pm}\alpha}$$
Since $\mathfrak{h}^{\alpha}$ is $\theta$-stable (by construction) and $\mathfrak{g}_{c_{\alpha}^{\pm}\alpha}$ is non-compact (by Proposition \ref{prop:calphafacts}.2), we have
$$c_{\alpha}^{\pm}\mathfrak{b} \cap \mathfrak{k} = (\mathfrak{h}^{\alpha} \oplus \mathfrak{g}_{c_{\alpha}^{\pm}\alpha}) \cap \mathfrak{k} \oplus \mathfrak{n}' \cap \mathfrak{k} = \mathfrak{t}^{\alpha} \oplus \mathfrak{n}' \cap \mathfrak{k}$$
and so
$$\dim(c_{\alpha}^{\pm}\mathfrak{b} \cap \mathfrak{k}) = \dim(\mathfrak{t}) + 1 + \dim(\mathfrak{n}' \cap \mathfrak{k}) = 1 + \dim(\mathfrak{b} \cap \mathfrak{k})$$
\end{proof}

\subsection{Principal nilpotent elements and quasi-split groups}\label{sec:principal}

We will need several basic facts about principal nilpotent elements.

\begin{proposition}[\cite{Kostant1959}, Sec 5]\label{prop:principalnilpotentelements}
The following are true:
\begin{itemize}
    \item[(i)] If $e$ is a principal nilpotent element belonging to the nilradical $\mathfrak{u}$ of a parabolic subalgebra $\mathfrak{q} \subset \mathfrak{g}$, then $\mathfrak{q}$ is a Borel subalgebra of $\mathfrak{g}$.
    \item[(ii)] If $e$ is a principal nilpotent element and 
    $$\phi: \mathfrak{sl}_2(\mathbb{C}) \to \mathfrak{g}$$
    is any homomorphism with $\phi(E) = e$, then $\phi(D)$ is $\mathbf{G}$-conjugate to
    $$\frac{1}{2}\sum_{\alpha \in \Delta^+(\mathfrak{g},\mathfrak{h})} \alpha^{\vee} \in \mathfrak{h}$$
    for any choice of Cartan subalgebra $\mathfrak{h}$ and positive system $\Delta^+(\mathfrak{g},\mathfrak{h})$.
    \item[(iii)] If $\mathfrak{h} \subset \mathfrak{g}$ is a Cartan subalgebra, $\Delta^+(\mathfrak{g},\mathfrak{h})$ is a positive system, and $\mathfrak{b} = \mathfrak{h}\oplus \mathfrak{n}$ is the corresponding Borel subalgebra of $\mathfrak{g}$, then $e \in \mathfrak{n}$ 
    $$e = \sum_{\alpha \in \Delta^+}c_{\alpha}X_{\alpha}$$
    is a principal nilpotent element if and only if $c_{\alpha} \neq 0$ for every simple root $\alpha \in \Delta^+(\mathfrak{g},\mathfrak{h})$.
\end{itemize}
\end{proposition}

Principal nilpotent elements are related to quasi-split groups.

\begin{proposition}[\cite{AdamsVogan1992}]\label{prop:definequasisplit}
The following are equivalent:

\begin{itemize}
    \item[(i)] $\mathfrak{g}$ contains a $\sigma$-stable Borel subalgebra $\mathfrak{b} \subset \mathfrak{g}$.
    \item[(ii)] $\mathfrak{g}$ contains a $\theta$-stable Cartan subalgebra $\mathfrak{h}$ and a $\theta$-stable positive system $\Delta^+(\mathfrak{g},\mathfrak{h})$ such that every simple imaginary root $\alpha \in \Delta^+(\mathfrak{g},\mathfrak{h})$ is noncompact (i.e. $\Delta^+(\mathfrak{g},\mathfrak{h})$ is large)
    \item[(iii)] $\mathfrak{g}$ contains a $\theta$-stable Cartan subalgebra $\mathfrak{h}$ and a $\theta$-stable positive system $\Delta^+(\mathfrak{g},\mathfrak{h})$ such that every imaginary simple root $\alpha \in \Delta^+(\mathfrak{g},\mathfrak{h})$ is noncompact.
    \item[(iv)] $\mathcal{N}_{\theta}$ contains a principal nilpotent element of $\mathfrak{g}$.
\end{itemize}
If any one of these equivalent conditions is satisfied, we say that $G$ (or $\mathfrak{g}_0$) is \emph{quasi-split}.
\end{proposition}

We will need a slight refinement of these results. 

\begin{proposition}\label{prop:noncompactrootsandAV}
Let $\mathfrak{q} \subset \mathfrak{g}$ be a $\theta$-stable parabolic subalgebra. Choose a $\theta$-stable Levi decomposition
$$\mathfrak{q} = \mathfrak{l} \oplus \mathfrak{u}$$
The following are equivalent:

\begin{itemize}
    \item[(i)] $\mathfrak{g}_0$ is quasi-split and the $\mathbf{K}$-saturation of $\mathfrak{u} \cap \mathfrak{p} + \mathcal{N}_{\mathfrak{l},\theta}$ has the same dimension as $\mathcal{N}_{\mathfrak{g},\theta}$.
    \item[(ii)] $\mathfrak{u} \cap \mathfrak{p} + \mathcal{N}_{\mathfrak{l},\theta}$ contains a principal nilpotent element of $\mathfrak{g}$
    \item[(iii)] There is a maximally compact $\theta$-stable Cartan subalgebra $\mathfrak{h}^{\mathrm{comp}} \subset \mathfrak{l}$ and a large, type Z  system $\Delta^+(\mathfrak{g},\mathfrak{h}^c)$, compatible with $\mathfrak{q}$.
    \item[(iv)] There is a maximally split $\theta$-stable Cartan subalgebra $\mathfrak{h}^{\mathrm{split}} \subset \mathfrak{l}$ and a large, type Z system $\Delta^+(\mathfrak{g},\mathfrak{h}^s)$, compatible with $\mathfrak{q}$.
\end{itemize}
\end{proposition}

\begin{proof}
\begin{description}
\item[(i) $\Rightarrow$ (ii)]
Recall from Section \ref{sec:threenilpotentcones} that $\mathcal{N}_{\mathfrak{g},\theta}$ decomposes into finitely-many $\mathbf{K}$-orbits. If $\mathcal{O} \subset \mathcal{N}_{\mathfrak{g},\theta}$ is one such $\mathbf{K}$-orbit
$$\dim(\mathcal{O}) = \frac{1}{2}\dim(\mathbf{G}\mathcal{O})$$
This is an easy consequence of part $(3)$ of Theorem \ref{thm:kostantsekiguchi2}. In particular, if $\mathcal{O}^{\mathrm{prin}} \cap \mathcal{N}_{\mathfrak{g},\theta}$ is nonempty, it decomposes into finitely-many $\mathbf{K}$-orbits $\mathcal{O}_1,...,\mathcal{O}_n$ and these are precisely the $\mathbf{K}$-orbits of maximal dimension on $\mathcal{N}_{\mathfrak{g},\theta}$.

Since $\mathfrak{g}_0$ is quasi-split, $\mathcal{O}^{\mathrm{prin}} \cap \mathcal{N}_{\mathfrak{g},\theta}$ is nonempty (by part $(4)$ of Proposition \ref{prop:definequasisplit}). Then the condition
$$\dim(\mathbf{K} \cdot (\mathfrak{u} \cap \mathfrak{p} + \mathcal{N}_{\mathfrak{l},\theta})) = \dim(\mathcal{N}_{\mathfrak{g},\theta})$$
implies that $\mathcal{O}_i \subset \mathbf{K} \cdot (\mathfrak{u} \cap \mathfrak{p} + \mathcal{N}_{\mathfrak{l},\theta})$ for some $i=1,...,n$. Hence, $\mathfrak{u} \cap \mathfrak{p} + \mathcal{N}_{\mathfrak{l},\theta}$ contains a principal nilpotent element of $\mathfrak{g}$.\\

\item[(ii) $\Rightarrow$ (i)]

Since $\mathfrak{u} \cap \mathfrak{p} + \mathcal{N}_{\mathfrak{l},\theta} \subset \mathcal{N}_{\mathfrak{g},\theta}$, and $\mathfrak{u} \cap \mathfrak{p} + \mathcal{N}_{\mathfrak{l},\theta}$ contians a principal nilpotent element, $\mathfrak{g}_0$ is quasi-split (by part $(4)$ of Proposition \ref{prop:definequasisplit}). Hence, $\mathfrak{u} \cap \mathfrak{p} + \mathcal{N}_{\mathfrak{l},\theta}$ has nonempty intersection with $\mathcal{O}_i$, for some $i=1,...,n$. Then by $\mathbf{K}$-invariance, $\mathcal{O}_i \subset \mathbf{K} \cdot (\mathfrak{u} \cap \mathfrak{p}+\mathcal{N}_{\mathfrak{l},\theta})$ and therefore
$$\dim(\mathbf{K} \cdot (\mathfrak{u} \cap \mathfrak{p} + \mathcal{N}_{\mathfrak{l},\theta})) = \dim(\mathcal{N}_{\mathfrak{g},\theta})$$

\item[(ii) $\Rightarrow$ (iii)]

Let $e_{\mathfrak{u}} \in \mathfrak{u} \cap \mathfrak{p}$, $e_{\mathfrak{l}} \in \mathcal{N}_{\theta}^{\mathfrak{l}}$, and assume $e_{\mathfrak{u}} + e_{\mathfrak{l}} \in \mathfrak{u} \cap \mathfrak{p} + \mathcal{N}_{\theta}^{\mathfrak{l}}$ is a principal nilpotent element of $\mathfrak{g}$. By Theorem \ref{thm:kostantsekiguchi1}, there is an embedding
$$\phi: \mathfrak{sl}_2(\mathbb{C}) \to \mathfrak{g}$$
intertwining $\theta$ with $\theta_s$ and $\sigma$ with $\sigma_s$ with the property that $\phi(E_c) = e_{\mathfrak{u}}+e_{\mathfrak{l}}$. Then $\phi(D_c)$ is a semisimple element of $\mathfrak{l} \cap \mathfrak{k}$. Choose a maximally compact $\theta$-stable Cartan subalgebra $\mathfrak{h}^{\mathrm{comp}} \subset \mathfrak{l}$ containing $\phi(D_c)$. By Proposition \ref{prop:principalnilpotentelements}, there is a positive system $\Delta^+(\mathfrak{g},\mathfrak{h}^{\mathrm{comp}})$ such that
$$\phi(D_c) = \frac{1}{2} \sum_{\alpha \in \Delta^+(\mathfrak{g},\mathfrak{h}^{\mathrm{comp}})} \alpha^{\vee}$$
Hence, the $2$-eigenspace of $\ad \phi(D_c)$ is the sum of the simple root spaces. Write $\Pi^+$ for the simple roots for $\Delta^+(\mathfrak{g},\mathfrak{h}^{\mathrm{comp}})$, and choose root vectors $X_{\alpha}$ for every $\alpha \in \Pi^+$. Then
\begin{equation}\label{eqn:e+e'decomp}
e_{\mathfrak{u}}+e_{\mathfrak{l}} = \sum c_{\alpha}X_{\alpha} \qquad c_{\alpha} \in \mathbb{C}
\end{equation}
If one of the $c_{\alpha}$ is zero, then $e_{\mathfrak{u}}+e_{\mathfrak{l}}$ is contained in the nilradical of the corresponding minimal parabolic $\mathfrak{p}_{\alpha} \subset \mathfrak{g}$, which is impossible by Proposition \ref{prop:principalnilpotentelements}. Hence, all $c_{\alpha}$ are nonzero. Since $e_{\mathfrak{u}}+e_{\mathfrak{l}} \in \mathfrak{p}$, the simple roots for $\Delta^+(\mathfrak{g},\mathfrak{h}^{\mathrm{comp}})$ are either complex (occuring in pairs) or noncompact imaginary. It remains to show that $\Delta^+(\mathfrak{g},\mathfrak{h}^{\mathrm{comp}})$ is compatible with $\mathfrak{q}$. 
By \ref{eqn:e+e'decomp}
$$e_{\mathfrak{u}} = \sum_{\alpha \in \Pi^+ \setminus \Delta^+(\mathfrak{l},\mathfrak{h}^{\mathrm{comp}})} c_{\alpha}X_{\alpha} \qquad e_{\mathfrak{l}} = \sum_{\alpha \in \Pi^+ \cap \Delta^+(\mathfrak{l},\mathfrak{h}^{\mathrm{comp}})} c_{\alpha}X_{\alpha} $$
The former implies that $\Pi^+ \setminus \Delta^+(\mathfrak{l},\mathfrak{h}^{\mathrm{comp}}) \subseteq \Delta(\mathfrak{u},\mathfrak{h}^c)$, and hence that $\Delta^+(\mathfrak{g},\mathfrak{h}^c) \setminus \Delta^+(\mathfrak{l},\mathfrak{h}^c) \subseteq \Delta(\mathfrak{u},\mathfrak{h}^c)$, since $\mathfrak{u}$ is invariant under the adjoint action of $\mathfrak{l}$. \\

\item[(iii) $\Rightarrow$ (ii)]

Since $\mathfrak{h}^{\mathrm{comp}}$ is maximally compact, all of its roots are complex or imaginary. The complex positive roots are $\theta$-stable, since $\Delta^+(\mathfrak{g},\mathfrak{h}^{\mathrm{comp}})$ is type $Z$. Choose positive root vectors $X_{\alpha}$ so that
$$\theta(X_{\alpha}) = - X_{\theta \alpha}$$
whenever $\alpha$ is complex, Define
$$e_{\mathfrak{u}} := \sum_{\substack{\alpha \in \Delta(\mathfrak{u},\mathfrak{h}^{\mathrm{comp}}) \\ \alpha \text{ complex or noncompact}}} X_{\alpha}$$
and
$$e_{\mathfrak{l}} := \sum_{\substack{\alpha \in \Delta^+(\mathfrak{l},\mathfrak{h}^{\mathrm{comp}}) \\ \alpha \text{ complex or noncompact}}} X_{\alpha}$$
By construction, $e_{\mathfrak{u}} \in \mathfrak{u} \cap \mathfrak{p}$ and $e_{\mathfrak{l}} \in \mathcal{N}_{\mathfrak{l}} \cap \mathfrak{p} = \mathcal{N}_{\mathfrak{l},\theta}$. Since $\Delta^+(\mathfrak{g},\mathfrak{h}^{\mathrm{comp}})$ is large, every simple root for $\Delta^+(\mathfrak{g},\mathfrak{h}^{\mathrm{comp}})$ appears in $e_{\mathfrak{u}} + e_{\mathfrak{l}}$ with nonzero coefficient. So by part $3$ of Proposition \ref{prop:principalnilpotentelements}, $e_{\mathfrak{u}} + e_{\mathfrak{l}}$ is a principal nilpotent element of $\mathfrak{g}$. \\

\item[(iii) $\Rightarrow$ (iv)]

There is a maximally split $\theta$-stable Cartan subalgebra $\mathfrak{h}^{\mathrm{split}} \subset \mathfrak{l}$ and sequence of noncompact simple imaginary roots
$$
\beta_1 \in \Delta^+_{i\mathbb{R}}(\mathfrak{l},\mathfrak{h}^{\mathrm{comp}}) \quad
\beta_2 \in \Delta^+_{i\mathbb{R}}(\mathfrak{l},d_{\beta}^+\mathfrak{h}^{\mathrm{comp}}) \quad ... \quad
\beta_n \in \Delta^+_{i\mathbb{R}}(\mathfrak{l},d_{\beta{n-1}}^+...d_{\beta_1}^+\mathfrak{h}^{\mathrm{comp}})
$$
such that
$$d_{\beta_n}^+ ... d_{\beta_1}^+\mathfrak{h}^{\mathrm{comp}} = \mathfrak{h}^{\mathrm{split}}$$
By Lemma \ref{lem:simplynoncompact} and an easy induction on $n$, we see that the positive system
$$d_{\beta_n}^+ ... d_{\beta_1}^+ \Delta^+(\mathfrak{g},\mathfrak{h}^{\mathrm{comp}}) \subset \Delta(\mathfrak{g},\mathfrak{h}^{\mathrm{split}})$$
is large. Applying simple reflections through complex simple roots, we can make this system type Z (see the proof of Proposition \ref{prop:maketypeZtypeL}).

Each $d_{\beta_i}^+$ acts on $\mathfrak{g}$ by an element of $\mathrm{Ad}(\mathfrak{l})$ and therefore preserves the nilradical $\mathfrak{u}$. Hence, this positive system is compatible with $\mathfrak{q}$.\\

\item[(iv) $\Rightarrow$ (iii)]

There is a maximally compact $\theta$-stable Cartan subalgebra $\mathfrak{h}^{\mathrm{comp}} \subset \mathfrak{l}$ and sequence of simple real roots
$$\alpha_1 \in \Delta^+_{\mathbb{R}}(\mathfrak{l},\mathfrak{h}^{\mathrm{split}}) \quad
\alpha_2 \in \Delta^+_{\mathbb{R}}(\mathfrak{l},c_{\alpha_1}^{\pm}\mathfrak{h}^{\mathrm{split}}) \quad ... \quad
\alpha_n \in \Delta^+_{\mathbb{R}}(\mathfrak{l},c_{\alpha_{n-1}}^{\pm}...c_{\alpha_1}^{\pm}\mathfrak{h}^{\mathrm{split}})
$$
such that
$$c_{\alpha_n}^{\pm} ... c_{\alpha_1}^{\pm}\mathfrak{h}^{\mathrm{split}} = \mathfrak{h}^{\mathrm{comp}}$$

By Lemma \ref{lem:simplynoncompact} and an easy induction on $n$, there is a sequence of signs $\epsilon_1,...,\epsilon_n$ so that the positive system
$$c_{\alpha_n}^{\epsilon_n} ... c_{\alpha_1}^{\epsilon_1} \Delta^+(\mathfrak{g},\mathfrak{h}^{\mathrm{split}}) \subset \Delta(\mathfrak{g},\mathfrak{h}^{\mathrm{comp}})$$
is large. Applying simple reflections through complex simple roots, we can arrange so that this positive system is type Z. It is compatible with $\mathfrak{q}$ for the same reasons as above.
\end{description}
\end{proof}

\subsection{The spherical principal series of infinitesimal character $0$}\label{sec:S0G}

Write 
$$\mathrm{Irrep}_0(G) \subset K^f(\mathfrak{g},\mathbf{K})$$
for the set of (isomorphism classes of) nonzero irreducible $(\mathfrak{g},\mathbf{K})$-modules of infinitesimal character $0$. By definition \ref{def:unipotentrep}, there is an inclusion
$$\mathrm{Unip}_R(\mathcal{O}) \subseteq \mathrm{Irrep}_0(G)$$
We will soon see that this inclusion is an equality, but this will require some work. 

Suppose $G$ is quasi-split. Then by Proposition \ref{prop:definequasisplit}, there is a $\sigma$-stable Borel subgroup $\mathbf{B} \subset \mathbf{G}$. Let $B = \mathbf{B}^{\sigma}$ and define
$$S_0(G) := \text{Harish-Chandra module of } \mathrm{Ind}^G_B \mathbb{C} = \mathbf{I}^{(\mathfrak{g},\mathbf{K})}_{(\mathfrak{b},\mathbf{T})} (- |\rho(\mathfrak{n})|)$$

\begin{proposition}\label{prop:S0G}
Suppose $G$ is quasi-split. Then $S_0(G)$ is independent (up to isomorphism) of $B$ and
$$[S_0(G)] \in \mathrm{Irrep}_0(G)$$
\end{proposition}

\begin{proof}
By Proposition \ref{prop:propsofind}, $S_0(G)$ has infinitesimal character $0$. It is nonzero by Theorem \ref{thm:realinductionexact}. Its irreducibility was established by Kostant in \cite[Thm 1]{kostant1969}. 
\end{proof}

In \cite{kostant1969}, Kostant also calculates the $\mathbf{K}$-structure of $S_0(G)$. He proves that $S_0(G)$ has the same $\mathbf{K}$-multiplicities as $\mathbb{C}[\mathcal{N}_{\mathfrak{g},\theta}]$, the ring of regular functions on $\mathcal{N}_{\mathfrak{g},\theta}$. Together with Theorem \ref{thm:resKinjective} this implies

\begin{theorem}\label{thm:grS0G}
There is an equality in $K^{\mathbf{K}}(\mathcal{N}_{\mathfrak{g},\theta})$
$$[\gr S_0(G)] = [\mathcal{O}_{\mathcal{N}_{\mathfrak{g},\theta}}]$$
In particular,
$$[S_0(G)] \in \mathrm{Unip}_R(\mathcal{O})$$
\end{theorem}

If $G$ is quasi-split, then $\mathrm{Irrep}_0(G) \neq \emptyset$ by Proposition \ref{prop:S0G}. The converse is also true.

\begin{proposition}\label{prop:qsplitrep0G}
$G$ is quasi-split if and only if
$$\mathrm{Irrep}_0(G) \neq \emptyset$$
\end{proposition}

\begin{proof}
If $G$ is quasi-split, then $[S_0(G)] \in \mathrm{Irrep}_0(G)$ by Proposition \ref{prop:S0G}.

Now suppose $\mathrm{Irrep}_0(G) \neq \emptyset$, and choose an element $[X] \in \mathrm{Irrep}_0(G)$. Let $Q^{\mathrm{min}} = L^{\mathrm{min}}U^{\mathrm{min}} \subset G$ be a minimal parabolic. By the Casselman subrepresentation theorem (Theorem \ref{thm:Casselman}), there is an irreducible finite-dimensional representation $V$ of $L^{\mathrm{min}}$ and an embedding of $(\mathfrak{g},\mathbf{K})$-modules
$$X \subseteq \mathbf{I}^{(\mathfrak{g},\mathbf{K})}_{(\mathfrak{l}^{\mathrm{min}},\mathbf{L}^{\mathrm{min}} \cap \mathbf{K})}V$$
By Proposition \ref{prop:propsofind}, the representation $V$ has infinitesimal character $-\rho(\mathfrak{u})$. Since the infinitesimal character of a finite-dimensional representation is always nonsingular, this means that $\mathfrak{l}^{\mathrm{min}}$ has no roots. Hence, $\mathfrak{l}^{\mathrm{min}}$ is a Cartan subalgebra and $Q^{\mathrm{min}}$ is a Borel. 
\end{proof}

To summarize: whenever $\mathrm{Irrep}_0(G)$ is nonempty, it contains a distinguished element $[S_0(G)]$. This representation has associated variety $\mathcal{N}_{\mathfrak{g},\theta}$ and is therefore an element of $\mathrm{Unip}_R(\mathcal{O})$. We will soon see that the classes $[S_0(L)]$, as $L\subset G$ varies, form the building blocks of $\mathrm{Unip}_R(\mathcal{O})$ (see Corollary \ref{cor:principalunipotentreps} for a precise statement and proof).

\subsection{Zuckerman parameters}\label{subsec:Zuckerman}

In this section, we define our second set of parameters for principal unipotent representations. 

\begin{definition}\label{def:ZP}
A Zuckerman parameter of infinitesimal character $0$ (a Z parameter, for short) is a $\mathbf{K}$-conjugacy class of triples $(\mathfrak{l},\mathfrak{q},\chi^{\#})$ consisting of
\begin{itemize}
    \item[(i)] a $\theta$-stable Levi subalgebra $\mathfrak{l} \subset \mathfrak{g}$, split modulo center,
    \item[(ii)] a $\theta$-stable parabolic subalgebra $\mathfrak{q} \subset \mathfrak{g}$ containing $\mathfrak{l}$ as a Levi subalgebra, and
    \item[(iii)] a one-dimensional $(\mathfrak{l},\mathbf{L} \cap \mathbf{K})$-module $\chi^{\#}$ satisfying $d\chi^{\#} = -\rho(\mathfrak{u})$
\end{itemize}
Denote the $\mathbf{K}$-conjugacy class of $(\mathfrak{l},\mathfrak{q},\chi^{\#})$ by $[\mathfrak{l},\mathfrak{q},\chi^{\#}]$ and denote the set of all such conjugacy classes by $Z_0(G)$.

A Z parameter is unipotent if it satisfies the additional condition
$$\mathfrak{u} \cap \mathfrak{p} + \mathcal{N}_{\mathfrak{l},\theta} \text{ contains a principal nilpotent element of } \mathfrak{g}$$
Write $\mathrm{Z}^*_0(G)$ for the set of unipotent Z parameters.
\end{definition}

\begin{remark}
If $[\mathfrak{l},\mathfrak{q},\chi^{\#}] \in Z_0(G)$, we can (and will) choose $\mathfrak{l}$ to be stable under $\sigma$. This allows us to define the Levi subgroup $L := Z_G(\mathfrak{l})$. 
\end{remark}

Define the function
$$\tilde{I}: Z_0(G) \to K^f_0(\mathfrak{g},\mathbf{K}) \qquad \tilde{I}[\mathfrak{l},\mathfrak{q},\chi^{\#}] = I(\mathfrak{l},\mathfrak{q}, \chi^{\#} \otimes S_0(L))$$
(see (\ref{eqn:eulerinduction})). There is a natural mapping
$$Z: \mathrm{BB}_0^*(G) \to \mathrm{Z}^*_0(G)$$
which intertwines $I$ and $\tilde{I}$. To define it, we will need a lemma

\begin{lemma}[\cite{AdamsLeeuwenTrapaVogan2017}, Lemma 16.1.4]\label{lem:weightoffindimrep}
Let $H \subset G$ be a $\theta$-stable Cartan subgroup. A character $\chi$ of $H$ is an extremal weight of an irreducible, finite-dimensional representation of $G$ if and only if
\begin{itemize}
    \item[(i)] $\langle d\chi,\alpha^{\vee}\rangle \in \mathbb{Z}$ for every root $\alpha \in \Delta(\mathfrak{g},\mathfrak{h})$, and
    \item[(ii)] $\chi(m_{\alpha}) = (-1)^{\langle d\chi,\alpha^{\vee}\rangle}$ for every real root $\alpha \in \Delta(\mathfrak{g},\mathfrak{h})$
\end{itemize}
\end{lemma}

Now, suppose $[\mathfrak{h},\mathfrak{b},\chi] \in \mathrm{BB}_0^*(G)$. Recall the parabolic subalgebra $\mathfrak{q}^Z = \mathfrak{l}^Z \oplus \mathfrak{u}^Z$ of $\mathfrak{g}$ defined in (\ref{eqn:definitionofqZ}).  By Condition (iii) of Definition \ref{def:unipotentBB}, every simple real root $\alpha \in \Delta^+(\mathfrak{g},H)$ is even for $\chi$. Hence, \emph{every} real root is even for $\chi$ by e.g. \cite[Cor 4.3.20]{Vogan1981}. Define a new character of $H$
$$\chi^L := \chi \otimes |\rho(\mathfrak{n} \cap \mathfrak{l}^Z)|$$
Since $|\rho(\mathfrak{n} \cap \mathfrak{l}^Z)|$ takes strictly positive values, every real root is also even for $\chi^L$. Now
$$d\chi^L = -\rho(\mathfrak{n}) + \rho(\mathfrak{n} \cap \mathfrak{l}^Z) = -\rho(\mathfrak{u}^Z)$$
Since $d\chi^L$ is the differential of a one-dimensional representation of $\mathfrak{l}^Z$, we have $\langle d\chi^L,\beta^{\vee}\rangle = 0$ for every $\beta \in \Delta(\mathfrak{l}^Z,\mathfrak{h})$. Hence by Proposition \ref{lem:weightoffindimrep}, $\chi^L$ is an extremal weight of a finite-dimensional representation of $L^Z$. Since $\langle d\chi^L,\beta^{\vee}\rangle = 0$ for every $\beta \in \Delta(\mathfrak{l}^Z,\mathfrak{h})$, $d\chi^L$ has minimal norm among its root lattice translates. So this finite-dimensional representation of $L^Z$ is necessarily a character. We will (somewhat abusively) denote this character (and its Harish-Chandra module) by $\chi^L$. 

\begin{proposition}\label{prop:BBtoZ}
Let $[\mathfrak{h},\mathfrak{b},\chi] \in \mathrm{BB}_0^*(G)$. Then $[\mathfrak{l}^Z,\mathfrak{q}^Z,\chi^Z] \in \mathrm{Z}_0^*(G)$. Furthermore, the mapping 
$$Z: \mathrm{BB}_0^*(G) \to \mathrm{Z}_0^*(G) \qquad Z[\mathfrak{h},\mathfrak{b},\chi] = [\mathfrak{l}^Z,\mathfrak{q}^Z,\chi^Z]$$
is surjective.
\end{proposition}

\begin{proof}
The Levi $L^Z$ is split modulo center, since all of the roots $\Delta(\mathfrak{l}^Z,\mathfrak{h})$ are real. the parabolic $\mathfrak{q}^Z$ is $\theta$-stable by Condition (ii) of Definition \ref{def:unipotentBB} and Proposition \ref{prop:typeZLparabolic}. The character $\chi^L$ satisfies $d\chi^L = -\rho(\mathfrak{u}^Z)$ by the calculation following Lemma \ref{lem:weightoffindimrep}. The final condition, namely that
$$\mathfrak{u} \cap \mathfrak{p} + \mathcal{N}_{\mathfrak{l},\theta} \text{ contains a principal nilpotent element of } \mathfrak{g}$$
follows from Condition (i) of Definition \ref{def:unipotentBB} and Proposition \ref{prop:noncompactrootsandAV}. Hence, $(L^Z,\mathfrak{q}^Z,\chi^Z) \in \mathrm{Z}_0^*(G)$, as desired.

Now suppose $(\mathfrak{l},\mathfrak{q},\chi^{\#}) \in \mathrm{Z}^*_0(G)$. By Proposition \ref{prop:noncompactrootsandAV}, there is a maximally split $\theta$-stable Cartan subalgebra $\mathfrak{h} \subset \mathfrak{l}$ and large, type Z system $\Delta^+(\mathfrak{g},\mathfrak{h})$, compatible with $\mathfrak{q}$. Let $\mathfrak{b}$ be the corresponding Borel subalgebra of $\mathfrak{g}$. Define a character $\chi$ of $H$ by
$$\chi := \chi^{\#}|_H \otimes |\rho(\mathfrak{n} \cap \mathfrak{l})|^{-1}$$
Then $[\mathfrak{h},\mathfrak{b},\chi] \in \mathrm{BB}_0^*(G)$ (Conditions (i) and (ii) are automatic by our choice of $\Delta^+(\mathfrak{g},\mathfrak{h})$ and Condition (iii) follows from Lemma \ref{lem:weightoffindimrep}) and $Z[\mathfrak{h},\mathfrak{b},\chi] = [\mathfrak{l},\mathfrak{q},\chi^{\#}]$. Hence, $Z$ is surjective onto $\mathrm{Z}_0^*(G)$. .
\end{proof}

\begin{proposition}\label{prop:commutativetriangle}
The triangle of functions
\begin{center}
\begin{tikzcd}
\mathrm{BB}^*_0(G) \arrow[r, "Z"] \arrow[dr, "I"] & \mathrm{Z}^*_0(G) \arrow[d, "\tilde{I}"]\\
& K_0^f(\mathfrak{g},\mathbf{K})
\end{tikzcd}
\end{center}
commutes.
\end{proposition}

\begin{proof}
Let $[\mathfrak{h},\mathfrak{b},\chi] \in \mathrm{BB}^*_0(G)$. Let $B^L \subset L^Z$ be the real Borel subgroup corresponding to $\mathfrak{l}^Z \cap \mathfrak{b}$. By \ref{eqn:realinductioninfinitesimal}, there is an isomorphism of $(\mathfrak{l}^Z, L^Z \cap K)$-modules
$$I^{(\mathfrak{l}^Z, L^Z \cap K)}_{(\mathfrak{l}^Z \cap \mathfrak{b}, T)} \chi \cong \mathrm{Ind}^{L^Z}_{B^L} \chi^Z$$
And since $\chi^Z$ extends to a character of $L^Z$
$$\mathrm{Ind}^{L^Z}_{B^L} \chi^Z \cong  \chi^Z \otimes \mathrm{Ind}^{L^Z}_{B^L} \mathbb{C}$$
Therefore by Theorem \ref{thm:realinductionexact}, there is an equality in $K^f(\mathfrak{l}^Z, L^Z \cap K)$
$$I[\mathfrak{h},\mathfrak{b}^{\mathfrak{l}},\chi] = \chi^L \otimes [S_0(L^Z)]$$
Using this and Proposition \ref{prop:indbystages}, we deduce
\begin{align*}
    I[\mathfrak{h},\mathfrak{b},\chi] &= I[\mathfrak{l}^Z,\mathfrak{q}^Z,I[\mathfrak{h},\mathfrak{b}^{\mathfrak{l}},\chi]]\\
    &= I[\mathfrak{l}^Z,\mathfrak{q}^Z,\chi^Z \otimes S_0(L^Z)]\\
    &= \tilde{I}[\mathfrak{l}^Z,\mathfrak{q}^Z,\chi^Z]
\end{align*}
\end{proof}

\subsection{Main results}\label{sec:mainresults}

\begin{proposition}\label{prop:nonzeroifflarge}
Let $[\mathfrak{h},\mathfrak{b},\chi] \in \mathrm{BB}_0(G)$. Then
$$I[\mathfrak{h},\mathfrak{b},\chi] \neq 0$$
if and only if $\Delta^+(\mathfrak{g},\mathfrak{h})$ is large.
\end{proposition}

\begin{proof}
By Proposition \ref{prop:typeZLparabolic}, we can assume without loss of generality that $[\mathfrak{h},\mathfrak{b},\chi]$ is type L. Recall the parabolic subalgebra $\mathfrak{q}^L = \mathfrak{l}^L \oplus \mathfrak{u}^L$ of $\mathfrak{g}$ defined in \ref{eqn:definitionofqL}. By Proposition \ref{prop:typeZLparabolic}, $\mathfrak{q}^L$ is $\sigma$-stable. Let $L^L = N_G(\mathfrak{l}^L)$, a $\theta$-stable Levi subgroup of $G$.

By Proposition \ref{prop:indbystages}
$$I[\mathfrak{h},\mathfrak{b},\chi] = I[\mathfrak{l}^L,\mathfrak{q}^L, I[\mathfrak{h}, \mathfrak{l}^L \cap \mathfrak{b},\chi]]$$
Since $\mathfrak{q}$ is real, $I[\mathfrak{l}^L,\mathfrak{q}^L, \cdot]$ is injective (see Theorem \ref{thm:realinductionexact}). Thus, by replacing $\mathfrak{g}$ with $\mathfrak{l}^L$, we can reduce to the case where $\Delta(\mathfrak{g},\mathfrak{h})$ has only imaginary roots. 

Now, assume $I[\mathfrak{h},\mathfrak{b},\chi] = 0$. Choose a positive system for $\Delta(\mathfrak{k},\mathfrak{t}) = \Delta_c(\mathfrak{g},\mathfrak{h})$. Then, by Theorem \ref{thm:LKTscohind}, the weight $-\rho(\mathfrak{n}) + 2\rho(\mathfrak{n} \cap \mathfrak{p})$ is non-dominant for $\Delta^+(\mathfrak{k},\mathfrak{t})$. Hence, there is a simple compact root $\alpha \in \Delta^+(\mathfrak{k},\mathfrak{t})$ with
\begin{align*}
    0 &> \langle -\rho(\mathfrak{n}) + 2\rho(\mathfrak{n}\cap \mathfrak{p}),\alpha^{\vee}\rangle\\
    &= \langle \rho(\mathfrak{n}) - 2\rho(\mathfrak{n}\cap \mathfrak{k}),\alpha^{\vee}\rangle\\
    &= \langle \rho(\mathfrak{n}),\alpha^{\vee}\rangle - 2\langle \rho(\mathfrak{n}\cap \mathfrak{k}), \alpha^{\vee}\rangle\\
    &= \langle \rho(\mathfrak{n}),\alpha^{\vee}\rangle - 2
\end{align*}
Since $\langle \rho(\mathfrak{n}),\alpha^{\vee}\rangle$ is an integer, this implies $\langle \rho(\mathfrak{n}),\alpha^{\vee}\rangle \leq 1$, and hence $\langle \rho(\mathfrak{n}),\alpha^{\vee}\rangle = 1$, which implies that $\alpha$ is simple for $\Delta^+(\mathfrak{g},\mathfrak{h})$.

Conversely, if there is a compact simple root, then $I[\mathfrak{h},\mathfrak{b},\chi] = 0$ by a character identity of Schmid (\cite{Schmid1977}, Theorem 1).
\end{proof}

\begin{proposition}\label{prop:BBunipirred}
Let $[\mathfrak{h},\mathfrak{b},\chi] \in \mathrm{BB}_0^*(G)$. Then
$$I[\mathfrak{h},\mathfrak{b},\chi] \in \unip_R(\cO)$$
\end{proposition}

\begin{proof}
By Proposition \ref{prop:nonzeroifflarge}, $I[\mathfrak{h},\mathfrak{b},\chi]$ is a nonzero element of $K^f(\mathfrak{g},\mathbf{K})$. It remains to show that $I[\mathfrak{h},\mathfrak{b},\chi]$ is irreducible and $\AV(\mathrm{Ann}(I[\mathfrak{h},\mathfrak{b},\chi]))) = \cN$.  

By Proposition \ref{prop:commutativetriangle}, we have
$$I[\mathfrak{h},\mathfrak{b},\chi] = I[\mathfrak{l}^Z,\mathfrak{q}^Z,\chi^L \otimes S_0(L)]$$
We know that $S_0(L)$ is irreducible by Proposition \ref{prop:S0G}. Hence, $I[\mathfrak{l}^Z,\mathfrak{q}^Z,\chi^L \otimes S_0(L)]$ is irreducible by Theorem \ref{thm:cohindirreducible}.

We can use Proposition \ref{prop:AVcohind} and Theorem \ref{thm:grS0G} to compute the associated variety of $I[\mathfrak{h},\fb,\chi]$:
$$\AV(I[\mathfrak{h},\fb,\chi]) = \AV(I[\mathfrak{l}^Z,\fq^Z,\chi^L \otimes S_0(L)]) = \mathbf{K}(\mathfrak{u} \cap \mathfrak{p} + \cN_{\fl,\theta})$$
Hence by Theorem \ref{thm:twoAVs}, we have
$$\AV(\mathrm{Ann}(I[\mathfrak{h},\fb,\chi])) = \mathbf{G}(\mathfrak{u} \cap \mathfrak{p} + \cN_{\fl,\theta})$$
which is $\cN$ by Proposition \ref{prop:BBtoZ}.
\end{proof}

Hence, we obtain a commutative triangle
\begin{center}
\begin{tikzcd}
\mathrm{BB}^*_0(G) \arrow[r, "Z"] \arrow[dr, "I"] & \mathrm{Z}^*_0(G) \arrow[d, "\tilde{I}"]\\
& \unip(\cO)
\end{tikzcd}
\end{center}

We will prove that $I: \mathrm{BB}^*_0(G) \to \mathrm{Rep}_0(G)$ is a bijection. Together with the surjectivity of $Z$ (see Proposition \ref{prop:BBtoZ}) this will imply that all three maps in the diagram above are bijections.

\begin{proposition}\label{prop:fromBBtoBBunip}
Let $[\mathfrak{h}_0,\mathfrak{b}_0,\chi_0] \in \mathrm{BB}_0(G)$ and assume
$$I[\mathfrak{h}_0,\mathfrak{b}_0,\chi_0] \neq 0$$
Then there is a collection of unipotent parameters $\Omega^* \subset \mathrm{BB}_0^*(G)$ such that
$$I[\mathfrak{h}_0,\mathfrak{b}_0,\chi_0] = \sum_{[\mathfrak{h},\mathfrak{b},\chi] \in \Omega^*} \pm I[\mathfrak{h},\mathfrak{b},\chi]$$
\end{proposition}

\begin{proof}
Let $S \subset \mathrm{BB}_0(G)$ be the set of all parameters which can be obtained from $[\mathfrak{h}_0,\mathfrak{b}_0,\chi_0]$ through a sequence of 

\begin{enumerate}
    \item Cayley transforms through odd simple real roots, and
    \item simple reflections through complex simple roots $\alpha$ satisfying
    $$\theta(\alpha) \in -\Delta^+(\mathfrak{g},\mathfrak{h})$$
\end{enumerate}
By a \emph{decomposition of} $[\mathfrak{h}_0,\mathfrak{b}_0,\chi_0]$, we will mean a subset $\Omega \subset S$ such that
\begin{enumerate}
    \item Every parameter $[\mathfrak{h},\mathfrak{b},\chi] \in \Omega$ has
    $$I[\mathfrak{h},\mathfrak{b},\chi] \neq 0$$
    and
    \item $$I[\mathfrak{h}_0,\mathfrak{b}_0,\chi_0] = \sum_{[\mathfrak{h},\mathfrak{b},\chi] \in \Omega} \pm I[\mathfrak{h},\mathfrak{b},\chi]$$
\end{enumerate}
Let $D$ be the set of all decompositions of $[\mathfrak{h}_0,\mathfrak{b}_0,\chi_0]$. Note that $D \neq \emptyset$, since $\{[\mathfrak{h}_0,\mathfrak{b}_0,\chi_0]\} \in D$. 

Define a function $\tilde{d}: D \to \mathbb{N}$ by 
$$\tilde{d}(\Omega) := \sum_{[\mathfrak{h},\mathfrak{b},\chi] \in \Omega} d[\fh,\fb,\chi]$$
If $\Omega_1,\Omega_2 \in D$, we declare that $\Omega_1 \leq \Omega_2$ if and only if \emph{every} parameter in $\Omega_2$ can be obtained from some parameter in $\Omega_1$ through a sequence of Cayley transforms and simple reflections (of the types described above). Note that $\widetilde{d}$ is strictly monotonic with respect to this relation (this is a consequence of Propositions \ref{prop:cayleytransformsincreased} and \ref{prop:simplereflectionincreasesd}). Hence, $\leq$ defines a partial order on $D$ (the only nontrivial condition is antisymmetry; this follows from the monotonicity of $\widetilde{d}$). Since $D$ is finite, it contains a maximal element $\Omega^*$ with respect to $\leq$.

Suppose $[\mathfrak{h},\mathfrak{b},\chi] \in \Omega^*$. We want to show that $[\mathfrak{h},\mathfrak{b},\chi] \in \mathrm{BB}_0^*(G)$. Conditions (ii) and (iii) of Definition \ref{def:unipotentBB} follow from the maximality of $\Omega^*$. Condition (i) follows from Proposition \ref{prop:nonzeroifflarge}. 
\end{proof}

\begin{corollary}\label{cor:Isurjective}
The map
$$I: \mathrm{BB}^*_0(G) \to \mathrm{Irrep}_0(G)$$
is surjective.
\end{corollary}

\begin{proof}
Assume $G$ is quasi-split (if it is not, $\mathrm{Rep}_0(G) = \emptyset$ by Proposition \ref{prop:qsplitrep0G} and the Corollary is vacuous). Let $B_0 = H_0N_0 \subset G$ be a Borel subgroup of $G$ and let $X \in \mathrm{Irrep}_0(G)$. By the Casselman Subrepresentation Theorem (Theorem \ref{thm:Casselman}) there is a character $\chi_0$ of $H$ and an embedding of $(\mathfrak{g},\mathbf{K})$-modules
$$X \subseteq \mathbf{I}^{(\mathfrak{g},\mathbf{K})}_{(\mathfrak{b}_0,T_0)} \chi_0$$
Since $X$ has infinitesimal character $0$, $d\chi_0 = -\rho(\mathfrak{n}_0)$.

Evidently $[\mathfrak{h}_0,\mathfrak{b}_0,\chi_0] \in \mathrm{BB}_0(G)$ and $I[\mathfrak{h}_0,\mathfrak{b}_0,\chi_0] \neq 0$, since $X \neq 0$ is a submodule. So Proposition \ref{prop:fromBBtoBBunip} furnishes a collection $\Omega^* \subset \mathrm{BB}^*_0(G)$ of unipotent BB parameters such that
$$I[\mathfrak{h}_0,\mathfrak{b}_0,\chi_0] = \sum_{[\mathfrak{h},\mathfrak{b},\chi] \in \Omega^*} \pm I[\mathfrak{h},\mathfrak{b},\chi]$$
By Proposition \ref{prop:BBunipirred}, the terms on the right are irreducible. Hence, there is a parameter $[\mathfrak{h},\mathfrak{b},\chi] \in \Omega^*$ with
$$X = I[\mathfrak{h},\mathfrak{b},\chi]$$
as desired. 
\end{proof}

\begin{proposition}\label{prop:Iinjective}
The map 
$$I: \mathrm{BB}^*_0(G) \to \mathrm{Irrep}_0(G)$$
is injective.
\end{proposition}

\begin{proof}
This is a consequence of Corollary \ref{cor:Gammainjective}.  and Proposition \ref{prop:BBunipirred}.
\end{proof}

\begin{corollary}
\label{cor:principalunipotentreps}
We have $\mathrm{Irrep}_0(G) = \unip_R(\cO)$, and there is a commutative triangle of bijections
\begin{center}
\begin{tikzcd}
\mathrm{BB}^*_0(G) \arrow[r, "Z"] \arrow[dr, "I"] & \mathrm{Z}^*_0(G) \arrow[d, "\tilde{I}"]\\
& \unip_R(\cO)
\end{tikzcd}
\end{center}
\end{corollary}

\begin{proof}
Since $Z: \mathrm{BB}_0^*(G) \to \mathrm{Z}_0^*(G)$ is a surjection (Proposition \ref{prop:BBtoZ}) and $I: \mathrm{BB}_0^*(G) \to \mathrm{Irrep}_0(G)$ is a bijection (Corollary \ref{cor:Isurjective} and Proposition \ref{prop:Iinjective}), 
$\tilde{I}: \mathrm{Z}^*_0(G) \to \mathrm{Irrep}_0(G)$ is a bijection. This proves both claims at once.
\end{proof}

\subsection{$\mathbf{K}$-types and associated varieties of principal unipotent representations}\label{subsec:Ktypes}

One advantage of the parameterization
$$\tilde{I}: Z_0^*(G) \xrightarrow{\sim} \unip_R(\cO)$$
described in Corollary \ref{cor:principalunipotentreps} is that $\mathbf{K}$-types and associated varieties of the modules $\tilde{I}[\mathfrak{l},\mathfrak{q},\chi^{\#}]$ are particularly easy to compute. Choose a Cartan subalgebra $\mathfrak{t} \subset \mathfrak{k}$ and a positive system $\Delta^+(\mathfrak{k},\mathfrak{t})$ compatible with $\mathfrak{q}$. Consider the dot-action of $W(\mathfrak{k})$ on $\mathfrak{t}^*$:
$$w \cdot (\lambda) = w(\lambda+\rho) - \rho$$
If $\lambda \in \mathfrak{t}^*$, there is at most one element $w \in W(\mathfrak{k})$ such that $w \cdot \lambda$ is dominant for $\Delta^+(\mathfrak{k},\mathfrak{t})$. Write $w_{\lambda}$ for this element (if it exists) and $\ell(w_{\lambda})$ for its length. If $\lambda$ is integral (for $\Delta(\mathfrak{k},\mathfrak{t})$) and dominant (for $\Delta^+(\mathfrak{k},\mathfrak{t})$), write $V(\lambda)$ for the unique irreducible $\mathbf{K}$-representation with highest weight $\lambda$ (if $\lambda$ is not integral or dominant, let $V(\lambda) = 0$). For $[\mathfrak{l},\mathfrak{q},\chi^{\#}] \in Z_0^*(G)$, consider the multiset of weights
$$\Lambda(\mathfrak{l},\mathfrak{q},\chi^{\#}) := \Delta(2\rho(\mathfrak{u} \cap \mathfrak{p}) - \rho(\mathfrak{u}) \otimes \mathbb{C}[\mathcal{N}_{\mathfrak{l},\theta}] \otimes S[\mathfrak{u} \cap \mathfrak{p}], \mathfrak{t})$$
The following theorem is an easy consequence of the Blattner formula and Theorem \ref{thm:grS0G}.

\begin{theorem}
Let $[\mathfrak{l},\mathfrak{q},\chi^{\#}] \in Z_0^*(G)$. Then
$$I[\mathfrak{l},\mathfrak{q},\chi^{\#}] \simeq_{\mathbf{K}} \sum_{\lambda \in \Lambda(\mathfrak{l},\mathfrak{q},\chi^{\#})} (-1)^{\ell(w_0) + \ell(w_{\lambda})} [V(w_{\lambda} \cdot \lambda)]$$
\end{theorem}

The associatied variety of $I[\mathfrak{l},\mathfrak{q},\chi^{\#}]$ can be computed using Proposition \ref{prop:AVcohind} and Theorem \ref{thm:grS0G} (see the proof of Proposition \ref{prop:BBunipirred}).

\begin{proposition}
Let $[\mathfrak{l},\mathfrak{q},\chi^{\#}] \in Z_0^*(G)$. Then
$$\AV\left(I[\mathfrak{l},\mathfrak{q},\chi^{\#}]\right) = \mathbf{K} \left(\mathfrak{u} \cap \mathfrak{p} + \mathcal{N}_{\mathfrak{l},\theta}\right)$$
\end{proposition}

\section{Langlands parameters of principal unipotent representations}\label{sec:Langlands}

In this section, we will describe the Langlands parameters of principal unipotent representations. 

\subsection{Langlands classification}

We begin by reviewing the Langlands classification of irreducible representations of real reductive groups, as formulated in \cite{ABV}. For details and proofs, we refer the reader to \cite[Sec 4-5]{ABV} or \cite[Sec 2-3]{AdamsVogan2015}. Let $\mathbf{G}$ be a complex connected reductive algebraic group and form the corresponding \emph{based root datum}
$$\Phi(\mathbf{G}) = (X^*,X_*,\Delta,\Delta^{\vee},\vee, \Pi,\Pi^{\vee})$$
(these symbols denote, respectively: the character lattice, the co-character lattice, the roots, the co-roots, the bijection between them, the simple roots, and the simple co-roots. This seven-tuple appears to depend on a choice of maximal torus and Borel. Up to canonical isomorphism, it does not). Let $\mathrm{Aut}(\mathbf{G})$ be the group of (algebraic group) automorphisms of $\mathbf{G}$ and let $\mathrm{Int}(\mathbf{G}) \subset \mathrm{Aut}(\mathbf{G})$ be the normal subgroup of inner automorphisms. Two automorphisms $\theta,\theta' \in \mathrm{Aut}(\mathbf{G})$ are \emph{inner} if they are in the same (left) $\mathrm{Int}(\mathbf{G})$-coset. There is a canonical map $\mathrm{Aut}(\mathbf{G}) \to \Aut(\Phi(\mathbf{G}))$, inducing a short exact sequence
\begin{equation}\label{eqn:ses}
1 \to \mathrm{Int}(\mathbf{G}) \to \mathrm{Aut}(\mathbf{G}) \to \Aut(\Phi(\mathbf{G})) \to 1
\end{equation}
Hence, the inner classes in $\mathrm{Aut}(\mathbf{G})$ are parameterized by automorphisms of $\Phi(\mathbf{G})$.

Now fix a pinning $(\mathbf{h}, \mathfrak{b}, \{X_{\alpha}\})$ of $\mathbf{G}$ (by this we mean a Cartan $\mathfrak{h}$, a Borel $\mathfrak{b} \supset \mathfrak{h}$, and a set $\{X_{\alpha}\} \subset \mathfrak{b}$ of simple root vectors). For each $\delta \in \Aut(\Phi(\mathbf{G}))$, there is a unique automorphism $\theta_0 \in \mathrm{Aut}(\mathbf{G})$ in the inner class of $\delta$ which preserves the chosen pinning
$$d\theta_0(\mathfrak{h}) = \mathfrak{h} \qquad d\theta_0(\mathfrak{b}) = \mathfrak{b} \qquad d\theta_0\{X_{\alpha}\} = \{X_{\alpha}\}$$
This automorphism is called the \emph{distinguished} automorphism in the inner class of $\delta$, and the assignment $\delta \mapsto \theta_0$ defines a splitting of (\ref{eqn:ses}). 

The dual group $\mathbf{G}^{\vee}$ comes equipped with a canonical isomorphism of based root data $\Phi(\mathbf{G}^{\vee}) \simeq \Phi(\mathbf{G})^{\vee}$ and hence a canonical isomorphism of groups (called tranpose)
$$t: \mathrm{Aut}(\Phi(\mathbf{G})) \simeq \mathrm{Aut}(\Phi(\mathbf{G}^{\vee}))$$
For what follows, we will fix both a pinning $(\fh^{\vee},\fb^{\vee},\{X_{\alpha^{\vee}}\})$ of $\mathbf{G}^{\vee}$ and an involution $\delta \in \mathrm{Aut}(\Phi(\mathbf{G}))$. Let $w_0 \in W(\mathbf{G})$ be the longest element of the Weyl group. Note that $w_0^2=1$ and $w_0(\Pi) = -\Pi$. Hence $-w_0$ can be regarded as an involution of $\Phi(\mathbf{G}^{\vee})$ (and as such, it commutes with every element of $\mathrm{Aut}(\Phi(\mathbf{G}^{\vee}))$). We will consider the involution $\delta^{\vee} := - w_0\delta^t \in \mathrm{Aut}(\Phi(\mathbf{G}^{\vee}))$. As explained in the previous paragraph, there is a unique distinguished involution $\theta_0^{\vee} \in \mathrm{Aut}(\mathbf{G}^{\vee})$ in the inner class of $\delta^{\vee}$.

The \emph{L-group} of $\mathbf{G}$ (with respect to $\delta$) is the semi-direct product
$$\mathbf{G}^L := \mathbf{G}^{\vee} \ltimes \{1,\theta_0^{\vee}\}$$
The Weil group of $\mathbb{R}$ is the Lie group defined by
$$W_{\RR} := \langle \CC^{\times},j\rangle \qquad jzj^{-1} =\overline{z}, \quad j^2=-1$$

\begin{definition}[\cite{Langlands1989}, see also \cite{Borel1979}, Sec 8.2]
A \emph{Langlands parameter} for $\mathbf{G}$ (an \emph{L parameter}, for short) is a $\mathbf{G}^{\vee}$-conjugacy class of continuous homomorphisms
$$\phi: W_{\RR} \to \mathbf{G}^L$$
such that
\begin{itemize}
    \item[(i)] $\phi(\CC^{\times})$ consists of semisimple elements
    \item[(ii)] $\phi(j) \in \mathbf{G}^{\vee}\theta_0^{\vee}$
\end{itemize}
Denote the $\mathbf{G}^{\vee}$-conjugacy class of $\phi$ by $[\phi]$, and denote the set of L parameters for $\mathbf{G}$ by $\Pi(\mathbf{G})$.
\end{definition}

For our purposes, a more concrete description of L parameters will be convenient. Suppose $(y,\lambda)$ is a pair consisting of an element $y \in \mathbf{G}^{\vee}\theta_0^{\vee}$ and a semisimple element $\lambda \in \fg^{\vee}$ such that $\exp(2\pi i \lambda) = y^2$. From $(y,\lambda)$, we obtain an L parameter as follows. First, conjugate $(y,\lambda)$ by $\mathbf{G}^{\vee}$ so that
$$y \in N_{\mathbf{G}^{\vee}\theta_0^{\vee}}(\mathbf{H}^{\vee}) \qquad \lambda \in \fh^{\vee}$$
Then, define
\begin{equation}\label{eqn:pairtoparam}\phi(j) = \exp(-\pi i \lambda)y \qquad \phi(\exp(\pi z)) = \exp(\pi z \lambda + \pi \overline{z}\Ad(y)\lambda)\end{equation}
It is easy to check that $[\phi] \in \Pi(\mathbf{G})$. Conversely, given an L parameter $[\phi] \in \Pi(\mathbf{G})$, we obtain a pair $(y,\lambda)$ as follows. First, conjugate $\phi$ so that 
$$\phi(\CC^{\times}) \subset \mathbf{H}^{\vee} \qquad \phi(j) \in N_{\mathbf{G}^{\vee}\theta_0^{\vee}}(\mathbf{H}^{\vee}) $$
Then, define
\begin{equation}\label{eqn:paramtopair}\lambda = \frac{1}{2}(d\phi(1) - id\phi(i)) \qquad y = \exp(\pi i y)\phi(j)\end{equation}

\begin{lemma}[\cite{ABV}, Prop 5.6]\label{lem:Lparamstopairs}
The formulas (\ref{eqn:pairtoparam}) and (\ref{eqn:paramtopair}) define mutually inverse bijections between $\Pi(\mathbf{G})$ and the set
$$\{(y,\lambda) \mid y \in \mathbf{G}^{\vee}\theta_0^{\vee}, \ \lambda \in \fg^{\vee} \text{ semisimple}, \ \exp(2\pi i \lambda) = y^2\}/\mathbf{G}^{\vee}$$
\end{lemma}

Now let $G$ be a real form of $\mathbf{G}$. Assume that the Cartan involution $\theta \in \mathrm{Aut}(\mathbf{G})$ is in the inner class of $\delta$. 

\begin{theorem}[Langlands \cite{Langlands1989}, see also \cite{Borel1979}, Sec 11]\label{thm:Langlands}
There is a natural map
$$\varphi: \mathrm{Irrep}(G) \to \Pi(\mathbf{G})$$
The fibers of this map are called L-packets in $\mathrm{Irrep}(G)$. If $G$ is quasi-split, then this map is surjective (i.e. all L-packets are non-empty).
\end{theorem}

\subsection{From Beilinson-Bernstein parameters to Langlands parameters}

Fix $\lambda \in \mathfrak{h}^*$ satisfying the dominance condition (\ref{eqn:integraldominance}). Recall from Section \ref{subsec:BBparameters} the set $\BB_{\lambda}(G)$ and the bijection
$$\BB_{\lambda}(G) \xrightarrow{\sim} \{\text{irreducibles in } M(\mathcal{D}_{\lambda-\rho},\mathbf{K})\}$$
Left-composing with $\Gamma: M(\mathcal{D}_{\lambda-\rho},\mathbf{K}) \to M_{\lambda}^f(\mathfrak{g},\mathbf{K})$, we obtain a map
$$\overline{I}: \BB_{\lambda}(G) \to \mathrm{Irrep}_{\lambda}(G) \cup \{0\}$$
which by Theorem \ref{thm:globalsections} is surjective onto $\mathrm{Irrep}_{\lambda}(G)$. Choose $[\fh,\fb,\chi] \in \BB_{\lambda}(G)$ with $\overline{I}[\fh,\fb,\chi] \neq 0$. Following \cite{AdamsVogan2015}, we will construct the L parameter $\varphi(\overline{I}[\fh,\fb,\chi])$. The construction is as follows:

\begin{enumerate}
    \item Define an involution $\theta^{\vee} \in \mathrm{Aut}(\mathbf{H}^{\vee})$ by
    $$\theta^{\vee} := -\theta^t$$
    Note that there is a uniquely defined $w \in W$ such that $\theta^{\vee} = w\delta^{\vee}$
    
    \item For each simple root $\alpha \in \Pi$, the pinning on $\mathbf{G}^{\vee}$ defines a canonical homomorphism $\phi_{\alpha}: SL_2(\CC) \to \mathbf{G}^{\vee}$ such that
    $$\phi_{\alpha}(D) = \alpha(1) \in \mathbf{H}^{\vee} \qquad \phi_{\alpha}(E) = X_{\alpha^{\vee}}$$
    Put
    $$\sigma_{\alpha} := \phi_{\alpha}\begin{pmatrix}0 & 1\\-1 & 0 \end{pmatrix}$$
    \item Form a reduced word decomposition $w = \prod w_{s_{\alpha_i}}$ and define $\sigma_w := \prod \sigma_{w_{\alpha_i}}$ (by a result of Tits \cite{Tits1966}, this element is well-defined).
    \item Choose any $\mu \in X^*(\mathbf{H})$ such that $X|_{\mathbf{T}} = \chi|_{\mathbf{T}}$. Define
    $$y := X(-1)e^{i\pi\lambda}\sigma_w\theta_0^{\vee} \in \mathbf{G}^{\vee}\theta_0^{\vee}$$
\end{enumerate}

One easily checks that the pair $(y,\lambda)$ is of the form described in Lemma \ref{lem:Lparamstopairs}. 

\begin{proposition}[\cite{AdamsVogan2015}, Sec 3]\label{prop:BBtoL}
Let $[\fh,\fb,\chi] \in \BB_{\lambda}(G)$ and suppose $\overline{I}[\fh,\fb,\chi] \neq 0$. Then the L parameter $\varphi(\overline{I}[\fh,\fb,\chi])$ corresponds under the bijection of Lemma \ref{lem:Lparamstopairs} to the pair $(y,\lambda)$ constructed above.
\end{proposition}

\subsection{Langlands parameters of principal unipotent representations}

Now suppose $G$ is quasi-split. 

\begin{definition}
An L parameter $[\phi] \in \Pi(\mathbf{G})$ is \emph{principal unipotent} if $\phi(\CC^{\times}) = 1$. Write $\Pi_0(\mathbf{G})$ for the set of principal unipotent L parameters.
\end{definition}

\begin{proposition}
Let $\cO \subset \cN$ be the principal nilpotent orbit. Then 
$$\varphi(\unip_R(\cO)) = \Pi_0(\mathbf{G})$$
\end{proposition}

\begin{proof}
By Proposition \ref{prop:BBtoL}, we have
$$\varphi(\mathrm{Irrep}_0(G)) \subseteq \Pi_0(\mathbf{G}) \qquad \text{and} \qquad \varphi^{-1}(\Pi_0(\mathbf{G})) \subseteq \mathrm{Irrep}_0(G)$$
Since $G$ is quasi-split, $\varphi$ is surjective (see Theorem \ref{thm:Langlands}). Hence
$$\varphi(\mathrm{Irrep}_0(G)) = \Pi_0(\mathbf{G})$$
By Corollary \ref{cor:principalunipotentreps}, $\mathrm{Irrep}_0(G) = \unip_R(\cO)$, completing the proof.
\end{proof}

The main result of Section \ref{sec:classification} provides a set of `normalized' representatives for $\Pi(\mathbf{G})$.

\begin{definition}
An element $y \in N_{\mathbf{G}^{\vee}\theta_0^{\vee}}(\mathbf{H}^{\vee})$ is \emph{principal unipotent} if 
\begin{itemize}
    \item[(i)] $y^2=1$ (hence, $\Ad(y)$ defines an involution of $\mathbf{G}^{\vee}$ which preserves $\mathbf{H}^{\vee}$),
    \item[(ii)] $\fb^{\vee}$ is small for $\Ad(y)$ (see Definition \ref{def:largetypeZtypeL}(ii))
    \item[(iii)] $\fb^{\vee}$ is type $L$ for $\Ad(y)$ (see Definition \ref{def:largetypeZtypeL}(iii))
\end{itemize}
\end{definition}

The following proposition is an immediate consequence of Corollary \ref{cor:principalunipotentreps} and Proposition \ref{prop:BBtoL}.

\begin{proposition}\label{prop:descriptionofprincipalLparams}
There is a natural bijection
$$\Pi_0(\mathbf{G}) \simeq \{\text{principal unipotent } y \in N_{\mathbf{G}^{\vee}\theta_0^{\vee}}(\mathbf{H}^{\vee})\}/\mathbf{H}^{\vee}$$
\end{proposition}

Specializing Proposition \ref{prop:descriptionofprincipalLparams} to the case when $G$ is split (and hence $\theta_0^{\vee} = \mathrm{id}$), we obtain the following result.

\begin{corollary}\label{cor:order2}
There is a natural bijection
$$\{y \in \mathbf{G}^{\vee} \mid y^2=1\}/\mathbf{G}^{\vee} \simeq \{y \in N(\mathbf{H}^{\vee}) \mid y^2=1,  \ \mathfrak{b}^{\vee} \text{ is large and type L for } \Ad(y)\}/\mathbf{H}^{\vee}$$
\end{corollary}

Note that Corollary \ref{cor:order2} is a purely structural fact about connected reductive algebraic groups.

\appendix

\section{Clifford Theory for Harish-Chandra Modules}

Fix $\mathfrak{g}$ and $\mathbf{K}$ as in Section \ref{subsec:induction}. Suppose $\mathbf{K}'\subset \mathbf{K}$ is an index-2 subgroup and write $A= \mathbf{K}/\mathbf{K}'$. The theory of induction from finite-index subgroups (see \cite{Clifford1937}) has a $(\mathfrak{g},\mathbf{K})$-module analog. 

\begin{proposition}[Clifford Theory for Harish-Chandra Modules]\label{prop:cliffordtheoryindex2}
Choose an element $s \in \mathbf{K}\setminus \mathbf{K}'$, so that
$$\mathbf{K} = \mathbf{K}' \sqcup s\mathbf{K}'$$
Let $\epsilon$ be the one-dimensional $(\mathfrak{g},\mathbf{K})$-module which is trivial as a $\mathfrak{g}$-module with $\mathbf{K}$ acting by the nontrivial character of $A$. Since $\mathbf{K}'$ has finite index in $\mathbf{K}$, we have that $\mathbf{K}^0 \subset \mathbf{K}'$, so $\epsilon$ is indeed a well-defined $(\mathfrak{g},\mathbf{K})$-module.

If $X$ is an irreducible $(\mathfrak{g},\mathbf{K})$-module, we can define a second (possibly isomorphic) irreducible $(\mathfrak{g},\mathbf{K})$-module $X \otimes \epsilon$. If $X'$ is an irreducible $(\mathfrak{g},\mathbf{K}')$-module, we can define a second (possibly isomorphic) irreducible $(\mathfrak{g},\mathbf{K}')$-module $X'_s$ by twisting the $K$-action on $X'$ by $s$, i.e.
$$k' \cdot x := (sk's^{-1}) \cdot x$$
The assignments $X \mapsto X \otimes \epsilon$ and $X' \mapsto X'_s$ define $A$-actions on the sets $\mathrm{Irr}(\mathfrak{g},\mathbf{K})$ and $\mathrm{Irr}(\mathfrak{g},\mathbf{K}')$ of irreducible $(\mathfrak{g},\mathbf{K})$- and $(\mathfrak{g},\mathbf{K}')$-modules, respectively. We have

\begin{enumerate}
    \item If $X \in \mathrm{Irr}(\mathfrak{g},\mathbf{K})$ is fixed by $A$, then $\mathrm{Res}^{(\mathfrak{g},\mathbf{K})}_{(\mathfrak{g},\mathbf{K}')}X$ is reducible, with two irreducible summands. If $X, Y \in \mathrm{Irr}(\mathfrak{g},\mathbf{K})$ are $A$-conjugate, then $\mathrm{Res}^{(\mathfrak{g},\mathbf{K})}_{(\mathfrak{g},\mathbf{K}')}X$ and $\mathrm{Res}^{(\mathfrak{g},\mathbf{K})}_{(\mathfrak{g},\mathbf{K}')}Y$ are isomorphic, irreducible $(\mathfrak{g},\mathbf{K}')$-modules.
    \item If $X \in \mathrm{Irr}(\mathfrak{g},\mathbf{K}')$ is fixed by $A$, then $\mathbf{I}^{(\mathfrak{g},\mathbf{K})}_{(\mathfrak{g},\mathbf{K}')}X$ is reducible, with two irreducible summands. If $X', Y' \in \mathrm{Irr}(\mathfrak{g},\mathbf{K}')$ are $A$-conjugate, then $\mathbf{I}^{(\mathfrak{g},\mathbf{K})}_{(\mathfrak{g},\mathbf{K}')}X'$ and $\mathbf{I}^{(\mathfrak{g},\mathbf{K})}_{(\mathfrak{g},\mathbf{K}')}Y'$ are isomorphic, irreducible $(\mathfrak{g},\mathbf{K})$-modules.
    \item In this fashion, induction and restriction define inverse bijections
    $$\mathrm{Irr}(\mathfrak{g},\mathbf{K}')/A \longleftrightarrow \mathrm{Irr}(\mathfrak{g},\mathbf{K})/A$$
    These bijections exchange one-element and two-element $A$-orbits. 
\end{enumerate}
\end{proposition}

Proposition \ref{prop:cliffordtheoryindex2} can be deduced from the corresponding result for classical induction and a description (\cite{KnappVogan1995}, Proposition 2.75) of $I^{(\mathfrak{g},\mathbf{K})}_{(\mathfrak{g},\mathbf{K}')}$. We leave the details to the reader.

\bibliographystyle{plain}
\bibliography{bibliography.bib}

\begin{thebibliography}{10}

\bibitem{ABV}
J.~Adams, D.~Barbasch, and D.~Vogan.
\newblock {\em The Langlands Classification and Irreducible Characters for Real
  Reductive Groups}.
\newblock Birkh\"auser, Boston-Basel-Berlin, 1992.

\bibitem{AdamsLeeuwenTrapaVogan2017}
J.~Adams, M.V. Leeuwen, P.~Trapa, and D.~Vogan.
\newblock Unitary representations of real reductive groups.
\newblock 2017.
\newblock Preprint.

\bibitem{AdamsVogan1992}
J.~Adams and D.~Vogan.
\newblock L-groups, projective representations, and the langlands
  classification.
\newblock {\em American Journal of Mathematics}, 114(1):45--138, 1992.

\bibitem{AdamsVogan2015}
J.~Adams and D.~Vogan.
\newblock Parameters for twisted representations.
\newblock {\em Progr. Math.}, 312:51--116, 2015.

\bibitem{AdamsVogan2019}
J.~Adams and D.~Vogan.
\newblock Associated varieties for real reductive groups.
\newblock {\em Pure and Applied Mathematics Quarterly}, 0(0), 2019.

\bibitem{Arthur1983}
J.~Arthur.
\newblock On some problems suggested by the trace formula.
\newblock In {\em Lie Group Representations II}, volume 1041, pages 1--49.
  Springer-Verlag, Berlin-Heidelberg-New York, 1983.

\bibitem{Arthur1989}
J.~Arthur.
\newblock Unipotent automorphic representations: conjectures 13--71.
\newblock In {\em Orbites Unipotentes et Repr\'esentations II. Groupes
  p-adiques et R\'eels, Ast\'erisque}, volume 171--172. 1989.

\bibitem{BarbaschSepanski1998}
D.~Barbasch and M.~Sepanski.
\newblock Closure ordering and the kostant-sekiguchi correspondence.
\newblock {\em Proc. Amer. Math. Soc.}, 126(1):311--317, 1998.

\bibitem{BarbaschVogan1985}
D.~Barbasch and D.~Vogan.
\newblock Unipotent representations of complex semisimple lie groups.
\newblock {\em Ann.\ of Math.}, 121:41--110, 1985.

\bibitem{Borel1979}
A.~Borel.
\newblock Automorphic l-functions.
\newblock 33, part 2:27--61, 1979.

\bibitem{BorhoBrylinski1985}
W.~Borho and J.L. Brylinski.
\newblock Differential operators on homogeneous spaces iii.
\newblock {\em Invent.\ Math.}, 80:1--68, 1985.

\bibitem{Casselman1978}
W.~Casselman.
\newblock Jacquet modules for real reductive groups.
\newblock In {\em Proceedings of the International Congress of Mathematicians},
  Helsinki, 1978.

\bibitem{Clifford1937}
A.H. Clifford.
\newblock Representations induced in an invariant subgroup.
\newblock {\em Ann. of Math. (2)}, 38(3):533--550, 1937.

\bibitem{HechtMilicicSchmidWolf}
H.~Hecht, D.~Mili\v{c}i\'{c}, W.~Schmid, and J.~Wolf.
\newblock Localization and standard modules for real semisimple lie groups i:
  The duality theorem.
\newblock {\em Invent.\ Math.}, 90:297--332, 1987.

\bibitem{Joseph1985}
A.~Joseph.
\newblock On the associated variety of a primitive ideal.
\newblock {\em J.~Algebra}, 93:509--523, 1985.

\bibitem{Knapp1996}
A.~Knapp.
\newblock {\em Lie Groups: Beyond an Introduction}.
\newblock Progress in Mathematics 140. Birkh\"auser, Boston-Basel-Berlin, 1996.

\bibitem{KnappVogan1995}
A.~Knapp and D.~Vogan.
\newblock {\em Cohomological induction and unitary representations}, volume~45
  of {\em Princeton Mathematical Series}.
\newblock Princeton University Press, Princeton, NJ, 1995.

\bibitem{Kostant1959}
B.~Kostant.
\newblock The principal three-dimensional subgroup and the betti numbers of a
  complex simple lie group.
\newblock {\em American Journal of Mathematics}, 81(4):973--1032, 1959.

\bibitem{kostant1969}
B.~Kostant.
\newblock On the existence and irreducibility of certain series of
  representations.
\newblock {\em Bull. Amer. Math. Soc.}, 75(4):627--642, 07 1969.

\bibitem{KostantRallis1971}
B.~Kostant and S.~Rallis.
\newblock Orbits and representations associated with symmetric spaces.
\newblock {\em Amer. J. Math.}, 93:753--809, 1971.

\bibitem{Langlands1989}
R.~P. Langlands.
\newblock On the classification of irreducible representations of real
  algebraic groups.
\newblock pages 101--170, 1989.

\bibitem{Schmid1977}
W.~Schmid.
\newblock Two character identities for semisimple lie groups.
\newblock In {\em Non-Commutative Harmonic Analysis}, Berlin, Heidelberg, 1977.
  Springer.

\bibitem{Sekiguchi1987}
J.~Sekiguchi.
\newblock Remarks on real nilpotent orbits of a symmetric pair.
\newblock {\em J. Math. Soc. Japan}, 39(1):127--138, 1987.

\bibitem{Tits1966}
J.~Tits.
\newblock Normalisateurs de tores. i. groupes de coxeter \'etendus.
\newblock {\em J. Algebra}, 4:96--116, 1966.

\bibitem{Trapa2001}
Peter~E. Trapa.
\newblock Annihilators and associated varieties of aq(lambda) modules for
  u(p,q).
\newblock {\em Compositio Mathematica}, 129(1):1–45, 2001.

\bibitem{Vergne1995}
M.~Vergne.
\newblock Instantons et correspondance de kostant-sekiguchi.
\newblock {\em C. R. Acad. Sci. Paris S\'er. I Math.}, 320(8):901--906, 1995.

\bibitem{Vogan1981}
D.~Vogan.
\newblock {\em Representations of real reductive Lie groups}, volume~15 of {\em
  Progress in Mathematics}.
\newblock Birkh\"auser, Boston, MA, 1981.

\bibitem{Vogan1991}
D.~Vogan.
\newblock Associated varieties and unipotent representations.
\newblock In {\em Harmonic Analysis on Reductive Groups}, volume 101 of {\em
  Progr. Math.}, pages 315--388. Birkh\"auser, Boston, MA, 1991.

\bibitem{VoganIC4}
Jr. Vogan, David~A.
\newblock Irreducible characters of semisimple lie groups. iv.
  character-multiplicity duality.
\newblock {\em Duke Math. J.}, 49(4):943--1073.

\bibitem{Wolf1974}
J.~Wolf.
\newblock Finiteness of orbit structure for real flag manifolds.
\newblock {\em Geometriae Dedicata}, 3, 1974.

\end{thebibliography}

\end{document}